\documentclass[11 pt, reqno]{amsart}

\usepackage{amsthm}
\usepackage{amsmath}
\usepackage{amssymb}
\usepackage{amsfonts}
\usepackage{latexsym}
\usepackage[hidelinks]{hyperref}
\usepackage{bbm}
\usepackage{bm}
\usepackage{comment}
\usepackage{enumitem}
\setlist[enumerate]{leftmargin=*}
\usepackage{mathtools} 

\usepackage{cleveref}
\usepackage[noadjust]{cite}
\usepackage{chngcntr}
\counterwithin*{equation}{section}
\usepackage{mathtools}
\usepackage{nicefrac}
\allowdisplaybreaks
\newcommand{\defeq}{\vcentcolon=}

\newcommand{\head}{\le}
\newcommand{\tail}{>}
\newcommand{\sequence}[1]{{#1 \in \mathbb{N}}}

\usepackage[T1]{fontenc}
\usepackage[letterpaper, hmarginratio=1:1]{geometry}
\usepackage{color}
\usepackage{latexsym}

\newtheorem{ccounter}{ccounter}[section]
\newtheorem{thm}[ccounter]{Theorem}

\newtheorem{lem}[ccounter]{Lemma}
\newtheorem{cor}[ccounter]{Corollary}

\newtheorem{prop}[ccounter]{Proposition}

\newtheorem{ex}[ccounter]{Example}

\theoremstyle{definition}
\newtheorem{defn}[ccounter]{Definition}
\newtheorem{rmk}[ccounter]{Remark}
\theoremstyle{plain}
\newtheorem{assume}{Assumption}

\def\be{\begin{equation}}
\def\ee{\end{equation}}
\def\leq{\leqslant}

\def\mathbf{\bm}
\def\={\coloneqq}
\def\hatu{\hat{\bm u}}
\def\hatw{\hat{\bm w}}
\def\N{\mathbb{N}}
\def\C{\mathcal{C}}
\def\mun{\mu^{(n)}}\def\tmun{\tilde\mu^{(n)}}
\newcommand\X{\mathcal{X}}
\def\bn{\bm b^{(n)}}
\def\P{\mathbb{P}}\def\Pn{\mathbb{P}^{(n)}}\def\tPn{\tilde{\mathbb P}^{(n)}}
\def\E{\mathbb{E}}\def\En{\mathbb{E}^{(n)}}\def\tEn{\tilde{\mathbb E}^{(n)}}
\def\S{\mathbb{S}}

\def\Wn{W^{(n)}}
\def\Ln{L^{(n)}}
\def\tr{\operatorname{Tr}}
\def\one{{\mathbbm 1}}

\def\bar{\overline}
\def\tilde{\widetilde}
\def\hat{\widehat}

\def\nun{\nu^{(n)}}

\def\R{{\mathbb {R}}}

\def\u{{\mathbf {u}}}\def\un{\bm u^{(n)}}

\def\eps{\varepsilon}
\def\epsilon{\varepsilon}
\newcommand{\V}{\mathbb{V}}

\renewcommand\a{{\bm{a}}}\def\an{\bm{a}^{(n)}}

\newcommand\w{{\mathbf {w}}}
\newcommand\x{{\mathbf {x}}}
\newcommand\y{{\mathbf {y}}}
\newcommand\g{{\mathbf {g}}}
\newcommand\trans{\mathsf{T}}

\renewcommand{\phi}{\varphi}
\newcommand{\wg}{\operatorname{Wg}}
\renewcommand{\phi}{\varphi}
\newcommand{\n}{\mathfrak{n}}
\newcommand{\m}{\mathfrak{m}}
\renewcommand\v{{\mathbf {v}}}
\newcommand\balpha{{\bm{u}}}
\newcommand\distequal{\,{\buildrel d \over =}\,}
\usepackage[normalem]{ulem}
\def\<#1{\left\langle{#1}\right\rangle}

\begin{document}
\title{Quenched Large Deviation Principles for Random Projections of $\ell_p^n$ balls}
\author{Patrick Lopatto}
\author{Kavita Ramanan}
\author{Xiaoyu Xie}

\begin{abstract}
Let  $(k_n)_{n \in \N}$ be a sequence of positive integers growing to infinity
    at a sublinear rate, $k_n \rightarrow \infty$ and $k_n/n \rightarrow 0$ as
    $n \rightarrow \infty$. 
    Given a  sequence of $n$-dimensional random vectors   $\{Y^{(n)}\}_{n \in \N}$  
    belonging to a certain class,  which includes 
    uniform distributions on suitably scaled  $\ell_p^n$-balls or $\ell_p^n$-spheres, $p \geq 2$, 
    and product distributions with sub-Gaussian
    marginals,  we study the large deviations behavior of the 
     corresponding  sequence of $k_n$-dimensional orthogonal 
    projections  $n^{-1/2} \a_{n,k_n} Y^{(n)}$, where $\a_{n,k_n}$ is an $(n \times k_n)$-dimensional
    projection matrix lying in the Stiefel manifold 
    of orthonormal
    $k_n$-frames in $\R^n$.
       For almost every sequence of projection matrices, we 
       establish a large deviation principle (LDP) for the corresponding sequence of projections,  
    with a fairly explicit rate function that does not depend on
    the sequence of projection matrices. As corollaries, we also obtain quenched LDPs for  sequences of $\ell_2$-norms and $\ell_\infty$-norms 
    of the coordinates of the projections.  
    Past work on LDPs for projections with growing dimension has mainly
    focused on the annealed setting, where one also averages over 
    the random projection matrix, chosen from 
    the Haar measure, 
    in which case the coordinates of the projection are exchangeable.  The quenched setting lacks such symmetry properties,
    and gives rise to  significant new challenges in the setting of growing projection dimension. 
       Along the way, we  establish new  Gaussian approximation results on
    the Stiefel manifold that may be of independent interest. 
    Such LDPs are of relevance in asymptotic convex geometry, statistical physics and high-dimensional statistics.
    \end{abstract}

\maketitle
{\setcounter{tocdepth}{1}
    \tableofcontents
}

\section{Introduction}

\subsection{Background and Motivation}

High-dimensional measures are ubiquitous in mathematics, and are often profitably studied through their lower-dimensional projections.
This approach has been successfully applied to problems in numerous fields, including statistics \cite{DiaFre84,FriTuk74,MaiMun09}, asymptotic functional analysis \cite{Johnson84} and convex geometry \cite{AntBalPer03,Klar07}.
In the case of the uniform measure on a high-dimensional convex body (a compact convex set with non-empty interior), low-dimensional projections are known to satisfy a central limit theorem.
This theorem states that most $k$-dimensional projections of an $n$-dimensional isotropic convex body are approximately Gaussian in total variation norm, if $n$ is sufficiently large and $k$ is sufficiently small relative to $n$ \cite{AntBalPer03,Klar07power}. The typical behavior of a low-dimensional projection is therefore  uninformative about the high-dimensional convex body from which it originated.  However, it was recently discovered that the tail behavior of random
orthogonal projections retains interesting information about the original measure \cite{GKR17,GKR16}. This tail behavior is quantified through large deviation principles (LDPs), of which there are two main types: \emph{quenched} LDPs, which 
provide almost-sure statements with respect to the projection or sequence of projections 
(which are independent of the high-dimensional measure), and \emph{annealed} LDPs, which consider an average over the measure from which the projection directions or projection matrices are sampled. In this article, we focus on quenched LDPs for random projections of random vectors  from a class of distributions that
includes the uniform distribution on  the unit ball in $\ell_p^n$, one of the most fundamental examples of a convex body. Denoting this ball by $\mathbb{B}_p^n$, we have
\begin{align*}
\mathbb{B}_p^n\coloneqq \left\{\mathbf{x} \in \mathbb{R}^n: \sum_{i=1}^n\left|x_i\right|^p \leq n\right\}.
\end{align*}
We consider projection directions  chosen uniformly from $\mathbb{S}^{n-1}$, the unit sphere in $\R^n$ (in the one-dimensional case), or projection matrices chosen
from the Haar measure on the Stiefel manifold $\V_{n,k}$ of orthonormal
$k$-frames in $\R^n$ (in the multi-dimensional case). 

We begin by reviewing previous work on LDPs for projections of $\mathbb{B}_p^n$, starting with the annealed case. Annealed LDPs are typically easier to analyze than quenched LDPs, since averaging over the randomness of the projection renders the entries of the projected vector exchangeable. For $\mathbb{B}_p^n$, annealed LDPs for one-dimensional projections were first established in \cite{GKR16, GKR17} for all $p\in[1,\infty]$.
Later, annealed LDPs for the $\ell_2$ norm of $k_n$-dimensional projections  were established in \cite{alonso2018large} for all $p \in [1,\infty]$ when $\lim_{n \rightarrow \infty} k_n /n = \lambda \in [0,1]$, with the additional requirement that $\lambda > 0$ when $p \le 2$. The restriction that $\lambda > 0$ was removed in \cite{KimLiaRam22}, which also proved a phase transition in the speed of the LDP with respect to the growth rate of $k_n$ (see \cite[Remark 3.6]{KimLiaRam22}).
Further, the article \cite{KimLiaRam22} established LDPs for the empirical measures of the coordinates
of the projections, in addition to the $\ell_2$ norm. Additionally, \cite{KimLiaRam22} went beyond the balls $\mathbb{B}_p^n$ and established results for general high-dimensional measures satisfying an asymptotic thin shell condition, including uniform measures on Orlicz balls and certain Gibbs measures.  A more refined version of this condition, and corresponding annealed sharp large deviation estimates for projection  were also subsequently obtained in \cite{LiaoThesis22}. 

Compared to the annealed case, much less is known about quenched LDPs. Previous works have focused exclusively on $k$-dimensional projections for $k$ independent of $n$.
For one-dimensional projections of the unit cube $\mathbb{B}_\infty^n$ and more general product measures, quenched LDPs were established in \cite{GKR16}.
For one-dimensional projections of $\mathbb{B}_p^n$,  quenched LDPs were established for all $p\in[1,\infty)$ in \cite{GKR17}.
In the case $p\in(1,\infty]$, the speed and rate function of the LDP are insensitive to  the choice of projection matrix, except for a measure zero set of so-called \emph{atypical} sequences of projection matrices. 
(In the case $p=1$, the LDP is more subtle and depends on the sequence; see \cite[Theorem 2.6]{GKR17}). 
Later, \emph{sharp} LDPs that identify the precise prefactor were obtained for one-dimensional projections of $\mathbb{B}_p^n$ in \cite{liao2020geometric}. 
Quenched LDPs for $k$-dimensional projections were established for $p\ge 2$ in \cite{kim2021large}, for any fixed positive integer $k$. Similar to the one-dimensional case, the speed and rate function are almost surely independent of the choice of sequence of projection matrices.
Quenched LDPs for projections of various radially symmetric measures on $\mathbb{B}_p^n$ were also recently studied in \cite{kaufmann2022large}. We also remark that LDPs for the $\ell_q$ norm of a random element of $\mathbb{B}_p^n$ (with $q\neq p$) were established in \cite{kabluchko2021high}.

In this work, we prove a quenched LDP for sequences of $k_n$-dimensional projections of
$\ell_p^n$ balls when $p\in [2, \infty)$ and $k_n$ grows sublinearly in $n$,
  that is, $\lim_{n\rightarrow \infty} k_n = \infty$ and $\lim_{n\rightarrow \infty} k_n/n = 0$ (see \Cref{t:dnp}).   
As in previous work, the speed and rate function are almost surely insensitive to the sequence of projection matrices.
As corollaries, we also obtain LDPs for the sequences of $\ell_2$ norms and $\ell_\infty$ norms of the projections
(see Corollaries \ref{c:convex} and \ref{c:largest}). 
The extension to LDPs for the $\ell_q$ norms of the projections for all $q\in [1,\infty)$ is also possible using our methods; see \Cref{r:lqnorms} below. Our results also generalize to projections of a larger class of random vectors, including those uniformly distributed on the $\ell_p^n$ sphere and a broad class of product measures with sub-Gaussian marginals   (see Theorem \ref{t:main}).
In addition to asymptotic convex geometry, LDP results of this
  kind are also relevant to problems arising in statistical mechanics. 
  Indeed, establishing the  LDP for $\ell_2$-norms of projections is  equivalent to identifying the scaled 
  logarithmic asymptotics for  expectations of exponential functionals of the sequence of
  $\ell_2$ norms of the projections.  The latter 
  has close parallels with the study of  the quenched log-partition function for statistical
  mechanical models with random disorder, such as the   Hopfield model of neural networks \cite{CD01}, 
   in which the projection
   matrix is replaced by an  $(n \times k_n)$-matrix of i.i.d.\   entries and $Y^{(n)}$ is sampled from a product distribution
   with bounded support. Similar 
 logarithmic asymptotics also appear in the study of the normalizing 
  constant of the posterior distribution in high-dimensional linear regression  \cite{MukSen22,JiaSen23},
  where the distribution of $Y^{(n)}$ corresponds to the prior, which is taken to be
  a product distribution in \cite{MukSen22} and uniform on $\mathbb{S}^{n-1}$ in \cite{JiaSen23}.

  \subsection{Proof Techniques. } 
There are two primary technical challenges involved in establishing our results.
First,  any LDP needs to be defined relative to a sequence of measures on a single probability space, but each element of a sequence of $k_n$-dimensional projections has a different codomain.  To remedy this problem, each projection must be embedded into a suitable parent space. In the annealed case,  the empirical measure of the coordinates completely determines the distribution of the projected vector, due to the exchangeability of the coordinates. Previous work on annealed LDPs for projections of growing dimension identified the projections with the empirical measures
of their coordinates, which are probability measures on $\mathbb{R}$, and established LDPs for the latter \cite{KimLiaRam22}.
However, in the quenched case, the lack of exchangeability of the coordinates renders the empirical measure an unsuitable state variable. 
Instead, we  take inspiration from the work of Comets and Dembo on large deviations for mean-field spin glass models \cite{CD01}, and embed our projections into a space consisting of infinite sequences whose coordinates appear in descending order (in absolute value),  paired with their $\ell_2$ norms.
 This approach, while leading to some technicalities, also captures a great deal of information about the sequence of projections, allowing us to also prove the additional LDPs for the $\ell_\infty$ and $\ell_2$ norms mentioned previously. We remark that the previous works \cite{KimLiaRam22,kim2021large} on multidimensional projections in the annealed setting did not address LDPs with the projection dimension $k_n$ growing to infinity for the $\ell_\infty$ norm, since the latter is not a sufficiently nice function of the empirical coordinate measure.

However, the paper \cite{CD01} considers the setting where  $Y^{(n)}$ is a product measure with a bounded distribution and
 the projection matrix is replaced with a matrix with i.i.d.  entries. 
In contrast,  no similar boundedness hypothesis can be made in our work, which creates several additional technical complications. 
Further, the second main technical obstacle in our setting is that, in the course of our proof, we must show that the entries of the 
projection matrix sampled uniformly from the Haar measure on $\mathbb{V}_{n,k_n}$ can be well approximated by independent Gaussians for averages of a function of the rows. 
This claim is made precise in \Cref{l:lln} below, and may be of independent interest, 
as it  generalizes several existing approximation results \cite{kim2021large,Jia05}. 
The result  \cite[Corollary 2.11]{kim2021large}  establishes 
convergence  of the empirical measure of rows of the projection matrix to a Gaussian 
in the space of Borel probability measures on $\mathbb R^k$ (equipped with  suitable Wasserstein topologies), 
only for fixed $k$.  However, such a statement does not address the case when  $k_n$ grows in $n$, in which case 
such empirical measures would live on different spaces. 
Instead, in \Cref{l:lln}  we establish  convergence of the empirical measures of scalar products of rows of the projection matrix with a
$k_n$-dimensional deterministic vector towards a $1$-dimensional Gaussian distribution in a suitable Wasserstein topology 
(see \Cref{r:explanatoryremark}) on the space of probability measures on $\R$, 
uniformly over any collection of up to  $e^{k_n}$ deterministic vectors.
Prior results on Gaussian approximations for entries of Haar-distributed random matrices from the Stiefel manifold, such
as those obtained in \cite{Jia03,Jia05}, appear inadequate for our purposes,  thus necessitating our alternative approach to such
Gaussian approximations using Weingarten calculus; we expand on this point in \Cref{r:Jiang}.

\subsection{Definitions}

For $k\in \N$, let $I_k$ be the $k \times k$ identity matrix, and for $n \ge k$, set
\begin{equation}\label{stiefelManifold}
\V_{n,k} \coloneqq \left\{ A \in \R^{n \times k} : A^{\trans} A= I_k \right\}.
\end{equation}
 The set $\V_{n,k}$ is called the Stiefel manifold of $k$-frames in $\R^n$ and consists of $k$-dimensional orthonormal bases in $\R^n$. We also use $O_n\defeq \V_{n,n}$ to denote the set of $n\times n$ orthogonal  matrices.

Let $\ell^2$ denote the set of infinite sequences $\x = (x_1, x_2, \dots)\in \R^\infty$ of real numbers such that the norm
\begin{equation}\label{l2}
\| \x \|_2 = \left( \sum_{i=1}^\infty x_i^2 \right)^{1/2}
\end{equation}
is finite. Given $\x, \y \in \ell^2$, we define the inner product
$\langle \x, \y \rangle \coloneqq \sum_{i=1}^\infty x_i y_i$.
We say that $\x \in \ell^2$ is ordered if $|x_k| \ge |x_{k+1}|$, and $x_k \ge x_{k+1}$ if $|x_k | = |x_{k+1}|$, for all $k\in \N$. Given a sequence $\x \in \ell^2$ with a finite number of nonzero entries, let
\[{\mathfrak n}(\x) \coloneqq \big|
\{ i\in \N : x_i \neq 0 \}\big|\]
denote the number of such entries. Let $[\x] \in \ell^2$ denote the sequence whose first ${\mathfrak n}(\x)$ coordinates $[x]_1, \dots, [x]_{{\mathfrak n}(\x)}$ are equal to the nonzero entries of $\x$ arranged so that $[\x]$ is ordered; this implies that $[x]_i = 0$ for $i\ge {\mathfrak n}(\x)+ 1$. Given $k \in \N$ and $\x \in \R^k$,
we define $[\x]$ by first defining  $\tilde \x \in \ell^2$ as follows: $\tilde \x_i = \x_i$ for $i \le k$ and $\x_i = 0$ for $i > k$, then setting $[\x] \coloneqq [\tilde \x]$. 
We define $\X$ as the set of pairs
\begin{equation}\label{Xdefinition}
\X \coloneqq \{ (\x, r) : \x\in \ell^2,\, r\in [0,\infty) ,\, \text{$\x$ is ordered},\, \| \x \|_2 \le r  \}.
\end{equation}
We equip $\X$ with the distance
\begin{equation}\label{linfinitydistance}
d\big((\x,r),(\mathbf y,s)\big) \coloneqq \| \x - \mathbf y \|_{\infty} + |r -s|.
\end{equation}
For  $\x \in \R^k$, let
\begin{equation}\label{def-Pi}
\Pi(\x) \coloneqq  \left( \left[  \x \right], \left\| \x \right\|_2 \right) \in \X
\end{equation}
denote the corresponding representative in $\X$.

The function $\log(x)$ denotes the natural logarithm. We let $\R_+$ denote the non-negative reals. We define the unit $\ell_n^p$ ball as the set
\begin{equation}
\mathbb B_p^n \coloneqq  \left\{ \x \in \R^n : \sum_{i=1}^n |x_i|^p \le n \right\}
\end{equation}
and the unit $\ell_n^p$ sphere by
\begin{equation}
\mathbb S_p^n \coloneqq  \left\{ \x \in \R^n : \sum_{i=1}^n |x_i|^p = n \right\}.
\end{equation}

Given a sequence  $\{k_n\}_\sequence{n}$  of positive integers, let $\sigma_{n}$ be the Haar measure on $\V_{n,k_n}$, and let $\sigma$ be a probability measure on $\V = \otimes_n \V_{n, k_n}$ whose $n$-th marginal is $\sigma_n$ for all $n\in \N$. All statements about probabilities on $\V_{n,k_n}$ and $\V$ in this work are made with respect to the measures $\sigma_n$ and $\sigma$.
For brevity, we do not explicitly denote  the dependence of $\V$, $\sigma_n$, and $\sigma$ on the sequence $\{k_n\}$.
We say that the sequence $\{k_n\}_\sequence{n}$ is increasing if $k_{n+1} \ge k_n$ for all $n \in \N$.

Finally, we recall the definition of an LDP.
\begin{defn}[{\cite[Section 1.2]{DZ93}}]\label{d:ratefunction}
Let $\mathcal T$ be a topological space with Borel $\sigma$-algebra $\mathcal B$.
A sequence $\{\P_n\}_\sequence{n}$ of probability measures on $\mathcal T$ satisfies a large deviation principle (LDP)  with speed $s_n: \N \rightarrow \R_+$ and rate function $I \colon \mathcal T \rightarrow [0, \infty]$ if for all $B \in \mathcal B$,
\be
- \inf_{x \in B^\circ} I(x)
\le \liminf_{n\rightarrow\infty} \frac{1}{s_n} \log \P_n(B)
\le \limsup_{n\rightarrow \infty} \frac{1}{s_n} \log \P_n(B)
\le - \inf_{x\in \overline{B}} I(x),
\ee
where $B^{\circ}$ and $\overline{B}$ denote the interior and closure of $B$, respectively. We say that $\{\P_n\}_\sequence{n}$ satisfies a weak LDP when these inequalities hold for all $B \in \mathcal B$ such that $\overline{B}$ is compact.

A sequence of $\mathcal T$-valued random variables $\{Z^{(n)}\}_\sequence{n}$ is said to satisfy an LDP with speed $s_n$ and rate function $I$ if and only if the corresponding sequence of image measures $\{ \P( Z^{(n)} \in \cdot )\}_\sequence{n}$ satisfies an LDP with that speed and rate function. When the speed is not mentioned explicitly, we use the default $s_n =n$. The function $I$ is said to be a good rate function if it has compact level sets $\Psi_I(\alpha) = \{ x \in \mathcal T : I(x) \le \alpha \}$ for all $\alpha \in [0,\infty)$.
\end{defn}

\subsection{Main Result: A Specific Setting}
We begin by introducing some general notation, which will be used throughout the paper.
Let $Z$ be a mean zero random variable

and let $\eta\colon \R \rightarrow \R_+$ be a continuous function.
We define
$\overline{\Lambda} \colon \R^2 \rightarrow \R \cup \{\infty\}$ to be the log-moment generating function of $(Z, \eta(Z))$: 
\be\label{asprev}
\overline{\Lambda} (t_1, t_2) =
\overline{\Lambda}_{Z,\eta} (t_1, t_2) \coloneqq \log \E\left[\exp\big(t_1 Z  + t_2 \eta( Z) \big)  \right].
\ee

Further, let $\{g_i\}_{i=0}^\infty$ be an infinite sequence of independent Gaussian random variables with mean zero and variance one, denote $\g=(g_1, g_2, \dots)$,
and define $\Lambda : \ell^2 \times \R^2 \rightarrow \R \cup \{\infty\}$
as\footnote{
The function $\Lambda$ represents a suitable average of the log-moment generating function $\overline{\Lambda}$ over the random ``environment'' (i.e. projection matrix), with the Gaussian vector $(g_0,\g)$ arising as an approximation to the typical row of a Haar-distributed element of $\V_{n,k}$ in the sublinear regime. See \Cref{l:lln}.
}
\begin{equation}
\label{d:g0}
\Lambda(\balpha,b,c )
\coloneqq \E \big[ \overline{\Lambda} \big( \langle \balpha, \g \rangle  + b g_0, c \big) \big].
\end{equation}
The Legendre transform $\Lambda^*$ of $\Lambda$ is defined by
\begin{equation}\label{lambdastar}
\Lambda^*(\bm w,r,s ) \coloneqq
\sup_{(\balpha,b,c) \in \ell^2 \times \R \times \R}
\left\{
\langle \balpha,  \bm w \rangle + br + cs  - \Lambda(\balpha,b,c )
\right\}.
\end{equation}
We  also set
\begin{equation}\label{ratefunction}
I (\bm w,r,s) \coloneqq \Lambda^*\left(\bm w, \sqrt{r - \|\bm w\|_2^2}, s\right)
\end{equation}
for $(\bm w,r,s) \in \mathcal X \times \R_+$. Let  $\rho\colon \R_+ \rightarrow \R_+$ be a continuous function. For $( \w , r) \in \mathcal X$, we define
\be\label{calI}
\mathcal I(\mathbf w,r)=
\inf_{s>0}I\left(\frac{\mathbf w}{\rho(s)},\frac{r}{\rho(s)},s\right) 
.
\ee

We now specialize the previous definitions for the purpose of stating our main results, concerning $\mathbb{B}_p^n$.
We fix $p \in [1, \infty)$, and let $f_p : \R \rightarrow \R$ be
the probability density function of the $p$-generalized normal distribution:
\be\label{pgeneral}
f_p(x)  \coloneqq \frac{1}{2 p^{1/p} \Gamma(1 + 1/p) } \exp\left( - \frac{|x|^p}{p}   \right).
\ee
Let $Z_p$ be a random variable with density $f_p$, let $ \eta_2(x) = x^2$, and define
$\overline{\Lambda}_p \coloneqq \overline{\Lambda}_{Z_p, \eta_2}$, as in \eqref{asprev}.
We let $\Lambda_p$ denote the functional in \eqref{d:g0} when $\overline{\Lambda}$ is replaced by  $\overline{\Lambda}_p$.
Also, the let definitions \eqref{lambdastar}, \eqref{ratefunction}, and \eqref{calI} hold when $\Lambda^*$, $I$, and $\mathcal I$ are replaced with  $\Lambda^*_p$, $I_p$, and $\mathcal I_p$, respectively, and $\rho(x) = x^{1/p}$. We also set
\be
\hat{\mathcal I}_p(\mathbf w,r)\coloneqq\inf_{s>0}I_p\left(\mathbf w,r,s\right).
\ee

The following theorem is our primary result. It establishes a quenched LDP for multidimensional projections of $\ell^n_p$ balls when the projection dimension grows sublinearly. It is proved in \Cref{s:mainresult}. We emphasize that in this theorem and its corollaries, $\a\in \V$ always denotes a fixed, deterministic sequence of projection matrices.
\begin{thm}\label{t:dnp}
Fix $p \in [2, \infty)$, and let $\{k_n\}_\sequence{n}$ be an increasing sequence of positive integers such that
$\lim_{n\rightarrow \infty} k_n = \infty$ and
$\lim_{n\rightarrow \infty} k_n/n =0$.
For each $n \in \N$, let $Y^{(n)}$ be uniformly distributed on $\mathbb B_p^n$.
Then for $\sigma$-a.e. $\a = \{\a^{(n)}\}_{n \in \N}\in \V$, the sequence
\begin{equation}\label{ballconclude}
\left\{\Pi\left(n^{-1/2} (\a^{(n)})^\trans Y^{(n)}\right)\right\}_\sequence{n},
\end{equation}
with $\Pi$ defined as in \eqref{def-Pi},
satisfies an LDP in $\mathcal X$ with speed $n$ and good rate function $\mathcal I_p$. 
\end{thm}
From \Cref{t:dnp}, we can deduce LDPs for the sequence of Euclidean norms and maximal coordinates of the random projection $n^{-1/2} (\a^{(n)})^\trans Y^{(n)}$. We state these in the following corollaries, which are also proved in \Cref{s:mainresult}. 
\begin{cor}\label{c:convex}
Let $p$,  $\{k_n\}_\sequence{n}$, and $Y^{(n)}$ be as in \Cref{t:dnp}.
Then for $\sigma$-a.e. $\mathbf a=\left\{\mathbf a^{(n)}\right\}_{n\in\mathbb N}\in\mathbb V$, the sequence
\begin{align}\label{oldLDP}
\begin{split}
    \left\{\left\|n^{-1/2}\left(\mathbf a^{(n)}\right)^\trans Y^{(n)}\right\|_2\right\}_{n\in \N}
\end{split}
\end{align}
satisfies an LDP $\mathbb R_+$ with speed $n$ and good rate function $\mathbb I_p$ given by
\begin{align}\label{eqn:rate2norm}
\begin{split}
    \mathbb I_p\left(r\right)\coloneqq \sup_{t_1,t_2\in\mathbb R}\left\{t_1r+t_2-\mathbb E\left[\overline{\Lambda}_p\left(t_1g_0,t_2\right)\right]\right\},
\end{split}
\end{align}
where $\overline{\Lambda}_p$ is defined in \eqref{asprev}. 
Furthermore, $\mathbb I_p$ is convex.
\end{cor}

\begin{cor}\label{c:largest}
Let $p$,  $\{k_n\}_\sequence{n}$, and $Y^{(n)}$ be as in \Cref{t:dnp}.
Then for $\sigma$-a.e. $\a = \{\a^{(n)}\}_{n \in \N}\in \V$, the sequence
\begin{equation}\label{maxldp}
\left\{ \left\| n^{-1/2} (\a^{(n)})^\trans Y^{(n)}\right\|_\infty \right\}_\sequence{n}
\end{equation}
satisfies an LDP in $\R$ with speed $n$ and good rate function
\be
\mathcal I_{\mathrm{max}}(r) \coloneqq  \mathcal I_p ((r,0,0\dots),r) .
\ee
\end{cor}
\begin{rmk}
The conclusions of \Cref{t:dnp}, \Cref{c:convex}, and \Cref{c:largest} also hold when $Y^{(n)}$ is uniformly distributed on $\mathbb S_p^n$.
The proofs of these results also contain, as intermediate steps, proofs of the analogous claims for $\mathbb S_p^n$.
\end{rmk}

\subsection*{Outline}
\Cref{t:dnp} is a consequence of a more general result, which provides LDPs for the projections of a large class of sequences 
of random vectors. 
This general result in stated in \Cref{s:maintech} below as \Cref{t:main}, along with the necessary assumptions and associated corollaries. The proof of \Cref{t:main} follows on
combining an upper bound and a lower bound, given in the next section as \Cref{l:ldpupper} and \Cref{l:ldplower}, respectively.
In \Cref{s:upper}, the upper bound \Cref{l:ldpupper} is proved assuming certain preliminary lemmas compiled in \Cref{s:upperpreliminary} and proved in \Cref{s:preliminaryproof}. 
The complementary lower bound is established in \Cref{s:lower}, building on preparatory results stated in \Cref{s:lowerpreliminary}, whose proofs are deferred to \Cref{s:preliminarylowerProof}.
\Cref{s:appendix} contains auxiliary computations for $p$-Gaussian random variables. \Cref{s:weingarten} recalls the Weingarten calculus and proves an auxiliary lemma.

\subsection{Open Problems. }

This work gives rise to several  open problems.  

\begin{enumerate}
\item  {\it The case $p \in (1,2)$.} It would be of interest to study the quenched LDP for the projection or its norm when 
  $Y^{(n)}$ is uniformly distributed on $\mathbb{B}_p^n$ with $p \in (1,2)$.
 The sub-Gaussianity assumption  $p \geq 2$ is used in several places in the proof, including in the proof of the 
  exponential tightness result in \Cref{l:tightness} and the concentration result in Lemma \ref{lemma:sub-Gaussian-norm-after-translation}.
  We believe that at least a  weak (quenched) LDP at speed $n$ will hold for sufficiently slowly growing $k_n$.
  But we expect a phase transition in the growth rate of $k_n$, wherein the speed of the LDP would change
  in a $p$-dependent way.
\item {\it The linear setting  $k_n = \lambda n$ for $\lambda \in (0,1)$.}  
  It would  be of interest to study the case when $k_n$ grows linearly in $n$,
  whis would likely require a 
  different approach  since the condition $k_n = o(n)$ is currently being used in several
  places in the proof, including the Gaussian approximation lemma, as well as to show that 
  the tail of the ordered projection vector has negligible $\ell_2$-norm, and associated concentration results. 
  \item {\it A broad class of high-dimensional vectors.} 
    It would be desirable to identify a broad sequence of random vectors $\{Y^{(n)}\}_{n \in \N}$ for which the corresponding
    (quenched) random projections
  satisfy an  LDP. 
 For example, does there exist a sufficient condition for LDPs of quenched projections, 
  analogous to the asymptotic thin shell condition in the annealed setting? 
\end{enumerate}
  
  The fact that all the above questions are fully understood in the annealed setting \cite{GKR17,alonso2018large,KimLiaRam22}
further points to the additional subtleties present in the quenched setting when compared to the annealed one.

\subsection*{Acknowledgments}
P.L.\ is  supported by NSF postdoctoral fellowship DMS-220289.
K.R.\ is supported by NSF grant DMS-1954351 and Vannevar Bush Faculty Fellowship N0014-21-1-2887.
X.X.\ is supported by NSF grant DMS-1954351.
The second author would like to thank Amir Dembo for bringing to her attention
the reference \cite{CD01}, which suggested the convenient topological space in
which the LDP in this paper is established.

\section{Main Technical Results}\label{s:maintech}

In this section, we list the general versions of our main technical results and the assumptions they require.
We first recall the following concept from convex analysis.
\begin{defn}[Essential Smoothness]
\label{d:esmooth}
Let $f:\R^d\rightarrow (-\infty,\infty]$ be a convex function and let $\mathcal D_f\coloneqq\{x\in\mathbb R^d:f(x)<\infty\}.$ Then $f$ is said to be essentially smooth if
\begin{enumerate}
    \item $\mathcal{D}_{f}^\circ$ is non-empty.
    \item $f(\cdot)$ is differentiable throughout $\mathcal{D}_{f}^\circ$.
    \item $f(\cdot)$ satisfies $\lim _{n \rightarrow \infty}\left|\nabla f\left(\lambda_{n}\right)\right|=\infty$ whenever $\left\{\lambda_{n}\right\}_\sequence{n}$ is a sequence in $\mathcal{D}_{f}^\circ$ converging to a boundary point of $\mathcal{D}_{f}^\circ$.
\end{enumerate}
\end{defn}
Given a sequence $\{ Y^{(n)} \}_\sequence{n}$ of random vectors $Y^{(n)} = ( Y_1, Y_2, \dots, Y_n) \in \R^n$, we consider the following assumptions.

\begin{assume}\normalfont
\label{assume1}
There exists a sequence of non-constant and independent, identically distributed (i.i.d.) real-valued random variables $\{X_j\}_\sequence{j}$, a non-constant continuous function $\eta \colon \R  \rightarrow \R_+$, and a continuous function $\rho\colon \R_+ \rightarrow \R_+$ such that we have the distributional equality
\begin{equation}\label{ourform}
Y^{(n)}_j \,{\buildrel d \over =}\, X_j \cdot \rho
\bigg(
\frac{1}{n} \sum_{i=1}^n \eta(X_i)
\bigg)
\end{equation}
for all $n \in \N$ and $j \in \{1,\dots , n\}$. We further suppose that $\rho(x) > 0$ for all $x>0$.
\end{assume}
\begin{assume}\normalfont
\label{assume2} The variable $X_1$ has mean zero
and 
 there exists a constant $C_2>  0$ such that for all $s\ge 0$
\be
\P\left( |X_1| \ge s  \right) \le 2 \exp( - s^2/C_2^2).
\ee
\end{assume}

\begin{assume}\normalfont
\label{assume3} Set $\overline{\Lambda}(t_1, t_2) = \overline{\Lambda}_{X_1, \eta}$, 
as defined in \eqref{asprev} and using the choice of $\eta$ from \Cref{assume1}.
There exists some $ 0 < T \le \infty$ such that $\overline{\Lambda}(t_1, t_2)$ is finite for all $(t_1, t_2 ) \in \R \times ( - \infty, T)$ and $\overline{\Lambda}$ is essentially smooth. Further, the derivatives $\partial_1^\alpha\partial_2^\beta\overline{\Lambda}(t_1,t_2)$ exist for all $(t_1, t_2 ) \in \R \times ( - \infty, T)$ and all integers $\alpha, \beta \ge 0$ such that $\alpha + \beta \le 2$.
\end{assume}
\begin{assume}\normalfont
\label{assume4} There exists a continuous function
$\tilde C\colon (-\infty, T) \rightarrow \R_+$ such that
\begin{equation}\label{polygrowth}
\left|\partial_1^\alpha\partial_2^\beta\overline{\Lambda}\left(t_1,t_2\right)\right|\leq \tilde C(t_2)\left(1+|t_1|^{2-\alpha}\right),
\end{equation}
for all $(t_1, t_2 ) \in \R \times ( - \infty, T)$ and all integers $\alpha, \beta \ge 0$ such that $\alpha + \beta \le 2$.
\end{assume}

\begin{rmk}
In addition to the uniform measure on the $\ell^p$ balls considered in \Cref{t:dnp}, 
the framework above allows us to deal with product measures whose marginals are symmetric and sub-Gaussian.
Indeed, by taking $\rho(x) = 1$ for all $x \in \R_+$, we see that these assumptions are satisfied when $Y^{(n)} = (Y_1, \dots, Y_n)$ is constructed from a sequence of i.i.d.\ random variables $\{Y_j\}_\sequence{j}$ satisfying Assumptions 2, 3, and 4. 
\end{rmk}

Additionally, we often consider increasing sequences $\{k_n\}_\sequence{n}$ of positive integers such that 
\be\label{knhypo}
\lim_{n\rightarrow \infty} k_n = \infty \qquad \text{and}\qquad
 \lim_{n\rightarrow \infty} \frac{k_n}{n} =0.
\ee

The following large deviation upper bound is proved in \Cref{s:proofofupperbound}.
\begin{prop}\label{l:ldpupper}
Let $\{k_n\}_\sequence{n}$ be an increasing sequence of positive integers satisfying \eqref{knhypo}. Suppose that $\{X_j\}_\sequence{j}$ is a sequence of i.i.d.\ random variables satisfying Assumptions 2, 3, and 4.
Let $X^{(n)} = (X_1, \dots, X_n) \in \R^n$ for all $n\ge 1$.
Then for $\sigma$-a.e. $\a = \{\a^{(n)}\}_{n \in \N}\in \V$, we have
\begin{equation}\label{eqn:ldp-upper-bound}
\limsup_{n \rightarrow \infty} \frac{1}{n}
\P\left( \left(\Pi(n^{-1/2} \left(\a^{(n)}\right)^\trans X^{(n)}),
\frac{1}{n} \sum_{i=1}^n \eta(X_i)
 \right) \in \mathcal S \right) \le - \inf_{(\w,r,s) \in \mathcal S} I(\w,r,s)
\end{equation}
for all closed sets $\mathcal S\subset \mathcal X \times \R_+$. 
\end{prop}
The following large deviations lower bound is proved in \Cref{s:lowerproof}.
\begin{prop}\label{l:ldplower}
Let $\{k_n\}_\sequence{n}$ be an increasing sequence of positive integers satisfying \eqref{knhypo}. Suppose that $\{X_j\}_\sequence{j}$ is a sequence of i.i.d.\ random variables satisfying Assumptions 2, 3, and 4.
Let $X^{(n)} = (X_1, \dots, X_n) \in \R^n$ for all $n\ge 1$.
Then for $\sigma$-a.e. $\a = \{\a^{(n)}\}_{n \in \N}\in \V$, we have
\begin{equation}\label{eqn:ldp-lower-bound}
\limsup_{n \rightarrow \infty} \frac{1}{n}
\P\left( \left(\Pi(n^{-1/2} \left(\a^{(n)}\right)^\trans X^{(n)}),\frac{1}{n} \sum_{i=1}^n \eta(X_i) \right) \in \mathcal O \right) \ge - \inf_{(\w,r,s) \in \mathcal O} I(\w,r,s)
\end{equation}
for all open sets $\mathcal O\subset \mathcal X \times \R_+$.
\end{prop}
The following theorem is proved in \Cref{s:mainresult} using the previous two lemmas, where it is then used to prove the main results stated in the previous section.
\begin{thm}\label{t:main}
Let $\{k_n\}_\sequence{n}$ be an increasing sequence of positive integers satifying \eqref{knhypo}.
Let $\{ Y^{(n)} \}_\sequence{n}$ be a sequence of random vectors $Y^{(n)} = ( Y_1, Y_2, \dots, Y_n)$ satisfying Assumptions 1--4.
Then for $\sigma$-a.e. $\a = \{\a^{(n)}\}_{n \in \N}\in \V$, the sequence
\begin{equation}\label{mainconclude}
\left\{\Pi\left(n^{-1/2} (\a^{(n)})^\trans Y^{(n)}\right)\right\}_\sequence{n}
\end{equation}
satisfies an LDP in $\mathcal X$ with good rate function $\mathcal I$.
\end{thm}

\section{Upper Bound}\label{s:upper}

\subsection{Preliminary Lemmas}\label{s:upperpreliminary}

In this section, we state some results required for the proof of \Cref{l:ldplower}.
The first lemma, which collects several topological properties of $\mathcal X$, is proved in
\Cref{s:61}.
\begin{lem}\label{l:topologicalpropertyofX}
\begin{enumerate}
\item The topology on $\mathcal X$ is equivalent to the product topology on $\mathbb R^{\mathbb N}\times\mathbb R_+$, where $\mathbb R^{\mathbb N}$ is itself equipped with the product topology.
\item The topology on $\mathcal X$ is equivalent to the product topology on $\mathbb R^{\mathbb N}\times\mathbb R_+$, where $\mathbb R^{\mathbb N}$ is itself equipped with the weak-$\ell^2$ topology.
\item For any fixed $A<\infty$, the set $\left\{(\mathbf w,r)\in\mathcal X:r\leq A\right\}$ is compact.
\end{enumerate}
\end{lem}
Our second lemma shows that the rate function in \Cref{l:ldpupper} is a good rate function (as defined in \Cref{d:ratefunction}). Its proof is deferred to \Cref{s:62}.
\begin{lem}\label{l:ratefunction}
The function $I$ defined in \eqref{ratefunction} is a good rate function.
\end{lem}
To state the following lemmas, we recall the definition of exponential tightness.
\begin{defn}
A family of measures $\{\mu_n\}_\sequence{n}$ on a topological space space $\mathcal T$ is exponentially tight with speed $s_n : \N \rightarrow \R_+$ if for every $\alpha < \infty$, there exists a compact set $K_\alpha \subset \mathcal T$ such that
\be
\limsup_{s_n \rightarrow \infty} \frac{1}{n} \log \mu_\eps ( K^c_\alpha ) \le - \alpha.
\ee
We say that a sequences of random variables $\{X_n\}_\sequence{n}$ is exponentially tight with speed $s_n$ if the sequence of measures  $\{\mu_n\}_\sequence{n}$ defined by $\mu_n(A) = \P( X_n  \in A)$ also has this property.
We default to $s_n = n$ when no speed is explicitly stated.
\end{defn}

The next two lemmas address exponential tightness of the sequence of the norms of the projections $n^{-1/2} (\a^{(n)})^\trans X^{(n)}$ and the sums $\frac{1}{n} \sum_{i=1}^n \eta(X_i)$. They are proved in \Cref{s:63}.
\begin{lem}\label{l:tightness}
Let $\{X_j \}_\sequence{j}$ be random variables satisfying \Cref{assume2}, and let $X^{(n)}=(X_1,\dots, X_n)$.
Let $\{k_n \}_\sequence{n}$ be an increasing sequence of positive integers satisfying \eqref{knhypo}.
Then there exists a constant $\gamma>0$, depending only on the sequence $\{k_n \}_\sequence{n}$ and the constant $C_2$ from \Cref{assume2}, such that for every $n\in \N$,  $\a^{(n)} \in \V_{n,k_n}$, and $t\ge C_2^2$,
\begin{equation*}
\frac{1}{n} \log \P\left( \|n^{-1/2} (\a^{(n)} )^\trans X^{(n)}\|_2^2 \ge t  + 1  \right) \le
 - \gamma t.
\end{equation*}

\end{lem}
\begin{lem}\label{l:tightness2}
Let $\left\{X_j\right\}_{j=1}^{\infty}$ be i.i.d.\ random variables satisfying Assumptions 2 and 3. Then
the sequence of random variables $\{Z_n\}_\sequence{n}$ given by $Z_n \coloneqq \frac{1}{n} \sum_{i=1}^n \eta(X_i)$ is exponentially tight.
\end{lem}

For $\u^{(n)} \in \R^{k_n}$, $\a^{(n)} \in \V_{n,k_n}$, and $c\in \R$, we let $\a_i^{(n)}$ denote the $i$-th row of $\a^{(n)}$, and define
\begin{equation}\label{eqn:quenched-log-moment}
F_n(\u^{(n)},\a^{(n)},c) \coloneqq
\frac{1}{n} \sum_{i=1}^n
\overline{\Lambda}\left(\langle \u^{(n)} ,\sqrt{n}\a_i^{(n)}\rangle, c \right).
\end{equation}

Recall that $g$ denotes a Gaussian random variable with zero mean and unit variance. The next lemma shows that $F_n(\u^{(n)},\a^{(n)},c)$ is closely approximated by $\E\left[\bar\Lambda\left(\left\|\un\right\|_2g,c\right)\right]$;
this holds uniformly for any subexponential collection of vectors in $\R^{k_n}$. 
 It is proved in \Cref{s:65}.

\begin{lem}\label{l:lln}
Fix a constant $D \in( 0,\infty)$, a deterministic sequence $\{d_n\}_\sequence{n}$  of positive integers such that  $\limsup_{n\rightarrow \infty} n^{-1} \log(d_n)  = 0$, and i.i.d.\ random variables $\left\{X_j\right\}_{j=1}^{\infty}$  satisfying Assumptions 2--4.
For every $n\in \N$, let
\begin{equation}
\mathcal W_n = \{ \v^{(n, 1)}, \v^{(n, 2)}, \dots, \v^{(n , d_n)} \}
\subset \R^{k_n}
\end{equation}
be a finite collection of vectors 
such that $\|\v^{(n,j)}\|_2 = D$ for all $n \in \N $ and $j \in \{1, \dots, d_n\}$.
Fix $c \in (-\infty, T)$, where $T$ is the constant from \Cref{assume3}. Then for $\sigma$-a.e. $\a = \{\a^{(n)}\}_{n \in \N}\in \V$,
\begin{equation*}
 \lim_{n\rightarrow\infty}\sup_{1\le j\leq d_n}\left|F_n(\v^{(n, j)}, \a^{(n)},c ) - \E\left[\bar\Lambda\left(Dg,c\right)\right] \right|=0,
\end{equation*}
where $g$ is a Gaussian random variable with zero mean and unit variance.
\end{lem}
\begin{rmk}\label{r:explanatoryremark}
From the proof of \Cref{l:lln}, it is clear that the result continues to hold if in the definition of $F_n$, $\bar\Lambda(\cdot, c)$ is replaced by any differentiable function $H \colon \R \rightarrow \R$ that satisfies, for some constant $C \in \R_+$, $|H(t)|^2 \le C(1 + t^2)$ and $| H' (t)|^2 \le C ( 1 + |t|)$.
\end{rmk}

\begin{rmk}
  \label{r:Jiang}
  As mentioned in the introduction, Jiang has proved a Gaussian approximation result for Haar-distributed elements of the Stiefel manifold
  in \cite[Theorem 5]{Jia05} by coupling elements of the Haar-distributed matrix to matrices of independent Gaussians through Gram-Schmidt orthogonalization. However, as noted in \cite[Theorem 3]{Jia03}, this coupling is  effective  only when $k_n = o(\log n /n)$.  
  Further, in order for this coupling to be strong enough to prove the almost-sure convergence in  \Cref{l:lln} (rather than the convergence in probability considered in \cite{Jia03}), 
  it appears that additional restrictions on $k_n$ and $d_n$  are required.  Hence, \cite[Theorem 5]{Jia05}
  is not sufficient to study the entire sublinear regime, which motivates the alternative approach to Gaussian approximation
  via the Weingarten calculus taken in our  proof of \Cref{l:lln} in \Cref{s:65}.  
\end{rmk}

For $\balpha=(u_1,u_2,\ldots) \in \ell^2$ and a given $m \in \N$, we let $\balpha_{\head} = \balpha_{\head m }\in\R^m$ be the vector containing the first $m$ coordinates of $\balpha$,
\begin{equation}\label{mtail}
\balpha_{\head}\coloneqq (u_1, u_2, \dots, u_m,0,0,\dots)^\trans,
\end{equation}
and let $\balpha_\tail= \balpha_{\tail m}\in\ell^2$ denote the vector defined by $\balpha_\tail = \balpha - \balpha_\head$, so that
\begin{equation}
\balpha_\tail\coloneqq ( 0, \dots, 0, u_{m+1}, u_{m+2}, \dots)^\trans.
\end{equation}
For $\hatu\in\mathbb R^m$ and $b,c \in \R$, and $\Lambda$ defined as in \eqref{d:g0}, set
\begin{equation}\label{lambdam}
\Lambda_m(\hatu,b,c)  \coloneqq \Lambda((\hatu^\trans,0,0,\ldots), b,c ).
\end{equation}
For all $\balpha \in \ell^2$ and $b,c \in \R$, we observe by Jensen's inequality that\footnote{Note that for any $t_2 \ge 0$, the function $t_1 \mapsto \overline{\Lambda}(t_1, t_2)$ is convex by H\"{o}lder inequality.
}
\begin{equation}
\Lambda(\balpha,b,c ) \ge \Lambda_m(\balpha_\head,b,c ).
\end{equation}
Additionally, using the dominated convergence theorem and \eqref{polygrowth}, we have the monotonic limit
\begin{equation}\label{monotonic}
\lim_{m\rightarrow \infty}
\Lambda_m(\balpha_{\head m}, b,c) = \Lambda(\balpha,b,c).
\end{equation}

\subsection{Proof of the Upper Bound}\label{s:proofofupperbound}
We first reduce the proof of the upper bound to
a verification of the upper bound for a simpler subclass of closed sets.

\begin{lem}\label{l:from-compact-to-closed}
Retain the notation and hypotheses of \Cref{l:ldpupper}. 
Suppose that for every closed set $\mathcal S \subset \mathcal X \times \R_+$ and constant $A \in ( 0, \infty)$, the upper bound \eqref{eqn:ldp-upper-bound} holds for the set $\mathcal S \cap {\mathcal K}_A$, where
\begin{equation}\label{d:KA}
    {\mathcal K}_A = \{(\mathbf w,r,s): 0\leq r\leq A,0\leq s\leq A \},
\end{equation}
for a set $\Omega_{\mathcal S, A}$ of sequences $\a = \left(\mathbf a^{(1)},\mathbf a^{(2)},\ldots\right)\in\mathbb V$
such that $\sigma(\Omega_{\mathcal S, A}) = 1$.
Then for almost every sequence $\a\in\mathbb V$, the upper bound \eqref{eqn:ldp-upper-bound} holds for all closed sets $\mathcal S$.
\end{lem}

\begin{proof}
We assume the hypotheses of the lemma, and for any $A\in(0, \infty)$, we let $\mathcal K_A \subset \X \times \R_+$ denote the set defined in \eqref{d:KA}.

\textit{Step 1.}
We first fix an arbitrary closed set $\mathcal S \subset \mathcal X \times \R_+$ and establish the claim in  \eqref{eqn:ldp-upper-bound} for $\mathcal S$; namely,
\begin{equation}\label{eqn:ldp-upper-bound2}
\limsup_{n \rightarrow \infty} \frac{1}{n} \log
\P\left( \left(\Pi(n^{-1/2} (\a^{(n)})^\trans X^{(n)}),
\frac{1}{n} \sum_{i=1}^n \eta(X_i)
 \right) \in \mathcal S \right) \le - \inf_{(\w,r,s) \in \mathcal S} I(\w,r,s)
\end{equation}
holds  for $\sigma$-almost every sequence $\a \in\mathbb V$.

For any $A\in (0,\infty)$  let $\Omega_{\mathcal S,A} \subset \V$ be as in the lemma statement, and set $\Omega_{\mathcal S} = \cap_{m = 1}^\infty \Omega_{\mathcal S,m}$.
Then clearly $\sigma(\Omega_{\mathcal S}) = 1$.
We next observe that we have $I(\w,r,s)\ge 0$ for all choices of $\w,r,s$; this follows on taking $\balpha = (0,0,\dots)$ and $b=c=0$ in \eqref{lambdastar}.
Let $\alpha\in [0, \infty)$ be a parameter such that $\mathcal S\subset\{I\ge \alpha\}$.
By \Cref{l:tightness} and \Cref{l:tightness2}, there exists some $m\in \mathbb{N} $ such that for every sequence $\a\in\mathbb V$,
\begin{align*}
\begin{split}
    \limsup_{n\rightarrow\infty}\frac{1}{n}\log\mathbb P\left(\left(\Pi(n^{-1/2} (\a^{(n)})^\trans X^{(n)}),
\frac{1}{n} \sum_{i=1}^n \eta(X_i)
 \right) \in \mathcal K_{m}^c\right)\leq -\alpha.
\end{split}
\end{align*}
By hypothesis,  the upper bound \eqref{eqn:ldp-upper-bound} holds for $\mathcal S\cap \mathcal K_{m}$ and all $\a \in \Omega_{\mathcal S}$. Since $\mathcal S\subset\{I\ge \alpha\}$, this implies that for every sequence $\a \in \Omega_{\mathcal S}$,
\begin{align*}
 \begin{split}
     \limsup_{n\rightarrow\infty}\frac{1}{n}\log\mathbb P\left(\left(\Pi(n^{-1/2} (\a^{(n)})^\trans X^{(n)}),
\frac{1}{n} \sum_{i=1}^n \eta(X_i)
 \right) \in \mathcal S\cap \mathcal K_{m}\right)\leq -\alpha.
 \end{split}
 \end{align*}
Together, the previous two displays imply that for all $\a \in \Omega_{\mathcal S}$,
\begin{align*}
\begin{split}
    \limsup_{n\rightarrow\infty}\frac{1}{n}\log\mathbb P\left(\left(\Pi(n^{-1/2} (\a^{(n)})^\trans X^{(n)}),
\frac{1}{n} \sum_{i=1}^n \eta(X_i)
 \right) \in \mathcal S\right)\leq -\alpha.
\end{split}
\end{align*}
We can set
$\alpha \coloneqq \inf_{(\w,r,s)\in \mathcal S} I(\w,r,s)$, since we required only that $\mathcal S \subset \{ I \ge \alpha\}$, and hence the claim of Step 1 follows.

\textit{Step 2.} We next show that for $\sigma$-almost every sequence $\a=\left(\mathbf a^{(1)},\mathbf a^{(2)},\ldots\right)\in\mathbb V$, the upper bound \eqref{eqn:ldp-upper-bound} holds for all closed sets $\mathcal S$.

Note that the space $\mathcal X\times\mathbb R_+$ possesses a countable basis $\{\mathcal O_i\}_\sequence{i}$  since it is a separable metric space. Consider the countable family of closed sets $\mathfrak C=\{\cap_{i\in\mathcal F} \mathcal O_{i}^c:\mathcal F \text{ is finite subset of }\mathbb N\}$ and define $\widehat\Omega=\cap_{\mathcal S\in\mathfrak C}\Omega_\mathcal S $. Then $\P(\widehat{\Omega})=1$.
For an arbitrary closed set $\mathcal S$, we have $\mathcal S=\cap_{i\in\mathcal F^{\prime}}\mathcal O_i^c$ for some $\mathcal F^{\prime}\subset\mathbb N$. For any $\varepsilon>0$, there exists a finite set  $\mathcal F\subset\mathcal F^{\prime}\subset\mathbb N$, such that for $\mathcal C=\cap_{i\in\mathcal F}\mathcal O_i^c$, we have
\begin{align*}
\begin{split}
    \inf_{\left(\w,r,s\right)\in\mathcal C}I\left(\w,r,s\right)>\inf_{\left(\w,r,s\right)\in\mathcal S}I\left(\w,r,s\right)-\varepsilon.
\end{split}
\end{align*}
Therefore, for $\a \in \widehat \Omega$,
using $\mathcal S \subset \mathcal C$ and the previous step, we have
\begin{align*}
\begin{split}
    &\ \limsup_{n\rightarrow\infty}\frac{1}{n}\log\mathbb P\left(\left(\Pi(n^{-1/2} (\a^{(n)})^\trans X^{(n)}),
\frac{1}{n} \sum_{i=1}^n \eta(X_i)
 \right) \in \mathcal S\right)\\
    \leq &\ \limsup_{n\rightarrow\infty}\frac{1}{n}\log\mathbb P\left(\left(\Pi(n^{-1/2} (\a^{(n)})^\trans X^{(n)}),
\frac{1}{n} \sum_{i=1}^n \eta(X_i)
 \right) \in \mathcal C\right)\\
    \leq &\ -\inf_{\left(\w,r,s\right)\in\mathcal C}I\left(\w,r,s\right)\\
    < &\ -\inf_{\left(\w,r,s\right)\in\mathcal S}I\left(\w,r,s\right)+\varepsilon.
\end{split}
\end{align*}
Since $\varepsilon>0$ is arbitrary, for every $\a \in \widehat \Omega$ the upper bound \eqref{eqn:ldp-upper-bound} holds for all closed sets $\mathcal S$. The proof is complete.
\end{proof}

\begin{proof}[Proof of \Cref{l:ldpupper}]
We begin by introducing some useful notation. Given positive integers $k \ge m$, let $\mathcal J_{k,m}$ denote the set of injective mappings from $\{1,\dots, m\}$ to $\{1,\dots , k \}$. For $\x\in \R^k$ and $\tau \in \mathcal J_{k,m}$, we define $\tau (\x) \in \R^k$ as the vector
\begin{equation}\label{d:tau}
\tau (\x) = (x_{\tau(1)}, x_{\tau(2)}, \dots, x_{\tau(m)}, x_{m+1}, x_{m+2}, \dots)
\end{equation}
whose first $m$ coordinates are $x_{\tau(1)}, x_{\tau(2)}, \dots, x_{\tau(m)}$, with the remaining coordinates taken in their original order.

By \Cref{l:topologicalpropertyofX}(3) and \Cref{l:from-compact-to-closed}, to prove the upper bound \eqref{eqn:ldp-upper-bound} for $\sigma$-a.e. sequence $\a \in \V$, it suffices to prove \eqref{eqn:ldp-upper-bound} for compact sets.
For the remainder of this proof, we fix a compact set $\mathcal C$ and a corresponding constant $A$ such that $\max\left(r,s\right)\leq A$ for all $\left(\w,r,s\right)\in\mathcal C$.
The proof consists of two further steps.\\

\noindent \emph{Step 1: Reduction to a finite number of coordinates.}
We start by establishing some general properties of the set $\mathcal C$.
Set
\begin{equation}\label{313}
\alpha =  \alpha(\C) \coloneqq  \inf_{(\w,r,s) \in \C }  I(\w,r,s).
\end{equation}
Let $\delta >0$ be a parameter and define
\be\label{314}
\alpha_\delta = \min(\alpha, \delta^{-1}) - \delta.
\ee
Then from the definition of $\alpha$ in \eqref{ratefunction}, and \eqref{lambdastar}, we see that for every $(\w,r,s )\in \C$, there exists a triple
\begin{equation}\label{abcdef}
(\balpha,b,c) = \big(\balpha(\w,r,s, \delta),b(\w,r,s, \delta), c(\w,r,s, \delta\big) \in \ell^2  \times \R_+ \times \R_+
\end{equation}
such that
\begin{equation}\label{319}
\langle \balpha, \w \rangle + bt + cs  - \Lambda(\balpha,b,c )  > \alpha_\delta,
\end{equation}
with $t = \sqrt{r^2 - \|\w\|^2_2 }$.
By \eqref{monotonic}, there exists $m=m(\w,r,s,\delta)\in \N$ such that
\begin{equation}
\langle \balpha_\head , \w_\head  \rangle +  b t_m + c s - \Lambda_m(\balpha_\head, b, c)  > \alpha_\delta,
\end{equation}
where 
$t_m=\sqrt{r_1^2 - \|\w_\head\|_2^2 }$, and $\balpha_\head = \balpha_{\head m}$.

Consider the open set of near-optimizers
\begin{equation}\label{bdef}
B_{\w,r,s}
=
\left\{
(\hatw,t,s) \in \R^m \times \R_+ \times \R_+: \balpha_\head \cdot \hatw + bt + cs -
\Lambda_m(\balpha_\head, b, c) > \alpha_\delta
\right\},
\end{equation}
where $(\balpha,b,c)$ are chosen as in \eqref{abcdef}, and for $(\w,r,s ) \in \X \times \R_+$, define the set
\begin{equation}
\mathcal W_{\w,r,s} = \left\{
(\w,r, s) \in \mathcal X \times \R_+ : \left( \w_\head, \sqrt{r^{2} - |\w_{\head}|^2 }, s  \right) \in B_{\w,r,s}, r < A, s < A
\right\}.
\end{equation}
This is an open set that contains $(\w,r,s)$, by \eqref{319}.
Since  $\C$  is compact, there exists a collection of finitely many open sets 
$\mathcal W^{(j)} = \mathcal W_{\w_j,r_j,s_j}$, $j=1, \dots, m_0$, with $(\w_j,r_j,s_j) \in \mathcal X \times \R_+$, such that 
\begin{equation}\label{compactcover}
\C \subset \bigcup_{j=1}^{m_0} \mathcal W^{(j)}.
\end{equation}
Let $B^{(j)} = B_{ \w_j, r_j, s_j}$, and let $(\balpha^{(j)},b^{(j)},c^{(j)})$ be the triple in \eqref{abcdef} associated with $(\w_j, r_j, s_j)$.

Next, we write $k = k_n$, suppressing the dependence on $n$, and recall the definition of $\mathcal J_{k,m}$ made before \eqref{d:tau}. Given $\w\in \R^k$, we consider the map $\tau^\w \in \mathcal J_{k,m}$ such that the corresponding induced map $\tau^\w \colon \R^k \rightarrow \R^m$ defined in \eqref{d:tau} satisfies $\tau^\w (\w) _{\head m}= [ \w ]_{\head m}$.
Then we see that for any $j=1, \dots, m_0$,
\begin{align}\label{bhat0} &\left\{ (\w, s) \in \R^k \times \R_+ : ( [\w], \|\w\|_2  , s ) \in \mathcal W^{(j)} \right\}\\
&\subset \left\{ (\w, s) \in \R^k \times \R_+ : ([\w]_\head , \| [\w]_\tail\|_2, s )\in B^{(j)} , \| \w \|_2\le A\right\}\notag \\
&=
\left\{  (\w, s) \in \R^k \times \R_+ : (\tau^\w (\w)_\head , \| \tau^\w (\w) _\tail\|_2, s )\in B^{(j)}, \| \w \|_2 \le A  \right\}.
\end{align}

Now, for $j=1, \dots, m_0$, define, for any $\tau \in \mathcal J_{k,m}$,
\begin{equation*}
\widehat B_{k, \tau}^{(j)} \coloneqq
\left\{ (\w,s) \in \R^k \times \R_+ :
( \tau (\w)_\head, \| \tau (\w)_\tail\|_2 , s) \in B,
s\le A, \| \w \|_2 \le A
\right\},
\end{equation*}
and note that \eqref{bhat0} implies 
\begin{equation}\label{bhat}
\{ (\w, s) \in \R^k \times \R_+ : ([\w], \|\w\|_2 , s ) \in \mathcal W \} \subset \widehat B^{(j)}_k
\coloneqq \bigcup_{\tau \in \mathcal J_{k,m}}  \widehat B^{(j)}_{k, \tau}.
\end{equation}

To apply these properties of $\mathcal C$, we introduce the notation
\begin{equation}\label{def-of-W-L}
W= W^{(n)} \coloneqq  n^{-1/2} (\a^{(n)})^\trans X^{(n)}, \qquad L  = L^{(n)} \coloneqq \frac{1}{n} \sum_{i=1}^n \eta(X_i).
\end{equation}
By \eqref{bhat0} and \eqref{bhat}, if $([W],\|W \|_2, L)  \in \mathcal W^{(j)}$, then $(W,L) \in \widehat B_k^{(j)}$.
Combining this observation with \eqref{compactcover} and a union bound, we find
\be
\P \left(( \Pi(W), L)  \in \mathcal{C} \right) \le \sum_{j=1}^{m_0}
\P \left(( \Pi(W), L)  \in \mathcal{V}^{(j)} \right) \le \sum_{j=1}^{m_0} \P\left( (W,L) \in \widehat B^{(j)}_k\right).\label{penultimate}
\ee

It therefore suffices to bound the probabilities  $
\P( (W,L) \in \widehat B^{(j)}_k)$. For concreteness, we focus on $
\P( (W,L) \in \widehat B^{(1)}_k)$. The argument for the other terms is analogous.\\

\noindent \emph{Step 2: Probability bound for $\widehat B^{(1)}_k$.} 
Set $\widehat B_k = \widehat B^{(1)}_k$ and $(\balpha, b,c) = (\balpha^{(1)}, b^{(1)}, c^{(1)})$. 
From \eqref{bhat}, we have
\begin{equation}\label{Hsummand}
\one_{\{(W, L) \in \widehat B_k\}}
\le
\sum_{\tau \in \mathcal{J}_{k,m}} \one_{\{ (\tau (W)_\head, \| \tau (W)_\tail\|_2, L) \in B\}}\ \one_{\{\| W \|_2 \le A \} } \one_{\{ L \le A\} }.
\end{equation}
Using the quantity
\begin{equation}\label{Habc}
H_{\balpha,b,c }(B) \coloneqq \inf_{(\hatw,t, s) \in B} \left\{ \langle \balpha_\head , \hatw \rangle + bt + cs\right\}
\end{equation}
to bound the first indicator in the summand in \eqref{Hsummand}, we obtain
\begin{multline}\label{firstsum}
\one_{\{(W,L) \in \widehat B_k\}}
\le
\exp\big( - n H_{\balpha,b,c}(B)\big) \\\times
\sum_{\tau \in \mathcal{J}_{k,m}} \exp\big(n (\balpha_\head \cdot \tau (W)_\head + b\| \tau (W)_\tail\|_2 + c L )\big) \one_{\{
\| W \|_2 \le A \} }\one_{\{ L  \le A \}}.
\end{multline}

With a view toward bounding the right-hand side of \eqref{firstsum}, we first let $\kappa>0$ be a parameter. 
We define
\begin{equation*}
\mathcal G \coloneqq \left\{ \v' \in \R^{k-m} : \| \v'\|_2 = b
\right\},
\end{equation*}
and let $\operatorname{Cov}(\mathcal G)$ denote a minimal cover of the set $\mathcal G$
by open $\ell^2$ balls of radius $b\kappa$ with centers in $\mathcal G$.
We let $\mathcal U'$ denote the set of the centers of balls in $\operatorname{Cov}(\mathcal G)$, and define
\begin{equation}\label{d:bigU}
\mathcal U \coloneqq \bigcup_{\tau \in \mathcal{J}_{k,m}}
\left\{
\v \in \R^k : \tau (\v)_\head = \balpha_\head, \tau (\v)_\tail \in \mathcal U' \cup \{0\}
\right\}.
\end{equation}
A standard volume estimate (see \cite[Proposition~4.2.12]{Ver18}) shows that
\[| \mathcal U'| \le (1 + 2/\kappa)^{k -m} \le (1 + 2/\kappa)^{k}.
\]
Together with the bound $|\mathcal J_{k,m} | \le k^m $, this implies that there exist $\hat \eps >0$ and $n_0=n_0( \{k_n\}_\sequence{n}, \kappa) \in \N$ such that for all $n \ge n_0$,
\begin{align}\label{volumeestimate}
\begin{split}
    |\mathcal U|\le \exp\big( (1-\hat \eps) n\big).
\end{split}
\end{align}

We now show that we can approximate the sum in \eqref{firstsum} by a sum over elements in $\mathcal U$. Fix $\tau \in \mathcal{J}_{k,m}$.
We first consider the event where $\| W \|_2 \le A $ and $\|\tau (W)_\tail\|_2  > \kappa$. Since
\[
U \coloneqq  b \frac{ \tau (W)_\tail}{\| \tau (W)_\tail \|_2 }
\] lies in $ \mathcal G$, there exists $\v \in \mathcal U$ 
such that $\tau (\v)_\head =\balpha_\head$ and $\|U - \tau (\v)_\tail \|_2 < b\kappa$.
Then
\begin{align}
&\balpha_\head \cdot \tau (W)_\head + b \| \tau (W)_\tail \|_2 \notag \\
&= \tau (\v)_\head \cdot \tau (W)_\head + U\cdot\tau (W)_\tail \notag \\
&= \tau (\v)_\head \cdot \tau (W)_\head
+ \tau(\v)_\tail \cdot \tau (W)_\tail
+ ( U - \tau(\v)_\tail) \cdot \tau (W)_\tail\notag  \\
&\le \v\cdot W + A b \kappa,\label{Wupper1}
\end{align}
where to obtain the last inequality, we used the Cauchy--Schwarz inequality, and the estimates $\|\tau(W)_\tail\|_2 \le \| W \|_2 \le A$ and $\| U - \tau(\v)_\tail \|_2 \le b \kappa$. 
On the other hand, on the event $\|W \|_2 \le A $ and $\|\tau (W)_\tail\|_2  \le  \kappa$, a direct bound shows that for any $\v$ such that $\tau (\v)_\head = \balpha_{\head}$ and $\tau (\v)_\tail = 0$ (which in particular satisfies $\v \in \mathcal U$), we have
\begin{equation}\label{Wupper}
\balpha_\head \cdot \tau (W)_\head + b \| \tau (W)_\tail \|_2
= \tau (\v)_\head \cdot  \tau (W)_\head
+ b \| \tau (W)_\tail \|_2
\le \v\cdot W + b \kappa.
\end{equation}
The upper bounds \eqref{Wupper1} and \eqref{Wupper} together with \eqref{firstsum} show that
\begin{equation}
\P ( (W,L)\in \widehat B_k) \le \exp( - n (H_{\u,b,c}(B) - (A+1) b \kappa ) ) \left| \mathcal J_{k,m} \right|
\sum_{\v \in \mathcal U } \E\left[ \exp(n (\v \cdot W  + c L ) ) \right].
\end{equation}
Taking logarithms, we find
\begin{align}\label{boundmax}
n^{-1} \log \P ( (W,L) \in \widehat B_k) \le&  -  (H_{\u,b,c}(B) - (A+1) b \kappa)
+ n^{-1} \log \left| \mathcal J_{k,m} \right| | \mathcal U |\\  & +
 \max_{\v \in \mathcal U }  n^{-1} \log \E \left[\exp(n (\v \cdot W  + c L ) )\right].\notag
\end{align}

Recalling the definitions of $W$ and $L$ in \eqref{def-of-W-L}, the fact that the variables $\{X_i \}_\sequence{i}$ are i.i.d., the definition of $\overline{\Lambda}$ in \eqref{asprev}, and the definition of $F_n$ in \eqref{eqn:quenched-log-moment}, we have
\begin{align*}
n^{-1} &\log \E \left[\exp(n (\v \cdot W  + c L ) )\right] \\
&= n^{-1}\log \E \left[\exp\left(n^{1/2} \sum_{i=1}^n  X_i \left( \sum_{j=1}^{k_n}   v_j a^{(n)}_{ji}  \right) + c\sum_{i=1}^n \eta(X_i) \right)\right]\\
&=n^{-1}\sum_{i=1}^n \log \E \left[\exp\left(  X_i
\langle \v, \sqrt{n} \a_i^{(n)} \rangle
+ c\eta(X_i)  \right) \right]\\
&= F_n(\v,\a^{(n)}, c).
\end{align*}
Further, by \eqref{d:bigU}, each $\v \in \mathcal U$ satisfies either $\| \v \|^2 = \|\bm\u_\head\|_2^2$ or $\| \v \|^2 = \|\bm\u_\head\|_2^2 + b^2$; that is, the $\ell^2$ norms of all vectors $\v \in \mathcal U$ take one of two values. 
Recalling the previous estimate on $|\mathcal U|$ from \eqref{volumeestimate}, taking the limit superior of both sides of \eqref{boundmax},  using the last display and \Cref{l:lln} to control the maximum therein, and recalling $k_n/n \rightarrow 0$ from the assumption \eqref{knhypo},
we obtain
\begin{align*}
\limsup n^{-1} \log \P \big( (W,L) \in \widehat B_k\big)
&\le -H_{\bm\u,b,c}(B) + (A+1) b \kappa + \max(\Lambda_m (\bm\u_\head,b,c),\Lambda_m (\bm\u_\head,0,c) )\\
&\le -H_{\bm\u,b,c}(B) + (A+1) b \kappa + \Lambda_m (\bm\u_\head,b,c) .
\end{align*}
In the last line, we used Jensen's inequality (as in \eqref{monotonic}) to simplify the maximum.
By the definitions of the sets $H_{\balpha, u ,c}$ and  $B$ in \eqref{Habc} and \eqref{bdef}, respectively, this becomes
\begin{align*}
\limsup n^{-1} \log \P \big( (W ,L)\in \widehat B_k\big)
&\le  - \alpha_\delta  - \Lambda_m (\bm\u_\head, b,c) + (A+1) b \kappa + \Lambda_m(\bm\u_\head,b,c)\\
&\le  - \alpha_\delta + (A+1) b \kappa.
\end{align*}
Since this bound holds for all $\kappa >0$, it implies that
\begin{equation}
\limsup n^{-1} \log \P \big( (W,L) \in \widehat B_k\big)
\le - \alpha_\delta.
\end{equation}
Finally, since this holds for all $\delta >0$, we have by \eqref{313} and \eqref{314} that
\begin{equation}
\limsup n^{-1} \log \P \big( (W,L) \in \widehat B_k\big)
\le - \alpha = \inf_{(\w,r,s) \in \mathcal C} I(\w, r, w).
\end{equation}
Inserting this and the analogous bounds for all $B_k^{(j)}$, $j=1,\dots, m_0$, into \eqref{penultimate} completes the proof.
\end{proof}

\section{Lower Bound}\label{s:lower}

\subsection{Preliminary Lemmas} \label{s:lowerpreliminary}
We begin by stating some preliminary results.
Given a metric space $(\mathcal T,d)$, $t_0 \in \mathcal T$ and $r \ge 0$, set
\be\label{eqn:ball-def}
B(t_0, r) \coloneqq \{ t \in \mathcal T : d(t_0, t) < r\}.
\ee
Given $m \in \N$ and $\mathbf v \in \ell^2$, we recall the quantities $\mathbf v_{\tail m}$ and $\mathbf v_{\head m}$ defined in \eqref{mtail}. Recalling the constant $T$ from \Cref{assume3}, we set
\be
\mathcal D_{m}
\coloneqq \{ (\v, b, c) : (\v,b) \in \R^m \times \R, c \in (-\infty, T) \}.
\ee
Let $\Lambda^*_m$ denote the Legendre transform of the function $\Lambda_m$ from \eqref{lambdam}, defined by
\be\label{lambdamstar}
\Lambda_m^*(\hat  \w, t,s)\coloneqq
\sup_{(\hat \balpha,b,c) \in \R^m \times \R \times \R}
\left\{
\langle \hat \balpha,  \hat \w \rangle + bt + cs  - \Lambda_m (\hat \balpha,b,c ) \right\}.
\ee
We denote the domain of $\Lambda_m^*$ by
\begin{equation}
\label{def-of-Lambdam-star}
\mathcal D_m^*\defeq{}\left\{(\hatw,r,s):\Lambda_m^*(\hatw,r,s)<\infty\right\}.
\end{equation}
Finally, let $\nabla \Lambda_m \in \R^{m+2}$ denote the gradient of 
$\Lambda_m$:
\begin{multline}
\nabla \Lambda_m (\v, b, c)
= \\\left(
\partial_{v_1} \Lambda_m (\v, b, c),
\partial_{v_2} \Lambda_m (\v, b, c),
\dots,
\partial_{v_m} \Lambda_m (\v, b, c),
\partial_{b} \Lambda_m (\v, b, c),
\partial_c \Lambda_m (\v, b, c)
\right).
\end{multline}

The first result pertains to the existence of a point with useful properties in any open set $\mathcal O\subset\mathcal X\times \R$. Its proof is relegated to \Cref{s:top-lem}.
\begin{lem}\label{l:additional-good-properties}
Fix $\eps>0$, an open set $\mathcal O \subset \mathcal X \times \R$, and a point $(\mathbf w,r,s) \in \mathcal O$ such that $I(\mathbf w,r,s)< \infty$.
Then there exist $\bar \kappa >0$, $m_0 \in \N$ and $(\overline{\mathbf w},\bar t,\bar s)\in\mathbb R^{m_0}\times\mathbb R_+\times\mathbb R_+$ such that the following claims hold:
\begin{enumerate}
\item $(\bar{\mathbf w},\bar t,\bar s)=\nabla\Lambda_{m_0}(\bar{\mathbf v},\bar b,\bar c)$ for some $(\bar{\mathbf v},\bar b,\bar c)\in\mathcal D_{m_0}$;
\item $|\bar{\mathbf w}_j|>|\bar{\mathbf w}_{j+1}|>0,$
for all $1\le j \le m_0 -1$;
\item $\Lambda_{m_0}^*\left(\bar{\mathbf w},\bar t,\bar s\right)\leq I(\mathbf w,r,s)+\varepsilon$;
\item For all $(\mathbf w',r',s') \in \X \times \R_+$, if
\begin{equation}\label{claim:2}
\left(\mathbf w'_{\head m_0},\sqrt{r'^2-\|\mathbf w'_{\head m_0}\|_2^2},s'\right)\in B\left((\overline{\mathbf w},\bar t,\bar s),\bar \kappa/2\right),
\end{equation}
then $(\mathbf w',r',s')\in \mathcal O$.
\end{enumerate}
\end{lem}

We now present the second result, whose proof is deferred to \Cref{proof:lemma-small-ball}. Recall the definitions of $W^{(n)}$ and $L^{(n)}$ from \eqref{def-of-W-L}.
\begin{lem}\label{lemma:small-ball-probability}
Fix $m\in\mathbb N$, $\kappa>0$, and $(\hatu,b,c)\in\mathcal D_{m}$. Define $(\hatw,t,s)\coloneqq\nabla\Lambda_m(\hatu,b,c)$ and $B\coloneqq B\left((\hatw,t,s),\kappa \right)$. Then, for $\sigma$-a.e. $\mathbf a=(\a^{(1)},\a^{(2)},\cdots)$, the sequence $\{W^{(n)},L^{(n)}\}$, $n\in\mathbb N$, satisfies
\begin{align}\label{eqn:small-ball-probability}
\begin{split}
    \liminf _{n \rightarrow \infty} \frac{1}{n} \log \P \left(\left(W_{\head m}^{(n)},\left\|W_{\tail m }^{(n)}\right\|_{2},L^{(n)}\right) \in B,\left\|W_{ \tail m }^{(n)}\right\|_{\infty} \leq 2 m^{-1 / 2}\right)\geq-\Lambda_m^*(\hatw,t,s).
\end{split}
\end{align}
\end{lem}

\subsection{Proof of the Lower Bound.} \label{s:lowerproof}
\begin{proof}[Proof of \Cref{l:ldplower}]
Recall the definitions of $W^{(n)},L^{(n)}$ in \eqref{def-of-W-L}. We first establish the following claim.

\emph{Claim 1.} For every $\eps>0$,  open set $\mathcal O\subset \mathcal X\times \R_+$, and point $(\mathbf w,r,s)\in \mathcal O$, there exists a set $\Omega_{\mathcal O,\varepsilon,(\mathbf w,r,s)} \subset \V$ of sequences $\a=\{\a^{(n)}\}_{n \in \N}$ with $\sigma (\Omega_{\mathcal O,\varepsilon,(\mathbf w,r,s)}) = 1$ and such that 
\begin{align}\label{eqn:simplified-ldp-lower-bound}
\begin{split}
    \liminf_{n\rightarrow\infty} \frac{1}{n} \log \P \left(\left([W^{(n)}],\|W^{(n)}\|_2,L^{(n)}\right)\in \mathcal O\right)\geq-I(\mathbf w,r,s)-\varepsilon.
\end{split}
\end{align}

We may assume that $I(\mathbf w,r,s)<\infty$, since otherwise the claim is trivial.
In this case, there exist $\bar \kappa>0$, $m_0 \in \N$, and
 $(\bar{\mathbf w},\bar t,\bar s) \in \R^{m_0} \times \R_+ \times \R_+$ that satisfy the properties of \Cref{l:additional-good-properties}.

Fix $m>m_0$ such that $3m^{-1/2}<|\bar{\mathbf w}_{m_0}|$ (such an $m$ exists since $|\bar{\bm w}_{m_0}|>0$ by property (2) of \Cref{l:additional-good-properties}) and
set
\[\kappa\coloneqq\min\left\{\frac{\bar \kappa}{2},m^{-1/2},\frac{1}{3}\min_{1\leq j\leq m_0-1}\big(|\bar{\mathbf w}_j|-|\bar{\mathbf w}_{j+1}|\big)\right\}.\]
Note that property (2) of \Cref{l:additional-good-properties} ensures $\kappa>0$. Set
\begin{align*}
\begin{split}
 \widehat{\mathbf w}\coloneqq (\bar{\mathbf w},0,\ldots,0)\in\mathbb R^m,\quad
 \,B_1\coloneqq B\left((\widehat{\mathbf w},\bar t,\bar s),\kappa\right)
 , \quad
B_2 \coloneqq  B\left((\overline{\mathbf{w}}, \bar{t},\bar s), \bar\kappa /2\right)
.
\end{split}
\end{align*}
Property (1) of \Cref{l:additional-good-properties} implies that there exists $\hatu\coloneqq(\bar{\bm v},0,\ldots,0)\in\mathbb R^m$ such that 
\begin{equation}\label{eqn:4.2setup}
    (\hatu,\bar b,\bar c)\in\mathcal D_{m_0},\text{ and }\nabla\Lambda_m(\hatu,\bar b,\bar c)=(\hatw,\bar t,\bar c).
\end{equation}
Given  $(\bm w',r',s')\in\ell^2\times \mathbb R_+\times \mathbb R_+$ with $r'\geq \|\bm w'\|_2$, we write
\begin{align*}
    t'_{\bm w',r'}\coloneqq \sqrt{(r')^2-\|\bm w'_{\head m}\|_2^2},\quad  t''_{\bm w',r'}\coloneqq \sqrt{(r')^2-\|\bm w_{\head m_0}\|_2^2}.
\end{align*}
Consider  $\left(\mathbf{w}^{\prime}, r^{\prime},s'\right) \in \ell^2 \times\R_+ \times\R_+$ with $(\mathbf{w}_{\head m}^{\prime}, t'_{\bm w',r'},s') \in B_1$. By the definition of $B_1$, it follows that $\|\bm w'_{\head m}-\hat{\bm w}\|_\infty \leq \kappa$. The definition of $\kappa,m$ and the strict ordering in the first $m_0$ coordinates of $\hat{\bm w}$ imply that $\bm w'$ satisfy the following properties:
\begin{enumerate}
    \item $|\bm w'_j|>|\bm w'_{j+1}|>2m^{-1/2}$ for all $1\leq j\leq m_0-1$;
    \item $|\bm w'_j|\leq m^{-1/2}$ for all $m_0+1\leq j\leq m$.
\end{enumerate}
If we further require that $\|\bm w'_{\tail m}\|_\infty \leq 2m^{-1/2}$, then it follows that $[\bm w']_{\head m_0}=\bm w'_{\head m_0}$.

The above discussion, and the definitions of $t'_{\bm w',r'}$ and $t''_{\bm w',r'}$, imply that
\begin{align}\label{eqn:reduction-to-tail}
\begin{split}
&\{\left(\mathbf{w}^{\prime}, r^{\prime},s'\right) \in \ell^2 \times\R_+ \times\R_+:\left(\mathbf{w}_{\head m}^{\prime}, t'_{\bm w',r'},s'\right) \in B_1\text{ and }\left\|\mathbf{w}_{\tail m}^{\prime}\right\|_\infty \leq 2 m^{-1 / 2}\}\\
&\subset\{\left(\mathbf{w}^{\prime}, r^{\prime},s'\right) \in \ell^2 \times\R_+ \times\R_+:[\mathbf w']_{\head m_0}=\mathbf w'_{\head m_0}\text{ and }\left(\mathbf{w}^{\prime}_{\head m_0}, t^{\prime \prime}_{\bm w',r'},s'\right) \in B_2\}.
\end{split}
\end{align}
Thus,
for $\sigma$-a.e. $\mathbf a=(\a^{(1)},\a^{(2)},\cdots)$, we see that
\begin{align}\label{eqn:lower-bound-step-1}
\begin{split}
    \liminf_{n\rightarrow\infty}&\ \frac{1}{n} \log \P\left(([W^{(n)}],\|W^{(n)}\|_2,L^{(n)})\in \mathcal O\right)\\
    \geq&\liminf_{n\rightarrow\infty}\ \frac{1}{n}\log \P\big(\left([W^{(n)}]_{\head m_0},\|W^{(n)}_{\tail m_0}\|_2,L^{(n)}\right)\in B_2\big)\\
    \geq&\liminf_{n\rightarrow\infty}\ \frac{1}{n}\log \P\left(\left(W^{(n)}_{\head m},\|W^{(n)}_{\tail m}\|_2,L^{(n)}\right)\in B_1,\|W^{(n)}_{\tail m}\|_\infty\leq 2m^{-1/2}\right)\\
    \geq&-\Lambda_m^{*}(\widehat{\mathbf w},\bar t,\bar s)\\
    \geq&-I(\mathbf w,r,s)-\varepsilon,
\end{split}
\end{align}
where we used \eqref{claim:2} in the first step, \eqref{eqn:reduction-to-tail} in the second step, \eqref{eqn:4.2setup} and \Cref{lemma:small-ball-probability} (with $\delta = \kappa$) in the third step and property (3) of \Cref{l:additional-good-properties} in the last step.
This proves the claim \eqref{eqn:simplified-ldp-lower-bound}.

Next let $\{\mathbf w_l,r_l,s_l\}_{l=1}^\infty$ be a
sequence such that
\[
\lim_{l \rightarrow \infty} I(\mathbf w_l,r_l,s_l) = \inf_{(\mathbf w,r,s) \in \mathcal O}
I(\mathbf w,r,s),
\]
and define $\Omega_{\mathcal O}\coloneqq\cap_{l=1}^{\infty}\Omega_{\mathcal O,\frac{1}{l},(\mathbf w_l,r_l,s_l)}$, where $\Omega_{\mathcal O,\frac{1}{\ell},(\bm w_\ell,r_\ell,s_\ell)}$ is the set in from Claim 1. Then $\sigma(\Omega_{\mathcal O})=1$ and
\begin{align}\label{eqn:remove-epsilon}
\begin{split}
    \liminf_{n\rightarrow\infty}n^{-1}\log \P_n\left(\left([W^{(n)}],\|W^{(n)}\|_2,L^{(n)}\right)\in \mathcal O\right)\geq-\inf_{(\mathbf w,r,s)\in \mathcal O}I(\mathbf w,r,s).
\end{split}
\end{align}
Next, observe that since $\mathcal X \times \R_+$ is a separable metric space, it possesses a countable basis $\{\mathcal O_i\}_\sequence{i}$. Setting $\widehat\Omega\coloneqq\cap_\sequence{i}\Omega_{\mathcal O_i}$, we have $\sigma(\widehat\Omega)=1$.
Now consider an arbitrary open set $\mathcal O \subset \mathcal X\times\mathbb R_+$ and any $(\mathbf w,r,s)\in \mathcal O$. By the definition of a basis, there exists a basis element $\mathcal O_i$ such that $(\mathbf w,r,s)\in \mathcal O_i \subset \mathcal O$. For sequence $\bm a\in\widehat\Omega\subset \Omega_{\mathcal O_i}$, the corresponding sequence $W^{(n)}\coloneqq W^{\bm a^{(n)}}$ satisfies
\begin{align}\label{prevdisp}
\begin{split}
    &\liminf_{n\rightarrow\infty} \frac{1}{n} \log \P \left(\left([W^{(n)}],\|W^{(n)}\|_2,L^{(n)}\right)\in \mathcal O \right) \\ &\geq \liminf_{n\rightarrow\infty} \frac{1}{n} \log \P \left(\left([W^{(n)}],\|W^{(n)}\|_2,L^{(n)} \right)\in \mathcal O_i \right)\\
&\geq-I(\mathbf w,r,s),
\end{split}
\end{align}
where the latter follows from \eqref{eqn:remove-epsilon} applied with $\mathcal O$ replaced with $\mathcal O_i$.
Since $(\mathbf w,r,s)\in \mathcal O$ is arbitrary, taking the supremum of the right-hand side of \eqref{prevdisp} over $(\bm w,r,s)\in\mathcal O$ yields \eqref{l:ldplower}.
\end{proof}

\section{Proof of the Main Result}\label{s:mainresult}

\begin{proof}[Proof of \Cref{t:main}]
Combining \Cref{l:ldpupper} and \Cref{l:ldplower}, we conclude that, for $\sigma$-a.e. $\a=(\a^{(1)}, \a^{(2)}, \dots ) \in \V$, the sequence of random variables
\begin{align}
\begin{split}
\left(\Pi\left(n^{-1/2}\left(\mathbf a^{(n)}\right)^\trans X^{(n)}\right),\frac{1}{n}\sum_{i=1}^n\eta(X_i)\right),\quad n\in\mathbb N,
\end{split}
\end{align}
satisfies an LDP with speed $n$ and
good rate function $I(\w,r,s)=\Lambda^{*}(\mathbf x,\sqrt{r_1^2-\|\mathbf x\|_2^2},r_2)$.
By \Cref{assume1} and the definition of $\Pi$ in \eqref{def-Pi},
\begin{align*}
\begin{split}
\Pi\left(n^{-1/2}\left(\mathbf a^{(n)}\right)^\trans Y^{(n)}\right)\stackrel{d}{=}\Pi\left(n^{-1/2}\left(\mathbf a^{(n)}\right)^\trans X^{(n)}\right)\times\rho\left(\frac{1}{n}\sum_{i=1}^n\eta(X_i)\right),
\end{split}
\end{align*}
where $\rho:\mathbb R_+\rightarrow\mathbb R_+$ is continuous.
By applying the contraction principle (see \cite[Theorem 4.2.1]{DZ93}) to the continuous map
$
f\colon \mathcal X\times\mathbb R_+ \rightarrow \mathcal X$ given by 
$f(\w,r,s) \coloneqq (\rho(r_2)\mathbf x,\rho(r_2)r_1),$
we obtain an LDP for $\{\Pi(n^{-1/2}\left(\mathbf a^{(n)}\right)^\trans Y^{(n)})\}_{n \in \N}$ with speed $n$ and good rate function
\begin{align*}
\begin{split}
\mathcal I(\mathbf y,r)=\inf\{I\left(\w,r,s\right):\mathbf y=\rho(r_2)\mathbf x,\ \rho(r_2)r_1=r\}=
\inf_{r_2\in \R_+}I\left(\frac{\mathbf y}{\rho(r_2)},\frac{r}{\rho(r_2)},r_2\right).
\end{split}
\end{align*}
\end{proof}

Given \Cref{t:main}, the proof of \Cref{t:dnp} closely follows \cite[Section 5]{kim2021large}. We include the details for completeness.

\begin{proof}[Proof of \Cref{t:dnp}]

We now verify that $Y^{(n)}$ satisfies Assumptions 1--4
when uniformly distributed on the $\ell_n^p$ sphere for some $p\in[2,\infty)$.

Assumption 1: 
By \cite[Lemma~1]{SZ90}, the relation \eqref{ourform} holds with $\{X_i\}_{i\in\mathbb N}$ being the i.i.d.\ sequence with common distribution equal to the $p$-Gaussian distribution (the probability measure on $\mathbb R$ with density proportional to $e^{-|y|^{p} / p}$), $\eta(x)=|x|^p$, and $\rho(y)=y^{-1/p}$.

Assumption 2: This is a direct consequence of the fact that $X_1$ is a $p$-Gaussian distribution, and hence sub-Gaussian for $p\in[2,\infty)$.

Assumption 3: It is straightforward to see that the domain of $\overline\Lambda_p$ is $\mathbb R\times (-\infty,\frac{1}{p})$. The essential smoothness is established in \cite[Lemma~5.8]{GKR17}.

Assumption 4: The growth conditions on $\overline\Lambda_p$ and the derivatives of $\overline\Lambda_p$ are established in Lemma \ref{l:bound-on-Lambda}.

By \Cref{t:main}, this proves the LDP when $Y^{(n)}$ is uniformly distributed on $\S^p_n$, and the corresponding rate function is $\mathcal I_p$ defined in \eqref{calI}.

Next, we consider the case where $Y^{(n)}$ is uniformly distributed on $\mathbb B_p^n$.
Let $U$ be a uniform random variable on $[0,1]$. We recall that if the random vector $X^{(n)} \in \R^n$ is uniformly distributed on the sphere $\mathbb S_p^n$ and independent from $U$, then $U^{1/n} X^{(n)}$ is uniform on $\mathbb B_p^n$ \cite[Lemma~1]{SZ90}. It then suffices to prove the LDP for $Y^{(n)} =U^{1/n} X^{(n)}$.
By \cite[Lemma 3.3]{GKR17}, the sequence $\{U^{1/n}\}_{n\in\N}$ satisfies an LDP with good rate function $I_U$ given by $I_U(u) = - \log u$ for $u \in (0,1]$, and $I_U(u) = + \infty$ otherwise. Recalling that $U$ and $X^{(n)}$ are independent, and using \Cref{t:main}, we find that $
\{ (\Pi\left(n^{-1/2} (\a^{(n)})^\trans X^{(n)}\right), U^{1/n} ) \}_{n\in\N}$ 
satisfies an LDP in $\mathcal X \times \R_+$ with good rate function $ I_U(u) + \mathcal I_p (\x, r) $. By the contraction principle applied to the continuous map
    $\mathcal X\times\R_+\ni (\bm x,r,u)\mapsto (u\bm x,ur)\in\mathcal X,$
and the positive homogeneity of the map $\Pi$ defined in \eqref{def-Pi}, it follows that $\left\{\Pi\left(U^{1/n} n^{-1/2} (\a^{(n)})^\trans X^{(n)}\right)\right\}_{n\in\N}$
satisfies an LDP on $\mathcal X$ with rate function
\begin{align}
\begin{split}
\widehat I_p(\x,r)
&= \inf_{ (\y, s) \in\mathcal X, u\in\R}
\{I_U(u) + \mathcal I_p (\y, s)
: (\x, r) = (u \y, us)
\}\\
\label{theinfimum}
& = \inf_{u \in (0,1]}
\left\{
-\log u  + \mathcal I_p
\left( \frac{\x}{u}, \frac{r}{u} \right)
\right\}.
\end{split}
\end{align}

It is  clear from the definition \eqref{calI} that the map $t \mapsto \mathcal I_p( t \x, tr )$ is increasing for $t \in \R_+$. Since $u \mapsto u^{-1}$ is monotonically decreasing, we deduce that $u \mapsto \mathcal I_p( u^{-1} \x, u^{-1} r )$ is decreasing.
Combined with the fact that $ - \log u$ is decreasing, we find that the infimum in \eqref{theinfimum} is attained at $u = 1$. Thus
\be
\widehat I_p(\x,r) = \mathcal I_p(\x ,r),\quad (\bm x,r)\in\X,
\ee
as claimed.
\end{proof}

\begin{proof}[Proof of \Cref{c:convex}]
It was shown in the  proof of \Cref{t:dnp} that $Y^{(n)}$ satisfies the assumptions of \Cref{t:main} when uniformly distributed on $\mathbb S^p_n$ for some $p\in[2,\infty)$.
By applying the contraction principle to \Cref{t:main} with the continuous projection map $\X\ni (\bm y,r)\mapsto r\in\mathbb R_+$,
we conclude that $\{\|n^{-1/2}\left(\mathbf a^{(n)}\right)^\trans Y^{(n)}\|_2\}_{n\in\N}$
satisfies an LDP with speed $n$ and good rate function
\begin{align*}
\begin{split}
    \mathbb I_p(r)=\inf\left\{\mathcal I_p(\mathbf y,r):(\mathbf y,r)\in\X\right\}.
\end{split}
\end{align*}
Substituting the definition of $\mathcal I_p$ from \eqref{calI} and recalling the definition of $I_p$ from \eqref{ratefunction}, we have
\begin{align}\label{eqn:rewriteIr}
\begin{split}
    \mathbb I_p(r)&=\inf\left\{I_p\left(\w,r,s\right):\left(\w,r,s\right)\in\X\times\mathbb R_+,\ r_1|r_2|^{-1/p}=r\right\}\\
    &=\inf\left\{I_p\left(\mathbf x,r_1,\left(\frac{r_1}{r}\right)^{p}\right):\left(\mathbf x,r_1\right)\in\mathcal X\right\}\\
    &=\inf\left\{\Lambda_p^*\left(\mathbf x,\sqrt{r_1^2-\|\mathbf x\|_2^2},\left(\frac{r_1}{r}\right)^p\right):\left(\mathbf x,r_1\right)\in\mathcal X\right\}.
\end{split}
\end{align}
By definitions  of $\Lambda_p$ and $\Lambda_p^*$ from \eqref{d:g0} and \eqref{lambdastar}, respectively, and an elementary property of Gaussian variables, we see that 
\begin{align}\label{eqn:rewriteLambdaStar}
\begin{split}
    &\Lambda_p^*\left(\mathbf x,\sqrt{r_1^2-\|\mathbf x\|_2^2},\left(\frac{r_1}{r}\right)^p\right)\\
    &=\sup_{\left(\mathbf u,b,c\right)\in\ell^2\times\mathbb R\times\mathbb R}\left\{\left\langle \mathbf u,\mathbf x\right\rangle+b\sqrt{r_1^2-\|\mathbf x\|_2^2}+c\left(\frac{r_1}{r}\right)^p-\Lambda_p\left(\mathbf u,b,c\right)\right\}\\
    &=\sup_{\left(\mathbf u,b,c\right)\in\ell^2\times\mathbb R\times\mathbb R}\left\{\left\langle \mathbf u,\mathbf x\right\rangle+b\sqrt{r_1^2-\|\mathbf x\|_2^2}+c\left(\frac{r_1}{r}\right)^p-\mathbb E\left[\overline\Lambda_p\left(\|\left(\mathbf u,b\right)\|_2g,c\right)\right]\right\}
\end{split}
\end{align}
By the Cauchy--Schwarz inequality and the fact that $\| \x \|_2 \le r_1$, we have $\left\langle \mathbf u,\mathbf x\right\rangle+b\sqrt{r_1^2-\|\mathbf x\|_2^2}\leq \|(\bm u,b)\|_2r_1$. Further, if $\x$ and $r_1$ are given, there exists a pair $(\u, b) \in \ell^2 \times \R$ such that equality is attained. 
Hence, \eqref{eqn:rewriteLambdaStar} implies
\begin{align}\label{eqn:rewriteLambdaStar2}
    \begin{split}
        &\Lambda_p^*\left(\mathbf x,\sqrt{r_1^2-\|\mathbf x\|_2^2},\left(\frac{r_1}{r}\right)^p\right)\\
        &=\sup_{\left(\mathbf u,b,c\right)\in\ell^2\times\mathbb R\times\mathbb R}\left\{\left\|\left(\mathbf u,b\right)\right\|_2r_1+c\left(\frac{r_1 }{r}\right)^p-\mathbb E\left[\overline\Lambda_p\left(\|\left(\mathbf u,b\right)\|_2g,c\right)\right]\right\}\\
        &=\sup_{\left(v,c\right)\in\mathbb R_+\times\mathbb R}\left\{vr_1+c\left(\frac{r_1 }{r}\right)^p-\mathbb E\left[\overline\Lambda_p\left(vg,c\right)\right]\right\}\\
        &=\sup_{(v,c)\in\mathbb R\times\mathbb R}\left\{vr_1+c\left(\frac{r_1 }{r}\right)^p-\mathbb E\left[\overline\Lambda_p\left(vg,c\right)\right]\right\},
    \end{split}
\end{align}
where the last equality follows from the observation that, since $r_1\geq 0$ and $g$ is symmetric, replacing $v$ by $|v|$ increases the quantity we are taking the supremum over.
Combining \eqref{eqn:rewriteIr} and \eqref{eqn:rewriteLambdaStar2}, we have
\begin{align*}
\begin{split}
    \mathbb I_p(r)=\inf_{\tau\in\mathbb R_+}\sup_{\left(v,c\right)\in\mathbb R\times\mathbb R}\left\{v\tau+c\left(\frac{\tau}{r}\right)^p-\mathbb E\left[\Lambda\left(vg,c\right)\right]\right\}.
\end{split}
\end{align*}
By \cite[Lemma~2.1]{liao2020geometric}, this implies
\begin{align*}
\begin{split}
    \mathbb I_p(r)=\sup_{(v,c)\in\mathbb R\times\mathbb R}\left\{vr+c-\mathbb E\left[\Lambda\left(vg,c\right)\right]\right\}.
\end{split}
\end{align*}
The convexity of $\mathbb I$ now follows from a standard result in convex analysis (see \cite[Theorem~11.1]{RW09}).
\end{proof}

\begin{rmk}\label{r:lqnorms}
The space $\mathcal X$ is defined in \eqref{Xdefinition} using the $\ell_2$ norm in one coordinate.
This allowed us to deduce
\Cref{c:convex} from \Cref{t:dnp} by applying the contraction principle in the previous proof.
The $\ell_2$ norm in this definition could be replaced by $\ell_q$ for any $1 \le q < \infty$, and a parallel argument would give the analogue of \Cref{c:convex} for the $\ell_q$ norm of the projections. For brevity, we do not take this up here.
\end{rmk}

\begin{proof}[Proof of \Cref{c:largest}]
Consider the map $\pi \colon \mathcal X \rightarrow \R$ given by $\pi(\w, r) = w_1$, the first element of the ordered sequence $\w$. By \Cref{t:main} and the contraction principle, the sequence \eqref{maxldp} satisfies an LDP with
good rate function
\be
\mathcal I_{\max{}} (r) = \inf_{(\w,s) \in \mathcal X}
\{ \mathcal I_p (\w,s) : r = \pi(\w,s) \}
= \mathcal I_p ((r,0,0\dots),r),
\ee
where the last inequality follows
from the definition \eqref{calI} and the fact that $\Lambda$ is even in its first two arguments.

The proof for $\mathbb B_p^n$ is nearly identical, so we omit it.
\end{proof}

\section{Preliminary Lemmas: Upper Bound}\label{s:preliminaryproof}
\subsection{Topological properties of \texorpdfstring{$\mathcal X$}{X}}
\label{s:61}
\begin{proof}[Proof of \Cref{l:topologicalpropertyofX}]
Due to the ordering of the first component in $\mathcal X$, each $\left(\mathbf w,r\right)\in\mathcal X$ satisfies
\begin{align}\label{eqn:orderedproperty}
\begin{split}
    |w_{m}|\leq m^{-1/2}\|\mathbf w\|_2\leq m^{-1/2}r.
\end{split}
\end{align}

\emph{Proof of Claim (1).}
Suppose $(\mathbf w_n,r_n)\rightarrow(\mathbf w,r)$ in $\mathcal X$. Then, by the definition of the metric in \eqref{linfinitydistance}, $r_n \rightarrow r$, and $\mathbf w_n\rightarrow\mathbf w$ in $\ell_\infty$. The latter implies that for any coordinate projection map $p_i:\mathbf y\mapsto y_i$, we have $p_i(\mathbf w_n-\mathbf w)\rightarrow0$ as $n\rightarrow\infty$. Thus $(\mathbf w_n,r_n)\rightarrow(\mathbf w,r)$ in $\mathbb R^{\mathbb N}\times\mathbb R_+$ when $\mathbb R^{\mathbb N}$ is equipped with the product topology.

On the other hand, suppose $(\mathbf w_n,r_n)\rightarrow(\mathbf w,r)$ in $\mathbb R^{\mathbb N}\times\mathbb R_+$ when $\mathbb R^{\mathbb N}$ is equipped with the product topology.
To show that $(\w_n,r_n) \rightarrow (\w, r)$ in $\mathcal X$, by \eqref{linfinitydistance} it suffices to show that $\| \w_n - \w\|_\infty \rightarrow 0$ as $n\rightarrow \infty$.
Pick any $\varepsilon>0$. Then there exists $N\in\mathbb N$ such that $\left\|(\mathbf w_n)_{\head N}-\mathbf w_{\head N}\right\|_{\infty}<\varepsilon$ and $N^{-1/2}<\varepsilon/(2r+1)$, and for all $n>N$, $r_n<r+1$. Let $n>N$ be arbitrary. By \eqref{eqn:orderedproperty}, we have
\begin{align*}
\begin{split}
    \left\|\mathbf w_n-\mathbf w\right\|_{\infty}&=\max\left(\left\|(\mathbf w_n)_{\head N}-\mathbf w_{\head N}\right\|_{\infty},\left\|(\mathbf w_n)_{\tail N}-\mathbf w_{\tail N}\right\|_{\infty}\right)\\
    &\leq\max\left(\left\|(\mathbf w_n)_{\head N}-\mathbf w_{\head N}\right\|_{\infty},N^{-1/2}r_n+N^{-1/2}r\right)\\& \leq\varepsilon.
\end{split}
\end{align*}
Sending first $n\rightarrow \infty$, and then $\eps \rightarrow 0$, we obtain the desired conclusion.

\emph{Proof of Claim (2).} Suppose $(\mathbf w_n,r_n)\rightarrow(\mathbf w,r)$ in $\mathcal X$. Pick any $\mathbf u\in\ell^2$. 
It suffices to show that $\left\langle \mathbf u,\mathbf w_n-\mathbf w\right\rangle\rightarrow0$. Without loss of generality, assume $\mathbf u\neq0$.
Fix $\eps >0$.
Since $\mathbf u\in\ell^2$, there exists $m\in\mathbb N$ such that $\|\mathbf u_{\tail m}\|_2<\varepsilon/(2r+1)$. 
Let $N = N(\eps, m)\in\mathbb N$ be large enough so that $\|\mathbf w_n-\mathbf w\|_{\infty}<\varepsilon m^{-1/2}/\|\mathbf u\|_2$ and for all $n>N$, $r_n<r+1$. Then applying the Cauchy--Schwarz and triangle inequalities, and 
invoking the identities $\| w_n \| =r$ and $\| w \|_2 =r$ (which follow from the definition of $\mathcal X$ in  \eqref{Xdefinition}), we conclude that for any $n>N$,
\begin{align*}
\begin{split}
    \left\langle \mathbf u,\mathbf w_n-\mathbf w\right\rangle&=\left\langle \mathbf u_{\head m},\left(\mathbf w_n-\mathbf w\right)_{\head m}\right\rangle+\left\langle \mathbf u_{\tail m},\left(\mathbf w_n-\mathbf w\right)_{\tail m}\right\rangle\\
    &\leq \|\mathbf u\|_2\|\left(\mathbf w_n-\mathbf w\right)_{\head m}\|_2+\|\mathbf u_{\tail m}\|_2\|\left(\mathbf w_n-\mathbf w\right)\|_2\\
    &\leq \|\mathbf u\|_2\|\left(\mathbf w_n-\mathbf w\right)_{\head m}\|_{\infty}m^{1/2}+\|\mathbf u_{\tail m}\|_2\left(r_n+r\right)\\
    &\leq2\varepsilon.
\end{split}
\end{align*}
Since $\varepsilon$ is arbitrary, this shows
$\langle \mathbf u,\mathbf w_n-\mathbf w\rangle = 0$, as desired.

On the other hand, suppose $(\mathbf w_n,r_n)\rightarrow(\mathbf w,r)$ in $\mathbb R^{\mathbb N}\times\mathbb R_+$ when $\mathbb R^{\mathbb N}$ is equipped with the weak-$\ell^2$ topology. Pick any $\varepsilon>0$. Choose $m\in\mathbb N$ large enough such that $m^{-1/2}<\varepsilon/(2r+1)$. Choose $N\in\mathbb N$ large enough such that for all $n>N$, $\left\|(\mathbf w_n)_{\head m}-(\mathbf w)_{\head m}\right\|_2<\varepsilon$ and $r_n<r+1$. 
Together with \eqref{eqn:orderedproperty}, this implies that for any $n>N$,
\begin{align*}
\begin{split}
    \left\|\mathbf w_n-\mathbf w\right\|_{\infty}&=\max\left(\left\|\left(\mathbf w_n\right)_{\head m}-\mathbf w_{\head m}\right\|_\infty,\left\|\left(\mathbf w_n\right)_{\tail m}-\mathbf w_{\tail m}\right\|_{\infty}\right)\\
    &\leq \max\left(\left\|\left(\mathbf w_n\right)_{\head m}-\mathbf w_{\head m}\right\|_2,m^{-1/2}\left(r_n+r\right)\right)\\
    &<2\varepsilon.
\end{split}
\end{align*}
Since $\varepsilon$ is arbitrary, by \eqref{linfinitydistance} this implies $(\mathbf w_n,r_n)\rightarrow(\mathbf w,r)$ in $\mathcal X$.

\emph{Proof of Claim (3).} By definition, for all $\left(\mathbf w,r\right)\in\X$, $|w_i|\leq \|\mathbf w\|_2\leq r$. Therefore, $\left\{\left(\w,r\right)\in\X:r\leq A\right\}$ is a closed subset of $\prod_\sequence{i} [-A,A]\times [0,A]$. By item (1), the topology of $\X$ is equivalent to $\mathbb R^{\mathbb N}\times\mathbb R_+$ where $\mathbb R^{\mathbb N}$ is equipped with the product topology, $\prod_\sequence{i} [-A,A]\times [0,A]$ is compact by Tychonoff's Theorem. Thus, as a closed subset of $\prod_\sequence{i} [-A,A]\times [0,A]$, the set $\left\{\left(\w,r\right)\in\X:r\leq A\right\}$ is compact.
\end{proof}
\subsection{Rate Function}\label{s:62}
\begin{proof}[Proof of \Cref{l:ratefunction}]
\textit{Step 1.}
Fix $m \in \mathbb{N}$, and recall from \Cref{s:lowerpreliminary} that $\Lambda^*_m$ denotes the Legendre transform of the function $\Lambda_m$ defined in \eqref{lambdam}.
We first show that the following inequality holds for all $(\mathbf w,r,s)\in\mathcal X$:
\begin{align}\label{eqn:good-rate-fcn-step1}
\begin{split}
    I(\mathbf w,r,s)\geq \Lambda_m^*\left(\mathbf w_{\head},\sqrt{r^2-\|\mathbf w_\head\|_2^2},s\right).
\end{split}
\end{align}
If $\mathbf w_\tail=\mathbf 0$, 
then $\Lambda(\w , r ,s) = \Lambda_m(\w_{\head}, r , s)$, and for any $\u \in \ell^2$, $\langle \u , \w \rangle = \beta$ if and only if $\langle \u_\head , \w_\head \rangle = \beta$. It follows that $\Lambda^*(\w,b,c) = \Lambda^*_m(\w_\head , b, c)$ for all $b,c \in \R_+$. Hence, the definition of \eqref{ratefunction} of $I$ shows that the inequality in \eqref{eqn:good-rate-fcn-step1} becomes an equality if $\w_\tail = \mathbf 0$. 

We henceforth suppose that $\| w_\tail \|_2 >0$, which then implies $r>\|\mathbf w_{\head}\|_2$.
Pick any $\hatu\in\mathbb R^m$ and $b\geq 0$. Let $\mathbf u\coloneqq \left(\hatu,b\mathbf w_{\tail}/\sqrt{r^2-\|\mathbf w_{\head}\|_2^2}\right)$. Then by the definition of $\Lambda^*$ in \eqref{lambdastar}, we have for any $c\in \R$ that
\begin{align*}
\begin{split}
    I\left(\mathbf w,r,s\right)&=\Lambda^*\left(\mathbf w,\sqrt{r^2-\|\mathbf w\|_2^2},s\right)\\
    &\geq\langle \mathbf u,\mathbf w\rangle+\left(\frac{b\sqrt{r^2-\|\mathbf w\|_2^2}}{\sqrt{r^2-\|\mathbf w_\head\|_2^2}}\right)\sqrt{r^2-\|\mathbf w\|_2^2}+cs-\Lambda\left(\mathbf u,\frac{b\sqrt{r^2-\|\mathbf w\|_2^2}}{\sqrt{r^2-\|\mathbf w_\head\|_2^2}},c\right)\\
    &=\left\langle \hatu,\bm w_\head\right\rangle+\frac{b\|\bm w_\tail\|_2^2}{\sqrt{r^2-\|\bm w_\head\|_2^2}}+\left(\frac{b\sqrt{r^2-\|\mathbf w\|_2^2}}{\sqrt{r^2-\|\mathbf w_\head\|_2^2}}\right)\sqrt{r^2-\|\mathbf w\|_2^2}+cs\\&\quad -\Lambda\left(\mathbf u,\frac{b\sqrt{r^2-\|\mathbf w\|_2^2}}{\sqrt{r^2-\|\mathbf w_\head\|_2^2}},c\right)\\
    &=\left\langle \hatu,\bm w_\head\right\rangle+\frac{b(\|\bm w_\tail\|_2^2+r^2-\|\bm w\|_2^2)}{\sqrt{r^2-\|\bm w_\head\|_2^2}}+cs-\Lambda\left(\mathbf u,\frac{b\sqrt{r^2-\|\mathbf w\|_2^2}}{\sqrt{r^2-\|\mathbf w_\head\|_2^2}},c\right)\\
    &=\langle \hatu,\mathbf w_\head\rangle+b\sqrt{r^2-\|\mathbf w_\head\|_2^2}+cs -\Lambda_m\left(\hatu,b,c\right),
\end{split}
\end{align*}
where we used $\Lambda(\mathbf u,b',c)=\Lambda_m(\mathbf u_{\head},\sqrt{(b^{\prime})^2+\|\mathbf u_{\tail}\|_2^2},c)$ in the last step.
Taking the supremum of the right-hand side over $\w \in \R^m$ and $b \in \R_+$, we obtain \eqref{eqn:good-rate-fcn-step1}.

\textit{Step 2.} We next claim that $I$ is lower semicontinuous.

First note that $\Lambda_m^*$ is lower semicontinuous, since it is the supremum of a collection of continuous functions. By claim (2) of \Cref{l:topologicalpropertyofX}, the topology on $\X$ is equivalent to the topology on $\mathbb R^{\mathbb N}\times\mathbb R_+$, where $\mathbb R^{\mathbb N}$ is equipped with the weak-$\ell^2$ topology. Therefore $\Lambda^*$, which is equal to  the supremum of a collection of continuous functions,  is lower semicontinuous.
Let $(\mathbf w^{(n)},r^{(n)},s^{(n)})_{n \in \N}$ be a sequence converging to $(\mathbf w,r,s)$ in $\mathcal X\times\mathbb R_+$. Applying 
\eqref{eqn:good-rate-fcn-step1}, the lower semicontinuity of $\Lambda_m^*$, and \eqref{lambdam},
we obtain
\begin{align*}
\begin{split}
    \liminf_{n\rightarrow\infty}I\left(\mathbf w^{(n)},r^{(n)},s^{(n)}\right)&\geq \liminf_{n\rightarrow\infty}\Lambda_m^*\left(\mathbf w^{(n)}_{\head},\sqrt{\left(r^{(n)}\right)^2-\left\|\mathbf w^{(n)}_{\head}\right\|_2^2},s^{(n)}\right)\\
    &\geq \Lambda_m^*\left(\mathbf w_{\head},\sqrt{r^2-\|\mathbf w_\head\|_2^2},s\right)\\
    &=\Lambda^*\left(\left(\mathbf w_{\head},\mathbf 0\right),\sqrt{r^2-\|\mathbf w_\head\|_2^2},s\right).
\end{split}
\end{align*}
Recalling that $\w_{\head}$ is the abbreviation for $\w_{\head m}$, passing to the $m\rightarrow\infty$ limit in the last display,  and using the lower semicontinuity of $\Lambda^*$, the lower semicontinuity of $I$ follows.

\textit{Step 3.} We claim $\{I\leq \alpha\}$ is compact for any $\alpha\geq 0$.

Since $I$ is lower semicontinuous, by Step 2, $\{I\leq \alpha\}$ is closed. We will show that
\begin{equation}\label{containment}
\{ (\w, r,s ) \in \mathcal X \times \R :
I(\w , r,s) \le \alpha \}
\subset \{ (\w ,r ) \in \X : r \le A \} \times [0 , B].
\end{equation}
for some $A,B>0$. 
This suffices to prove the claim, because the latter set is compact by \Cref{l:topologicalpropertyofX}, and closed subsets of compact sets are compact.
To show \eqref{containment}, we argue by contradiction. Fix $\alpha \ge 0$, and suppose that there exists a sequence
\begin{equation}\label{starassumption}
(\mathbf w^{(n)},r^{(n)},s^{(n)})\in\{I\leq \alpha\}, \qquad n \in \mathbb{N}\end{equation}
such that $r^{(n)}\rightarrow\infty$ or $s^{(n)}\rightarrow\infty$. Fix $0<c<T$, where $\mathbb R\times(-\infty,T)$ is the domain of $\Lambda$ from \Cref{assume3}.
By \Cref{assume4} on $\overline{\Lambda}$, and the definition of $\Lambda_m$ in \eqref{lambdam} and \eqref{d:g0}, it follows that there exists an absolute constnat $C \in (0, \infty)$ such that for all $\hatu=(u_1,\ldots,u_m) \in \R^m$ and $b \ge 0$,
\begin{equation}\label{62}
    \Lambda_m\left(\hatu,b,0\right)=\mathbb E\left[\overline{\Lambda}\left(\sum_{i=1}^mu_ig_i+bg_0,0\right)\right]\leq C\left(1+\|(\hatu,b)\|_2^2\right).
\end{equation}

Suppose first that $r^{(n)}\rightarrow\infty$. By \eqref{eqn:good-rate-fcn-step1} and \eqref{62}, it follows that
\begin{align*}
\begin{split}
    I\left(\mathbf w^{(n)},r^{(n)},s^{(n)}\right)&\geq \Lambda_m^*\left(\mathbf w^{(n)}_{\head},\sqrt{\left(r^{(n)}\right)^2-\|\mathbf w^{(n)}_{\head}\|_2^2},s^{(n)}\right)\\
    &=\sup_{(\hatu,b,c) \in \X \times \R_+}\left\{\langle \hatu,\mathbf w^{(n)}_{\head}\rangle+b\sqrt{\left(r^{(n)}\right)^2-\|\mathbf w^{(n)}_{\head}\|_2^2}+cs^{(n)}-\Lambda_m\left(\hatu,b,c\right)\right\}\\
    &\geq\sup_{(\hatu,b) \in \X}\left\{\langle \hatu,\mathbf w^{(n)}_{\head}\rangle+b\sqrt{\left(r^{(n)}\right)^2-\|\mathbf w^{(n)}_{\head}\|_2^2}-C\left(1+\|(\hatu,b)\|_2^2\right)\right\}.
\end{split}
\end{align*}
Using the Cauchy--Schwarz inequality, we note that
\begin{align*}
 \begin{split}
     \langle \hatu,\mathbf w_{\head}^{(n)}\rangle+b\sqrt{\left(r^{(n)}\right)^2-\|\mathbf w^{(n)}_{\head}\|_2^2}\leq r^{(n)}\|(\hatu,b)\|_2,
 \end{split}
 \end{align*}
and equality holds when $(\hatu,b)=\alpha (\mathbf w_{\head}^{(n)},\sqrt{\left(r^{(n)}\right)^2-\|\mathbf w^{(n)}_{\head}\|_2^2})$ for some $\alpha\in\mathbb R$. Therefore
\begin{align*}
    &\sup_{(\hatu,b) \in \X}\left\{ \langle \hatu,\mathbf w^{(n)}_{\head}\rangle+b\sqrt{\left(r^{(n)}\right)^2-\|\mathbf w^{(n)}_{\head}\|_2^2}-C\left(1+\|(\hatu,b)\|_2^2\right)\right\}\\
    &=\sup_{(\hatu,b)\in \X}\left\{ r^{(n)}\|(\hatu,b)\|_2-C\|(\hatu,b)\|_2^2-C \right\}=\frac{\left(r^{(n)}\right)^2}{4C}-C,
\end{align*}
where to get the last equality, we compute the supremum directly as a function of $\|(\hatu,b)\|_2$.
The last three displays together show that
\[
I(\mathbf w^{(n)},r^{(n)},s^{(n)}) \ge \frac{(r^{(n)})^2}{4C} - C,
\]
and this lower bound goes to infinity as $n\rightarrow \infty$, due to the assumption that $r^{(n)} \rightarrow \infty$. However, this contradicts the assumption \eqref{starassumption}.

Next, suppose $s^{(n)}\rightarrow\infty$. By \eqref{eqn:good-rate-fcn-step1} and \eqref{lambdam},
\begin{align*}
\begin{split}
    I\left(\mathbf w^{(n)},r^{(n)},s^{(n)}\right)\geq \Lambda_m^*\left(\mathbf w^{(n)}_{\head},\sqrt{\left(r^{(n)}\right)^2-\|\mathbf w^{(n)}_{\head}\|_2^2},s^{(n)}\right)
    \ge
    \sup_{t \in \R} \left\{
    t s^{(n)}  - \Lambda(\mathbf 0,0, t)
    \right\}.
\end{split}
\end{align*}
The supremum on the right side of the previous display is infinite, by \eqref{d:g0}, \eqref{asprev}, and \eqref{assume4}.
Again, this contradicts the assumption that $(\mathbf w^{(n)},r^{(n)},s^{(n)})\in\{I\leq \alpha\}$. The proof is complete.
\end{proof}
\subsection{Exponential Tightness}\label{s:63}
Our proof will appeal to the Hanson--Wright concentration inequality for quadratic forms.
We recall the following definition from \cite[Section 2.5]{Ver18}.
\begin{defn}\label{d:subGaussian}
A random variable $X$ is said to be sub-Gaussian
if there exists a constant $C>0$ such that
\be
\P( |X| \ge x ) \le 2 \exp\left(
- \frac{x^2}{C^2}
\right).
\ee
In this case, the sub-Gaussian norm of $X$ is defined as
\begin{equation}
\| X \|_{\psi_2} \coloneqq \inf \Big\{ t > 0 : \E \big[ \exp (X^2/t^2) \big]  \le 2 \Big\}.
\end{equation}
Moreover, a random vector $X\in\R^n$ is called sub-Gaussian if all one dimensional marginals $\left\langle X,\bm x\right\rangle$, $\x \in \R^n$, are sub-Gaussian random variables. The sub-Gaussian norm of $X$ is defined as
\begin{align}\label{d:subGaussian-vector}
\begin{split}
    \left\|X\right\|_{\psi_2}\coloneqq \sup_{\bm x\in\mathbb S^{n-1}}\left\|\left\langle X,\bm x\right\rangle\right\|_{\psi_2}.
\end{split}
\end{align}
\end{defn}
\begin{thm}[{\cite[Theorem 6.2.1]{Ver18}}]\label{t:hansonwright}
There exists a constant $\gamma_0 >0$ such that the following holds.
Let $X = (X_1,\dots, X_n)$ be a random vector with i.i.d., mean zero, sub-Gaussian entries. Let $A$ be an $n \times n$ matrix. Then for every $t\ge 0$,
\begin{equation}
\P\left( \left| X^\trans A X - \E [ X^\trans A X ]  \right| \ge t   \right) \le 2 \exp
\left[ - \gamma_0 \min\left(\frac{t^2}{R^4 \|A \|_F^2}, \frac{t}{R^2 \| A \|_2}  \right)  \right],
\end{equation}
where $R\coloneqq\| X \|_{\psi_2}$ is the sub-Gaussian norm of $X_1$ defined in \Cref{d:subGaussian}, and $\| A \|_2$ denotes the Frobenius norm of $A$.
\end{thm}
\begin{proof}[Proof of \Cref{l:tightness}]
For $n\in \N$, we set
\begin{equation}\label{W}
A^{(n)} = \frac{1}{n} \a^{(n)} (\a^{(n)})^\trans, \quad
Q^{(n)} = \frac{1}{n} (X^{(n)})^\trans \a^{(n)} (\a^{(n)})^\trans X^{(n)} =
\frac{1}{n} \sum_{i,j=1}^n  A^{(n)}_{ij}  X_{i} X_j.
\end{equation}
For the remainder of this proof, we suppress the dependence on $n$ in the notation by omitting all superscripts.
Observe that
\begin{equation}\label{expectationQ}
\E\left[ X^\trans A X  \right] = \frac{1}{n} \sum_{i=1}^n A_{ii} \E[X_i^2] = \frac{1}{n} \sum_{\ell}^{k_n} \sum_{i=1}^n
a_{i\ell}^2  = \frac{k_n}{n}.
\end{equation}
We  write the matrix $A$ as a sum of rank one matrices, $A = n^{-1} \sum_{\ell=1}^{k_n} Y_\ell Y^\trans_\ell$, where
$Y_\ell
=
[a_{1\ell}, \dots, a_{n\ell}]^\trans$.
Since $\a \in \mathbb V_{n, k_n}$, the vectors $\{ Y_{\ell} \}_{1 \le \ell \le k_n}$ are orthonormal. This implies
\begin{align}
\| A \|_F^2 &= \tr A A^\trans  \notag \\
&=  n^{-2} \tr \left( \sum_{\ell=1}^{k_n} Y_\ell Y^\trans_\ell\right)
\left( \sum_{\ell=1}^{k_n} Y_\ell Y^\trans_\ell \right)^\trans\notag \\
&= n^{-2} \tr  \sum_{\ell=1}^{k_n} Y_\ell Y^\trans_\ell\notag \\
&= n^{-2} \tr A\notag \\ &= n^{-2} k_n.\label{norm1}
\end{align}
Further, we also have
\begin{align}
\| A \|^2_2 &= \max_{\|v \| =1 } \| A v \|^2_2 \notag \\
& = \max_{\|v \| =1 } v^\trans A^\trans A v  \notag \\
& = \max_{\|v \| =1 }  n^{-2} \sum_{\ell=1}^{k_n}
 v^\trans Y_\ell   Y^\trans_\ell v\notag \\
&=  \max_{\|v \| =1 }  n^{-2} \langle Y_\ell, v \rangle^2.\label{norm2}
\end{align}
The last quantity is equal to $n^{-2}$ because the $Y_\ell$ are unit vectors. 
Also, note that by \Cref{d:subGaussian}, $\| X \|_{\psi_2} = C_2$, where $C_2$ is the constant in Assumption 2.
By applying \Cref{t:hansonwright} to the quantity $Q$ from \eqref{W} (and hence with $A= A^{(n)}$ and $R=C_2$), and using \eqref{expectationQ}, \eqref{norm1}, and \eqref{norm2}, we find that for any $\a \in \V_{n,k_n}$,
\begin{equation}
\P\left( \left| \|n^{-1/2} \a^\trans X^{(n)}\|_2^2 - \frac{k_n}{n} \right| \ge t   \right) \le 2 \exp
\left[ -  \gamma_0 \min\left( k_n^{-1} C_2^{-4} (nt)^2 , C_2^{-2} nt  \right)  \right],
\end{equation}
where $\gamma_0 >0$ is the constant from \Cref{t:hansonwright}.
This implies the conclusion after noting that 
since $k_n/n \le 1$ for $n\in \N$, for $t\ge C^2$, we have
\be
\min\left( k_n^{-1} (C_2^{-2} nt)^2 , C_2^{-2} nt  \right) = C_2^{-2} nt.
\ee
\end{proof}

\begin{proof}[Proof of \Cref{l:tightness2}]
Pick any $M>0$ and fix $\lambda>0$ such that $\overline{\Lambda}(0,\lambda)<\infty$. We note that such a $\lambda$ exists by \Cref{assume3}. Fix $\alpha>0$ satisfying $\alpha>(\overline{\Lambda}(0,\lambda)+M\big)/\lambda$. By Markov's inequality, we have
\begin{align*}
\begin{split}
    \P\left(\frac{1}{n}\sum_{i=1}^n\eta\left(X_i\right)\geq \alpha\right)\leq e^{-n\lambda\alpha}\left(\E\Big[\exp\big(\lambda\eta(X_1)\big)\Big]\right)^n.
\end{split}
\end{align*}
This implies that
\begin{align*}
\begin{split}
    \frac{1}{n}\log \P\left(\frac{1}{n}\sum_{i=1}^n\eta\left(X_i\right)\geq \alpha\right)\leq-\lambda\alpha+\E\Big[\exp\big(\lambda\eta(X_1)\big)\Big]<-M,
\end{split}
\end{align*}
which proves the exponential tightness of $(\frac{1}{n}\sum_{i=1}^n\eta(X_i))_{n\in\N}$.
\end{proof}

\subsection{Preliminary Lemmas for Gaussian Approximation}\label{s:gauss1}
Let $\mu$ be the standard Gaussian measure on $\R$, corresponding to a Gaussian random variable with mean zero and variance one.
 Fix $k, n \in \N$ with $k \le n$, and   $\a = \a^{(n)} \in \V_{n,k}$. Let $\a_i$ be the $i$-th row of $\a$.
For any choice of $\u \in \R^k$, let
$\mu_{n,\bm u}$ be the measure on $\R$ given by
\begin{equation}\label{d:mun}
\mu_n \equiv \mu_{n,\bm u}\defeq \frac{1}{n} \sum_{i=1}^n \delta_{\langle \u, \sqrt{n}\a_i\rangle}.
\end{equation}
\begin{lem}\label{Gaussianmoments}
Fix $k,n \in \N$ with $k \le n$, and $\u \in \R^k$. Let $\a$
be a random element of $\V_{n,k}$  distributed according to the Haar measure $\sigma_n$.
Then for every $m \in \N$, the quantity $\mu_n$ from \eqref{d:mun} satisfies
\begin{equation}
\E \left[\int x^m \, d \mu_{n,\bm u} \right]= \| \u \|^m_2 \left(\int x^m \, d \mu \right) \left( 1 + T_{m,n} \right),
\qquad | T_{m,n} | \le \frac{C_m}{n},
\end{equation}
for some $T_{m,n} \in \R$ and $C_m > 0$, where $C_m$ depends only on $m$ (and not $k$ or $n$).

\end{lem}
\begin{proof}
We just treat the even $m$ case. The odd $m$ case is clear by symmetry, as the moments on the left and right side of the equality both vanish.
By exchangeability 
of the rows of $\a$,
\begin{equation}
\E\left[\int x^m \, d \mu_{n,\bm u}\right] =
\frac{1}{n} \E \left[ \sum_{i=1}^n \left(\sum_{j=1}^k u_j (\sqrt{n}a_{ij}^{(n)} ) \right)^m \right]  = \E \left[  \left(\sum_{j=1}^k u_j (\sqrt{n}a_{1j} ) \right)^m \right] .
\end{equation}
We recognize this quantity as
\be
\E\left[ \langle  \u, \sqrt{n} \a_1 \rangle^m \right]=
\| \u \|^m_2 \cdot \E\left[ (\sqrt{n} \a_{11})^m \right].
\ee
where the equality is justified by the orthogonal invariance of $\mu_n$.
By Wick's theorem (\Cref{thm:wicktheorem}) and \Cref{momentexpectation},
\be
\E\left[ (\sqrt{n} \a_{11})^m \right]
=\left(\int x^m \, d \mu \right) \left( 1 + T_{m,n}\right), \qquad | T_{m,n} | \le \frac{C_m}{n}.
\ee
This completes the proof.

\end{proof}

The next lemma is a minor adaptation of a modified version of Gromov's concentration inequality for Haar measure on the special orthogonal group $SO_n\defeq\{\an\in O_n: \det(\an)=1\}$ (presented in \cite[Corollary~4.4.28]{anderson2010introduction}), tailored to the Haar measure $\sigma_n = \sigma_{n,k}$ on the Stiefel manifold $\V_{n,k}$.

\begin{lem}\label{t:gromov}
Fix $k,n\in\N$ with $k<n$. Let $f\colon \V_{n,k}\rightarrow \R$ be a Lipshitz function with Lipschitz constant $L$ in the Hilbert--Schmidt norm, meaning
\begin{equation}
| f(\an) - f(\bn) | \le L \| \an-\bn\|_{\mathrm{HS}}
\end{equation}
for all $\an,\bn \in \V_{n,k}$. 
Then there is a constant $C>0$ such that for all $\delta\geq0$,
\begin{equation}
\sigma_n\left( | f(\an) - {\E}_n[f(\an)] \ge  \delta   \right) \le 2 \exp\left( - \frac{C\delta^2n}{L^2}\right).
\end{equation}
\end{lem}
\begin{proof} We suppress the superscript $(n)$ in the proof for brevity. 
Fix $k\in\N$, and let $\pi=\pi_k:O_n\rightarrow \V_{n,k}$ be the canonical projection map, i.e.
$\pi: \a=(a_{ij})_{1\leq i,j\leq n}\mapsto \left(a_{ij}\right)_{1\leq i\leq n,1\leq j\leq k}$, for all $\a\in O_n$, let $\bar\sigma_n$ be the Haar measure on $O_n$ and let $\bar\E_n$ be the corresponding expectation. Then $\sigma_n=\sigma_{n,k}$ is the pushforward of $\bar\sigma_n$ under $\pi$. Finally, set $\bar f=f\circ\pi$ and let $H_n\subset O_n$ be the cyclic group generated by 
$$\bm h=\left(\begin{array}{lllll}1 & & & &  \\
& 1 & & &  \\
& & \ddots & &  \\& & & 1 & \\& & & & -1\end{array}\right),$$
that is, $H_n$ consists of the identity matrix and $\bm h$. Observe that $O_n=\{\bm b_1\bm b_2:\,\bm b_1\in SO_n,\bm b_2\in H_n\}$. Fix $\bm c\in O_n$. Then there exists $\bm c_1\in SO_n$ and $ \bm c_2\in H_n$ such that $\bm c=\bm c_1\bm c_2$. Let $\a$ be uniformly distributed on $O_n$ with respect to Haar measure. Using right-invariance of Haar measure in the second equality and combining the definition of $\bar f$, the fact that $k<n$ and $\bm c_2\in H_n$ in the third equality below, we have
\begin{align*}
\begin{split}
    \bar\E_n\left[\bar f(\a \bm c)\right]=\bar\E_n\left[\bar f\left(\a\bm c_1\bm c_2\right)\right]=\bar \E_n\left[\bar f\left(\a\bm c_2\right)\right]=\bar\E_n\left[\bar f\left(\a\right)\right].
\end{split}
\end{align*}
This implies that the function $O_n\ni \bm c\mapsto \bar\E_n[\bar f(\bm a\bm c)]$ is a constant over $O_n$ equal to $\bar\E_n[\bar f(\bm a)]$. This observation, together with \cite[Corollary~4.4.28]{anderson2010introduction} and the definitions of $\bar\sigma_n$ and $\bar f$, completes the proof.
\end{proof}

\subsection{Proof of Gaussian Approximation}\label{s:65}
\begin{proof}[Proof of \Cref{l:lln}.]
Fix $c \in (-\infty, T)$ and a sequence $\{d_n\}_{n \in \N}$ as in the lemma statement. 
To rewrite $F_n$ in \eqref{eqn:quenched-log-moment} more succinctly, set $\Lambda_1(t) = \overline \Lambda(t, c)$ with $\overline{\Lambda}$ as in \eqref{asprev}, and define the maps $f^{(n,j)}:\mathbb V_{n,k_n}\rightarrow \mathbb R$
\begin{align}\label{def:fnj}
\begin{split}
    f^{(n,j)}(\an)\defeq \frac{1}{n}\sum_{i=1}^n\Lambda_1\left(\left\langle \bm v^{(n,j)},\sqrt n\an_i\right\rangle\right) ,\quad j=1,2,\ldots,d_n.
\end{split}
\end{align}
A direct calculation shows that, for $i=1,\ldots,n$ and $l=1,\ldots,k_n$,
\begin{align*}
\begin{split}
    \partial_{il}f^{(n,j)}(\an)=\frac{1}{\sqrt n}\Lambda_1'\left(\left\langle \bm v^{(n,j)},\sqrt n\an_i\right\rangle\right) v^{(n,j)}_l,
\end{split}
\end{align*}
where $\partial_{il}f^{(n,j)}(\an)$ represents the derivative of $f^{(n,j)}$ with respect to $\an_{il}$. 
Hence, by \Cref{assume4}, there exists $\overline{C}\equiv \overline{C}(c)< \infty$ such that
\begin{align}\label{eqn:gradient-norm}
\begin{split}
    \left\|\nabla f^{(n,j)}(\an)\right\|_2^2=&\sum_{i=1}^n\sum_{l=1}^{k_n}\left(\partial_{il}f^{(n,j)}(\an)\right)^2\\
    =&\frac{1}{n}\sum_{i=1}^n\left(\Lambda_1'\left(\left\langle \bm v^{(n,j)},\sqrt n\an_i\right\rangle\right)\right)^2\sum_{l=1}^{k_n}\left(v^{(n,j)}_l\right)^2\\
    \leq&\frac{\left\|\bm v^{(n,j)}\right\|_2^2}{n}\sum_{i=1}^n C\left(1+\left\langle \bm v^{(n,j)},\sqrt n\an_i\right\rangle^2\right)\\
    =& \overline{C}\|\bm v^{(n,j)}\|_2^2\left(1+\left\|\bm v^{(n,j)}\right\|_2^2\right)\\
    =& \overline{C} D^2\left(1+D^2\right).
\end{split}
\end{align}
\Cref{eqn:gradient-norm} implies that $f^{(n,j)}$ is Lipshitz continuous in the Hilbert--Schmidt topology with Lipshitz constant $\sqrt{CD^2\left(1+D^2\right)}$. By \Cref{t:gromov} and union bound, it follows that there exists a constant $C'>0$ such that  for any $\epsilon>0$,
\begin{align}\label{gaussian-approx-01}
\begin{split}
    \sigma_n\left(\max_{1\leq j\leq d_n}\left|f^{(n,j)}(\an)-\E_n\left[f^{(n,j)}(\an)\right]\right|>\epsilon\right)\leq 2 d_n\exp\left(-\frac{C'\epsilon^2n}{D^2(1+D^2)}\right).
\end{split}
\end{align}
Recalling $d_n=\exp(o(n))$, the inequality in \eqref{gaussian-approx-01} and the Borel--Cantelli lemma together imply that for $\sigma$-a.e. $\bm a=\left(\bm a^{(1)},\bm a^{(2)},\ldots\right)\in\V$,
\begin{align}\label{gaussian-approx-02}
\begin{split}
    \lim_{n\rightarrow\infty}\max_{1\leq j\leq d_n}\left|f^{(n,j)}(\an)-\E_n\left[f^{(n,j)}(\an)\right]\right|= 0.
\end{split}
\end{align}

Next, consider the random measures
\begin{equation*}
\mu^{(n,j)} \coloneqq \frac{1}{n} \sum_{i=1}^n \delta_{\langle \v^{(n,j)}, \sqrt{n}\a_i^{(n)}\rangle},\qquad j=1,\dots, d_n. 
\end{equation*}
By the definition of $f^{(n,j)}$ in \eqref{def:fnj}, we have
\begin{align}\label{rewrite-fnj}
\begin{split}
    f^{(n,j)}\left(\an\right)=\int_\R \Lambda_1(t)\, d\mu^{(n,j)}(t).
\end{split}
\end{align}
By \Cref{Gaussianmoments}, the sequence 
\begin{equation*}
\left\{\E_1\left[\mu^{(1,1)}\right],\ldots,\E_1\left[\mu^{(1,d_1)}\right],\E_2\left[\mu^{(2,1)}\right],\ldots,\E_2\left[\mu^{(2,d_2)}\right],\ldots,\E_n\left[\mu^{(n,1)}\right],\ldots,\E_n\left[\mu^{(n,d_n)}\right],\ldots\right\}
\end{equation*}
converges in the sense of moments to a Gaussian measure, with mean zero and variance $D^2$. Combined with \cite[Lemma~B.1]{BS10} and \cite[Theorem~1.7]{ollivier2014optimal}, this implies that for any $f:\R\rightarrow\R$ with at most polynomial growth,
\begin{equation}\label{wass-converge}
\lim_{n\rightarrow\infty}\max_{1\leq j\leq d_n}\left|\E\left[\int fd\mu^{(n,j)}\right]-\E\left[f(Dg)\right]\right|=0,
\end{equation}
where $g$ is a Gaussian random variable with zero mean and unit variance.
By \eqref{rewrite-fnj}, \eqref{wass-converge} and the fact that $\Lambda_1$ can be bounded by a quadratic polynomial due to \Cref{assume4}, it follows that
\begin{equation}\label{gaussian-approx-03}
\lim_{n\rightarrow\infty}\max_{1\leq j\leq d_n}\left|\E\left[f^{(n,j)}\left(\an\right)\right]-\E\left[\Lambda_1(Dg)\right]\right|=0.
\end{equation}

To complete the proof, combine \eqref{gaussian-approx-02} and \eqref{gaussian-approx-03}, and recall the definition of $F_n$ in \eqref{eqn:quenched-log-moment}.
\end{proof}

\section{Preliminary Lemmas: Lower Bound}\label{s:preliminarylowerProof}
\subsection{Notations and Conventions}\label{s:lowerBoundNotations}
We adopt the following notation for this section.
Let $T$ be the constant from \Cref{assume3} and let $m\in\mathbb N$ and $(\hatu,b,c)\in\R^m\times\R\times(-\infty,T)$ be parameters. 
Recall that $\Lambda_m$ was defined in \eqref{lambdam}, and set 
\be\label{gradrel}
(\hatw,t,s)\coloneqq\nabla\Lambda_m(\hatu,b,c).
\ee
Let $\{k_n\}_\sequence{n}$ be an increasing sequence of positive integers such that $\lim_{n\rightarrow \infty} k_n = \infty$. Recall that
$\sigma=\sigma_n$ denotes the Haar measure on $\mathbb V_{n,k_n}$.
Define $m \in \N$ and $\bm b \in \R^M_+$ by
\begin{align}\label{eqn:def-of-M}
\begin{split}
    M= M(m)\=\left\lceil (t+1)^2m\right\rceil,\quad\bm b\=\underbrace{\left(bM^{-1/2},\ldots,bM^{-1/2}\right) }_{M\text{ times}},
\end{split}
\end{align}
and also define a sequence of vectors $\{ \u^{(n)} \}_{n > m+M} \in \R^{k_n}$ as follows:
\begin{align}\label{eqn:def-of-u}
\begin{split}
    \mathbf u^{(n)}:=(\hatu,\mathbf b,\underbrace{0,0,\ldots,0}_{k_n-m-M\text{ times}}).
\end{split}
\end{align}
For $j\in \{1,\ldots,m+M\}$, define\footnote{We sometimes write $\delta(x)$ instead of $\delta_x$ in this subsection for legibility. The difference is only notational.}
\begin{align}\label{eqn:nu-n-j}
\begin{split}
\nun_j\coloneqq \frac{1}{n}\sum_{i=1}^n\delta\big(\sqrt na_{ij}^{(n)},\langle \mathbf u^{(n)},\sqrt n\mathbf a_i^{(n)}\rangle\big).
\end{split}
\end{align}
Finally, define
\begin{align}\label{eqn:lawOfJointGaussians}
\begin{split}
    \nu_j\coloneqq\operatorname{Law}\left(g_{j}, \sum_{l=1}^{m+M} u_{l} g_{l}\right),
\end{split}
\end{align}
where $g_0,\ldots,g_{m+M}$ are independent, mean zero, variance one Gaussian random variables.

\subsection{Proof of \Cref{l:additional-good-properties}}\label{s:top-lem}
We begin with an auxiliary result that allows us to analyze open neighborhoods of a point in $\X \times \R$ by looking at open neighborhoods around the truncated version of the point in $\R^{m} \times \R^2_+$, for sufficiently large $m$.
Recall our notation for a metric ball  given in \eqref{eqn:ball-def}.
\begin{lem}\label{l:product-topology}
Fix an open set $O \subset \mathcal X \times \R_+$ and a point $(\mathbf w,r,s) \in O$.
Then there exists $m_0 \in \N$ and $\rho' >0$ such that the following holds for all $ m \ge m_0$: 
If $(\mathbf w',r',s')\in\mathcal X \times \R_+$ and
\begin{equation}\label{eqn:cond-lem}
    \left(\left(\mathbf w'_{\head m},\sqrt{(r')^2-\|\mathbf w'_{\head m}\|_2^2}\right),s'\right)\in B\left(\left(\mathbf w_{\head m},\sqrt{r^2-\|\mathbf w\|_2^2},s\right),\rho'\right),
\end{equation}
then $(\mathbf w',r',s')\in O$.
\end{lem}
\begin{proof} Fix $(\bm w,r,s)\in O$. By \Cref{l:topologicalpropertyofX}(1),
the topology on $\mathcal X$ is equivalent to the product topology on $\mathbb R^{\mathbb N}\times \mathbb R_+\times\mathbb R_+$. Therefore, using the definition of the product topology, there exists $m_0\in \mathbb N$ and $\delta_1,\delta_2,\delta_3>0$ such that for every $m>m_0$, the following claim holds:
\begin{align}\label{eqn:base-in-O}
\begin{split}
    \text{If }(\tilde{\bm w},\tilde r,\tilde s)\in\mathcal X\times \R_+ \text{ and } (\tilde{\mathbf w}_{\head m},\tilde r,\tilde s)\in B(\mathbf w_{\head m},\delta_1)\times B(r,\delta_2)\times B(s,\delta_3),\\\text{ then }(\mathbf w',r',s')\in O.
\end{split}
\end{align}
By increasing $m_0$ if necessary, we may also suppose that for all $m\ge m_0$, we have 
\begin{equation}\label{eqn:w-small-tail}
    \|\mathbf w_{\tail m}\|_2\leq \frac{\delta_2r}{2}.
\end{equation}
We set $t\defeq\sqrt{r^2-\|\mathbf w\|_2^2}$ and $\rho'\defeq\min\left(\delta_1,\delta_2/4,\delta_3\right)$.

Now, let $(\bm w',r',s')\in\mathcal X\times \mathbb R_+$ be such that \eqref{eqn:cond-lem} holds for some $m \ge m_0$ and our choice of $\rho'$, and let this choice of $m$ be fixed for the remainder of the proof. Then clearly,
\begin{gather}
    \mathbf w'_{\head m}\in B\left(\mathbf w_{\head m},\rho'\right)\subset B\left(\mathbf w_{\head m},\delta_1\right),\label{eqn:first-coord}\\
    s'\in B\left(s,\rho'\right)\subset B\left(s,\delta_3\right),\label{eqn:third-coord}
\end{gather}
and
\begin{align*}
\begin{split}
    \rho'>\left|\sqrt{r'^2-\|\mathbf w'_{\head m }\|_2^2}-\sqrt{r^2-\|\mathbf w\|_2^2}\right|.
\end{split}
\end{align*}
Multiplying both sides by $\sqrt{r'^2-\|\mathbf w'_{\head m}\|_2^2}+\sqrt{r^2-\|\mathbf w\|_2^2}$ we have
\begin{align*}
\begin{split}
\rho'\left(\sqrt{r'^2-\|\mathbf w'_{\head m}\|_2^2}+\sqrt{r^2-\|\mathbf w\|_2^2}\right)>\left|r^2-(r')^2-\|\mathbf w_{\head m}\|_2^2+\|\mathbf w'_{\head m}\|_2^2  - \|\mathbf w_{\tail m}\|_2^2\right|.
\end{split}
\end{align*}
Note because that $(\w, r, s)$ and $(\bm w',r',s')$ lie in $ \X\times\mathbb R_+$, we have  $\| \w \|_2 \le r$ and $\|\bm w'\|_2\leq r'$.
Combining this fact with the triangle inequality,  \eqref{eqn:w-small-tail}, \eqref{eqn:first-coord}, and the previous display, we have 
\begin{align*}
\begin{split}
    |r^2-r'^2|&<\left|\|\mathbf w_{\head m_0}\|_2^2-\left\|\mathbf w'_{\head m}\right\|_2^2\right|+\|\mathbf w_{\tail m}\|^2+\rho'(r+r')\\
    &\leq \left\|\mathbf w_{\head m}-\mathbf w'_{\head m}\right\|_2\big(\|\mathbf w_{\head m}\|_2+\|\mathbf w'_{\head m}\|_2\big)+\delta_2 r/2+\rho'(r+r')\\
    &\leq 2\rho'(r+r')+\delta_2 r/2.
\end{split}
\end{align*}
Dividing both sides by $(r+r')$ and recalling $2 \rho' \le \delta_2/2$ , this implies
\begin{align}\label{eqn:second-coord}
\begin{split}
    |r-r'|<2\rho'+\delta_2/2\leq \delta_2.
\end{split}
\end{align}
When combined, \eqref{eqn:first-coord}, \eqref{eqn:third-coord}, and \eqref{eqn:second-coord} show that the supposition in \eqref{eqn:base-in-O} holds with $(\tilde{\bm w},\tilde r,\tilde s)=(\bm w',r',s')$. Thus \eqref{eqn:base-in-O} implies $(\bm w',r',s')\in O$, which proves the lemma.
\end{proof}

In the proof of the next lemma, we use the following concept from convex analysis.
\begin{defn}[Relative interior]
For every non-empty convex set $C$, the relative interior of $C$, denoted $\operatorname{ri}(C)$, is defined as the set
\[\operatorname{ri}(C) = \{x \in C: \text{for all } y \in C, \text{ there exists some } \mu > 1 \text{ such that }
\mu x  + ( 1 - \mu) y \in C
\}.
\]
\end{defn}

\begin{proof}[Proof of Lemma \ref{l:additional-good-properties}] 

We first show that $\Lambda_{m}$, defined in \eqref{lambdam}, is essentially smooth for all $m\in \N$ (see \Cref{d:esmooth}). The finiteness of $\Lambda_{m}$ on $\R^{m_0}\times\R\times(-\infty,T)$ follows from the finiteness of $\bar\Lambda$ on $\R\times(-\infty,T)$.
The differentiability of $\Lambda_{m_0}$ at $(\mathbf v',b',c')\in\R^{m_0}\times\R\times(-\infty,T)$ is a direct consequence of the differentiability of $\bar\Lambda$ on $\R\times(-\infty,T)$, the bounds in \Cref{assume4}, and the dominated convergence theorem.
To check the third condition in \Cref{d:esmooth}, first note that for $t_2\in(-\infty,T)$,
\begin{equation}
\label{eqn:second-d-pos}
\partial_2\bar\Lambda\left(t_1,t_2\right)=e^{-\bar\Lambda(t_1,t_2)}\E\left[\eta(X_1)e^{t_1X_1+t_2\eta(X_1)}\right]\geq 0.
\end{equation}
Hence, we have
\begin{align*}
\begin{split}
\|\nabla\Lambda_{m_0}(\mathbf v',b',c')\|_{2}\geq \E\left[\partial_2\bar\Lambda\left(\sum_{l=1}^{m_0}v'_lg_l+b'g_0,c'\right)\right]\geq0.
\end{split}
\end{align*}
Taking the limit infimum over $(\bm v',b',c')\in \R^{m_0}\times\R\times(-\infty,T)$ above and using Fatou's lemma, one obtains
\begin{align*}
\begin{split}
\liminf_{\mathbf v',b',c'\rightarrow \mathbf v_0,b_0,T}\|\nabla\Lambda_{m_0}(\mathbf v',b',c')\|_{2}\geq \E\left[\liminf_{\mathbf v',b',c'\rightarrow \mathbf v_0,b_0,T}\partial_2\bar\Lambda\left(\sum_{l=1}^{m_0}v'_lg_l+b'g_0,c'\right)\right]=\infty,
\end{split}
\end{align*}
where the last equality follows from the essential smoothness of $\bar\Lambda$, which in turn follows from \Cref{assume3}. This establishes the essential smoothness of $\Lambda_{m}$.

Fix an open set $\mathcal O$ and a point $(\mathbf w,r,s)\in\mathcal O$ such that $I(\w, r, s) < \infty$. Let $m_0$ be as in \Cref{l:product-topology} and recall the definition of $\Lambda_{m_0}$ from \eqref{lambdam}, and the definitions  of  $\Lambda^*_{m_0}$ and $\mathcal D_{m_0}^*$ from \eqref{lambdamstar} and \eqref{def-of-Lambdam-star}. 
Since $\Lambda_{m_0}$ is essentially smooth, by \cite[Corollary~26.4.1]{Rock15} we have
\begin{align}
\label{eqn:convex-inclusion}
     \nabla\Lambda_{m_0}(\mathcal D_{{m_0}})\subset \mathcal D_{{m_0}}^*.
 \end{align}
Note that $\Lambda_{m_0}(0,0,0)=0$ and the cumulant generating function $\bar\Lambda(s_1,s_2)=\log\mathbb E[e^{s_1X_1+s_2\eta(X_1)}]$ satisfies $\bar\Lambda(s_1,0)\geq\mathbb E[s_1X_1]=0$. Hence, the mapping $(\mathbf v',b')\mapsto\Lambda_{m_0}(\mathbf v',b',0)$ achieves its minimum of 0 at $(\mathbf 0,0)$. Combining this with \eqref{eqn:second-d-pos}, we have $\nabla\Lambda_{m_0}(\mathbf 0,0,0)=(\mathbf 0,0,s_0)$ for some $s_0\in\mathbb R_+$. By differentiating $\Lambda_{m_0}$ twice, we have
\begin{align*}
\begin{split}
    \nabla^2\Lambda_{m_0}(\mathbf 0,0,0)=\begin{pmatrix}
    \partial_1^2\bar\Lambda(\mathbf 0,0)I_{m_0+1}&\mathbf 0\\
    \mathbf 0^\trans&\partial_2^2\bar\Lambda(\mathbf 0,0)
    \end{pmatrix}=\begin{pmatrix}
    \operatorname{Var}(X_1)I_{m_0+1}&\mathbf 0\\
    \mathbf 0^\trans&\operatorname{Var}(\eta(X_1))
    \end{pmatrix},
\end{split}
\end{align*}
which is positive definite  (since, by \Cref{assume1}, $X_1$ is not degenerate and $\eta$ is not a constant function). Thus, the inverse function theorem implies that $\nabla\Lambda_{m_0}$ is locally a diffeomorphism, which, along with \eqref{eqn:convex-inclusion}, implies there exists $\varepsilon_0>0$ such that
$$B_{\varepsilon_0}(0,0,s_0)\subset\operatorname{ri}(\mathcal D_{m_0}^*).$$

Now, pick $(\tilde{\mathbf w},\tilde t,\tilde s)\in B_{\varepsilon_0}(0,0,s_0)$ such that $\tilde{\mathbf w}$ is strictly ordered in the sense that $|\tilde{\mathbf w}_1|>\ldots>|\tilde{\mathbf w}_{m_0}|$.
Set $t=\sqrt{r^2-\|\bm w\|^2_2}$. By Jensen's inequality, $\Lambda_{m_0}(\bm w_{\head m_0},t,s)\leq I(\bm w,r,s)<\infty$ by \eqref{lambdamstar}, and hence, $(\bm w_{\head m_0},t,s)\in\mathcal D_{m_0}$.
Since $(\tilde{\mathbf w},\tilde t,\tilde s)\in \operatorname{ri}(\mathcal D_{{m_0}}^*)$, \cite[Theorem~6.1]{Rock15} and \cite[Corollary~26.4.1]{Rock15} imply that, for all $\lambda\in[0,1)$,
\begin{equation*}
\lambda (\mathbf w_{\head {m_0}},t,s)+(1-\lambda)(\tilde{\mathbf w},\tilde t,\tilde s)\in \operatorname{ri}(\mathcal D_{{m_0}}^*)\subset
\{ \nabla \Lambda_{m_0}(y) : y \in \mathcal D_{m_0} \}.
\end{equation*}
 By the definitions of $(\mathbf w_{\head {m_0}},t,s)$ and $(\tilde{\mathbf w},\tilde t,\tilde s)$, note that the points on the line segment
$$\ell\defeq\{\lambda (\mathbf w_{\head {m_0}},t,s)+(1-\lambda)(\tilde{\mathbf w},\tilde t,\tilde s):\lambda\in(0,1)\}$$
are strictly ordered in the first component. Since $\Lambda_{m_0}^*$ is a convex function (by definition of the Legendre transform), it is continuous along $\ell\subset\mathcal D_{{m_0}}^*$. Together with \Cref{l:product-topology} (to establish \eqref{claim:2}), this implies that there exists a point $(\bar{\mathbf w},\bar t,\bar s)\in \ell$ that satisfies all conditions of the lemma.
\end{proof}

\subsection{Proof of the modified lower bound}\label{proof:lemma-small-ball}
The proof of \Cref{lemma:small-ball-probability}, which is given in \Cref{s:proofSmallBall}, uses a change of measure.

Throughout this section, we fix $\mu$ to be the law of a random variable satisfying Assumptions 2, 3 and 4, with associated constants $C_2, T>0$ and function $\tilde C(\cdot)$. For $s_1\in\mathbb R,s_2<T$, recalling $\bar\Lambda$ defined in \eqref{asprev}, we define the exponentially tilted measure $\tilde\mu_{s_1,s_2}$ by
\begin{align}\label{def:tilted-mu}
\begin{split}
\frac{d\tilde \mu_{s_1,s_2} }{d\mu}(x)\coloneqq\exp(s_1x+s_2\eta(x) -\overline\Lambda\left(s_1,s_2\right)),\quad x\in\mathbb R.
\end{split}
\end{align}

\subsubsection{Preparatory results}\label{s:prelim-proof4.2} The first lemma claims that $\{\tilde\mu_{s_1,s_2}\}_{(s_1,s_2)\in\mathbb R\times (-\infty,T)}$ is uniformly sub-Gaussian in $s_1$.
\begin{lem}\label{lemma:sub-Gaussian-norm-after-translation}
Fix $(s_1,s_2)\in\mathbb R\times (-\infty, T)$. Let $\tilde X_1$ be a random variable on a probability space $(\Omega, \mathcal F,\P_{s_1,s_2})$ that has law $\tilde\mu_{s_1,s_2}$, 
and let $\tilde{\E}_{s_1,s_2}$ denote expectation with respect to $\tilde{\P}_{s_1,s_2}$.
Then there exists a constant $C = C(s_2)>0$ such that
\begin{align*}
\begin{split}
\tilde{\P}_{s_1,s_2}\left(|\tilde X_1-\tilde{\E}_{s_1,s_2}[\tilde X_1]|>t\right)\leq 2\exp(-t^2/C^2).
\end{split}
\end{align*}
\end{lem}
\begin{proof}[Proof of \Cref{lemma:sub-Gaussian-norm-after-translation}]
Fix $(s_1,s_2)\in\R\times(-\infty,T)$.
By the definition of $\bar\Lambda$ in \eqref{asprev} and the definition of $\tilde\mu_{s_1,s_2}$ in \eqref{def:tilted-mu}, we have
\begin{align*}
\begin{split}
    \tilde\E_{s_1,s_2}\left[\tilde X_1\right]=\int xe^{s_1x+s_2\eta(x)-\bar\Lambda\left(s_1,s_2\right)}d\mu(x)=\partial_1\bar\Lambda\left(s_1,s_2\right).
\end{split}
\end{align*}
Now \Cref{assume4} implies there exists a continuous function $\tilde C_0\colon (- \infty, T) \rightarrow \R$ such that
\begin{align*}
\begin{split}
    \tilde{\E}_{s_1,s_2}[\exp(\lambda (\tilde X_1-\tilde{\E}_{s_1,s_2}[\tilde X_1]))]=&e^{-\lambda\tilde{\E}_{s_1,s_2}[\tilde X_1]-\Lambda\left(s_1,s_2\right)}\int\exp(\lambda x)\exp(s_1x+s_2\eta(x))\, d\mu(x)\\
    =&\exp(\bar\Lambda(s_1+\lambda,s_2)-\bar\Lambda(s_1,s_2)-\lambda\partial_1\bar\Lambda(s_1,s_2))\\
    \leq&\exp\left(\tilde C_0(s_2)^2\lambda^2\right).
\end{split}
\end{align*}
The lemma then follows from a standard argument using Chernoff's bound (see \cite[Proposition~2.5.2]{Ver18}).
\end{proof}
Next, we state  a Gaussian approximation result for the Haar measure on $\mathbb V_{n,k_n}$ in \Cref{lemma:second-third-concentration}, and establish an asymptotic decorrelation result in \Cref{lemma:second-third-convergence3}. 
The proofs of these lemmas are deferred to \Cref{s:aux-proof}. 

Recall from \Cref{s:lowerBoundNotations} that $T$ denotes the constant from \Cref{assume3}, and that $M$ was defined in \eqref{eqn:def-of-M} in terms of $m$ and a point $(\hatu,b,c)\in\R^m\times\R\times(-\infty,T)$. Recall from \eqref{eqn:lawOfJointGaussians} that $g_1,g_2,\ldots,g_{m+M}$ denote i.i.d.\ Gaussian random variables with zero mean and unit variance. Additionally, given any $n\in \N$ and $j\in \{1, \dots, k_n\}$, and any function $h:\mathbb R\rightarrow\mathbb R$, 
we define a map $f_j^{(n)}\colon \V_{n,k_n} \rightarrow \R$ as follows:
\begin{align}\label{def:fjn}
\begin{split}
    f_j^{(n)}\left(\mathbf a^{(n)}\right)\coloneqq \frac{1}{n} \sum_{i=1}^{n} \sqrt{n} a_{i j}^{(n)} h\left(\left\langle\mathbf{u}^{(n)}, \sqrt{n} \mathbf{a}_{i}^{(n)}\right\rangle\right)=\int x h(y) \, d \nun_j(x, y),
\end{split}
\end{align}
where we recall that $\bm u^{(n)}=(u_1,\ldots,u_{k_n})$ and $\nun_j$ were specified in \eqref{eqn:def-of-u} and \eqref{eqn:nu-n-j} respectively.

\begin{lem}\label{lemma:second-third-concentration}
Fix $C_0\in (0, \infty)$ and a function $h:\mathbb R\rightarrow\mathbb R$ satisfying
\begin{align}\label{eqn:h-growth}
\begin{split}
    \left|\frac{d^\alpha}{dx^\alpha}h(x)\right|\leq C_0(1+|x|^{1-\alpha}), \qquad \alpha=0,1.
\end{split}
\end{align}
Fix $m \in \N$ and $(\hatu,b,c)\in\R^m\times\R\times(-\infty,T)$.
Then for $\sigma$-a.e. $\mathbf a=(\bm a^{(1)},\bm a^{(2)},\ldots)\in\mathbb V$, we have for  all $j\in \{1,\ldots,m+M\}$ that
\begin{equation*}
\lim_{n\rightarrow \infty} f_j^{(n)}\left(\mathbf a^{(n)}\right)= \mathbb E\left[g_jh\left(\sum_{l=1}^{m+M}u_lg_l\right)\right].
\end{equation*}
\end{lem}
\begin{rmk}\label{rmk:condition-Lambda}
For any fixed $x \in(-\infty,T)$, by \Cref{assume4}, $\partial_1\bar\Lambda(\cdot,x )$ satisfies the condition imposed on $h$ in \Cref{lemma:second-third-concentration}.
\end{rmk}
\begin{lem}\label{lemma:second-third-convergence3} 
Fix $m \in \N$ and $(\hatu,b,c)\in\R^m\times\R\times(-\infty,T)$, and retain the assumption \eqref{eqn:h-growth}. 
Then for $\sigma$-a.e. $\mathbf a=(\bm a^{(1)},\bm a^{(2)},\ldots)\in\mathbb V$,
\begin{align}\label{eqn:remaining-convergence}
\begin{split}
    \lim_{n\rightarrow \infty} \sum_{\ell=m+M+1}^{k_n}\left(f_\ell^{(n)}\left(\an\right)\right)^2= 0.
\end{split}
\end{align}
\end{lem}

\subsubsection{Proof of \Cref{lemma:small-ball-probability}}\label{s:proofSmallBall}
We now show how the preparatory results from \Cref{s:prelim-proof4.2} can be combined to establish \Cref{lemma:small-ball-probability}.
Let  $m\in\mathbb N$, $(\hatu,b,c)\in\mathcal D_m$ and  $(\hatw,t,s)=\nabla\Lambda_m(\hatu,b,c)$ be as in the statement of the lemma. 
Let $M$ be defined in terms of $t$ and $m$ as in \eqref{eqn:def-of-M} and let $\bm u^{(n)}\in\mathbb R^{k_n}$ be defined in terms of $\hatu$ and $b$ as in \eqref{eqn:def-of-u}.
Let $(X_1, \dots, X_n)$ be random variables defined on a measure space $(\Omega, \mathcal F)$ equipped with the probability measures $\P^{(n)}$ and $\tilde \P^{(n)}$. Suppose that under $\P^{(n)}$, the variables $(X_1, \dots, X_n)$ are i.i.d.\ with a common law $\mu$ that satisfies Assumptions 1--4, and let $\mu^{(n)}\coloneqq \mu^{\otimes n}$.
Let $\tilde\mu_{s_1,s_2}$ be defined in terms of $\mu$ as in \eqref{def:tilted-mu}, and let $\tPn$ be a probability measure on $(\Omega, \mathcal F)$ such that under $\tPn$, $(X_1, \dots , X_n)$ has law $\tmun$, where
\be\label{tildemun}
        \tilde\mu^{(n)}\coloneqq\otimes_{i=1}^n\tilde\mu_{\lambda_i,c},\quad\lambda_i\equiv\lambda_i^{(n)}\=\left\langle \bm u^{(n)},\sqrt n\an_i\right\rangle,\quad i=1, \dots, n.
\ee
Let $(\tilde X_1,\ldots,\tilde X_n)$ be an $\R^n$-valued random vector with law $\tmun$. 
Let $\En$ and $\tEn$ denote expectations with respect to $\P^{(n)}$ and $\tPn$, respectively.
Note that the Radon--Nikodym derivative of $\tilde\mu^{(n)}$ with respect to $\mu^{(n)}$ can be rewritten in terms of the function $F_n$ from \eqref{eqn:quenched-log-moment} as follows:
\begin{align}\label{eqn:def-tilde-mu-n}
    \begin{split}
        \frac{d\tilde\mu^{(n)}}{d\mu^{(n)}}(\bm x)=e^{n\left\langle\boldsymbol{u}^{(n)}, \frac{1}{\sqrt{n}}\left(\boldsymbol{a}^{(n)}\right)^\trans \boldsymbol{x}\right\rangle+c \sum_{i=1}^n \eta\left(x_i\right)-n F_n\left(\boldsymbol{u}^{(n)}, \boldsymbol{a}^{(n)}, c\right)},\ \forall \boldsymbol{x}=\left(x_1, \ldots, x_n\right) \in \mathbb{R}^n.
    \end{split}
\end{align}

Next, recall that $\kappa >0$ and $B=B((\hatw,t,s),\kappa )$ were fixed in the statement of \Cref{lemma:small-ball-probability}, and
for $n$ such that $k_n \ge M + m$, let 
\begin{align}\label{717}
\begin{split}
    \mathbf w^{(n)}=(w_1,\ldots,w_{k_n})\defeq(\hatw,\underbrace{tM^{-1/2},\ldots,tM^{-1/2}}_{M\text{ times}},\underbrace{0,\ldots,0}_{k_n-m-M\text{ times}})\in\mathbb R^{k_n}.
\end{split}
\end{align}
Fix $\delta\in(0,\min(\kappa,m^{-1/2}))$ and define
\begin{align*}
    E_n\coloneqq\bigg\{
    \bm x=(x_1,\cdots,x_n)\in\mathbb R^n:\left(\frac{1}{\sqrt n}(\an)^\trans\bm x,\,\frac{1}{n}\sum_{i=1}^n\eta(x_i)\right)\in B\left((\bm w^{(n)},s),\delta\right)
    \bigg\}.
\end{align*}
Let  $W^{(n)}$ and $L^{(n)}$ be defined in terms of $(X_1,\dots, X_n)$ as in \eqref{def-of-W-L}.
Then using the definitions of $\mu^{(n)}$, $\tilde\mu^{(n)}$, and $E_n$, and \eqref{eqn:def-tilde-mu-n}, we have
\begin{align}\label{eqn:lower-bound-derivations}
\begin{split}
&\Pn\left((W^{(n)}_{\head m},\|W^{(n)}_{\tail m}\|_2,L^{(n)})\in B,\|W^{(n)}_{\tail m}\|_\infty\leq 2m^{-1/2}\right)\\
\geq& \P^{(n)}\left(\left(\Wn,L^{(n)}\right)\in B\left((\mathbf w^{(n)},s),\delta\right)\right)\\
=&\int_{E_n}d \mun (\bm x)\\
=&\int_{E_n}e^{-n\left( \left\langle\boldsymbol{u}^{(n)}, \frac{1}{\sqrt{n}}\left(\boldsymbol{a}^{(n)}\right)^\trans \boldsymbol{x}\right\rangle+\frac{c}{n}\sum_{i=1}^n \eta\left(x_i\right)\right)+nF_n(\mathbf u^{(n)},\mathbf a^{(n)},c)}d\tmun(\bm x)\\
\geq&\tilde \P^{(n)}\left(\left(\Wn,L^{(n)}\right)\in B\left((\mathbf w^{(n)},s),\delta\right)\right)e^{-n\left(\langle \hatu,\hatw\rangle+bt+cs-F_n(\mathbf u^{(n)},\mathbf a^{(n)},c)+\delta\|(\hatu,b,c)\|_2\right)}.
\end{split}
\end{align}
We show below that the following three claims hold for $\sigma$-a.e. $\mathbf a\in\mathbb V$ and any $\varepsilon>0$:
\begin{align}
    \lim_{n\rightarrow\infty}\tilde \P^{(n)}\left(|W^{(n)}_j-w_j|\geq\varepsilon\right)&=0,\ j=1,\ldots,m+M,\label{eqn:auxiliary-inequality-1}\\
    \lim_{n\rightarrow\infty}\tilde \P^{(n)}\left(\sum_{l=m+M+1}^{k_n}(W^{(n)}_l)^2\geq\varepsilon\right)&=0,\label{eqn:auxiliary-inequality-3}\\
    \lim_{n\rightarrow\infty}\tilde \P^{(n)}\left(\left|L^{(n)}-s\right|\geq\varepsilon\right)&=0.\label{eqn:auxiliary-inequality-4}
\end{align}
In essence, the limits \eqref{eqn:auxiliary-inequality-1} and \eqref{eqn:auxiliary-inequality-4} are weak laws of large numbers for $\Wn_j,j=1,\ldots,m$, and $\Ln$ respectively, under the tilted measure, and \eqref{eqn:auxiliary-inequality-3} states that the $\ell^2$-norms of $\Wn_{\tail m+M}$ are asymptotically negligible under the tilted measure.

Postponing the proofs of \eqref{eqn:auxiliary-inequality-1}--\eqref{eqn:auxiliary-inequality-4}, we first note that together with  \eqref{eqn:def-tilde-mu-n}, they imply that for $\sigma$-a.e. $\mathbf a\in\mathbb V$,
\begin{align}\label{eqn:auxiliary-result}
\begin{split}
    \liminf_{n\rightarrow\infty}\tilde \P^{(n)}\left(\left(W^{(n)},L^{(n)}\right)\in B\left((\mathbf w^{(n)},s),\delta\right)\right)=1.
\end{split}
\end{align}
Moreover, $F_n(\mathbf u^{(n)},\mathbf a^{(n)},c)\rightarrow \Lambda_m(\hatu,b,c)$ for $\sigma$-a.e.\ $\mathbf a\in\mathbb V$ by Lemma \ref{l:lln}, with $d_n=1$. Thus, combining \eqref{eqn:lower-bound-derivations} and \eqref{eqn:auxiliary-result}, and recalling the defining of $\Lambda^*_m$ in \eqref{lambdamstar}, 
we conclude that for $\sigma$-a.e.\ $\mathbf a\in\mathbb V$,
\begin{align*}
\begin{split}
    \liminf_{n\rightarrow\infty}&\ n^{-1}\log \P^{(n)}\left((\Wn_{\head m},\|\Wn_{\tail m}\|_2,L)\in B,\|\Wn_{\tail m}\|_\infty\leq 2m^{-1/2}\right)\\
    \geq &-\left(\langle \hatu,\hatw\rangle+bt+cs-\Lambda_m\left(\hatu,b,c\right) +\delta\|(\hatu,b,c)\|_2\right)\\
    \ge &-\Lambda_m^*(\hatw,t,s)-\delta\|(\hatu,b,c)\|_2.
\end{split}
\end{align*}
Taking $\delta\searrow0$ completes the proof of \eqref{eqn:small-ball-probability}, given \eqref{eqn:auxiliary-inequality-1}--\eqref{eqn:auxiliary-inequality-4}.

We now turn to the proofs of \eqref{eqn:auxiliary-inequality-1}--\eqref{eqn:auxiliary-inequality-4}.

\begin{proof}[Proof of \eqref{eqn:auxiliary-inequality-1}]
By the definition of $\tilde \mu^{(n)}$ in \eqref{eqn:def-tilde-mu-n} and Lemma \ref{lemma:sub-Gaussian-norm-after-translation} (with $s_1 = \lambda_i$ from \eqref{tildemun}, and $s_2 = c$, where $c$ is from the statement of the lemma), we have, for each $i=1,\ldots,n$,
\begin{align*}
\begin{split}
\tilde \P^{(n)}\left(|X_i-\tilde \E^{(n)} \left[X_i\right]|>t\right)\leq 2\exp(-t^2/C^2),
\end{split}
\end{align*}
for some constant $C>0$ depending only on $c$. 
Also, the definition of $W^{(n)}$ in \eqref{def-of-W-L},  Hoeffding's inequality for sub-Gaussian random variables (see \cite[Theorem~2.6.2]{Ver18}), and a union bound imply that for all $\mathbf a\in\V$,
\begin{align}\label{eqn:first-concentration}
\begin{split}
&\tilde \P^{(n)}\left(\max_{1\leq j\leq m+M}\left|\Wn_j-\tilde \E^{(n)} [\Wn_j]\right|\geq\varepsilon/3\right)\\
&\leq (m+M)\max_{1\leq j\leq m+M}\tilde \P^{(n)}\left(\left|\sum_{i=1}^n\frac{a_{ij}^{(n)}}{\sqrt n}(X_i-\tilde \E^{(n)} [X_i])\right|\geq\varepsilon/3\right)\\
&\leq 2(m+M)\exp(-Cn\varepsilon^2),
\end{split}
\end{align}
which decays to zero as $n\rightarrow\infty$. Also, by the definitions of $\bar\Lambda$ and $\tmun$ in \eqref{asprev} and \eqref{eqn:def-tilde-mu-n}, respectively, we have, 
\begin{equation}\label{eqn:xi-expectation-tilted}
    \begin{split}
        \tEn[X_i]=\partial_1\bar\Lambda\left(\langle \un,\sqrt n\mathbf a_i^{(n)}\rangle,c\right),\quad i=1,\ldots,n.
    \end{split}
\end{equation}
By the definition of $W^{(n)}$ in \eqref{def-of-W-L} and the relation \eqref{eqn:def-tilde-mu-n}, it follows that
\begin{align}\label{eqn:wj-expectation-tilted}
\begin{split}
\tilde \E^{(n)} [\Wn_j]=\frac{1}{n}\sum_{i=1}^n\sqrt na_{ij}^{(n)}\partial_1\bar\Lambda\left(\langle \un,\sqrt n\mathbf a_i^{(n)}\rangle,c\right),\quad j=1,\ldots,n.
\end{split}
\end{align}
Then Lemma \ref{lemma:second-third-concentration} with $h=\partial_1\Lambda(\cdot,c)$ and \Cref{rmk:condition-Lambda} imply that for $\sigma$-a.e.\ $\mathbf a$ and each $j=1,\ldots,m+M$,
\begin{align}\label{eqn:second-third-concentration}
\begin{split}
  \lim_{n\rightarrow \infty}  \tilde \E^{(n)} [\Wn_j] =  \mathbb{E}\left[g_{j} \partial_1\bar\Lambda\left(\sum_{l=1}^{m+M} u_{l}^{(n)} g_{l},c\right)\right].
\end{split}
\end{align}
Let $g_0$ be a Gaussian random variable independent of $g_1,\ldots,g_{m}$. 
Recalling from \eqref{eqn:def-of-u} that $u_\ell^{(n)} = \hat u_\ell$ for $\ell =1, \dots, m$ and $u_{m+1}^{(n)} = \dots = u_{m+M}^{(n)} = b M^{-1/2}$, we have 
$\sum_{l=m+1}^{m+M}u_l^{(n)} g_l\distequal bg_0$
and $(g_0, g_1, \dots , g_m) \distequal ( M^{-1/2} \sum_{l=m+1}^{m+M} g_l, g_1, \dots , g_m)$. 
Together with \eqref{lambdam} and \eqref{gradrel}, for $j=1,2,\ldots,m$, we have
\begin{align}\label{eqn:sec-third-case1}
\begin{split}
    \mathbb{E}\left[g_{j} \partial_1\bar\Lambda\left(\sum_{l=1}^{m+M} u_{l} g_{l},c\right)\right]=\E\left[g_j\partial_1\bar\Lambda\left(\sum_{l=1}^mu_lg_l+bg_0,c\right)\right]=\partial_{j}\Lambda_m(\hatu,b,c)=w_j,
\end{split}
\end{align}
and for $j=m+1,\ldots,m+M$, also using \eqref{717}, we have 
\begin{align}\label{eqn:sec-third-case2}
\begin{split}
    \mathbb{E}\left[g_{j} \partial_1\bar\Lambda\left(\sum_{l=1}^{m+M} u_{l}^{(n)} g_{l},c\right)\right]=&\E\left[\frac{\sum_{l=m+1}^{m+M}g_l}{M^{1/2} }\partial_1\bar\Lambda\left(\sum_{l=1}^{m} \hat u_{l} g_{l}+bM^{-1/2}\sum_{l=m+1}^{m+M}g_l,c\right)\right]\\
    =&M^{-1/2}\E\left[g_0\partial_1\bar\Lambda\left(\sum_{l=1}^m\hat u_lg_l+bg_0,c\right)\right]\\
    =&M^{-1/2}\partial_2\Lambda_m\left(\hatu,b,c\right)\\
    =&tM^{-1/2}\\
    =&w_j^{(n)}.
\end{split}
\end{align}
Combining \eqref{eqn:first-concentration}, \eqref{eqn:second-third-concentration}, \eqref{eqn:sec-third-case1} and \eqref{eqn:sec-third-case2}, \eqref{eqn:auxiliary-inequality-1} follows.
\end{proof}

\begin{proof}[Proof of \eqref{eqn:auxiliary-inequality-3}]
By Markov's inequality and the independence of $\{X_i,\,i=1,\ldots,n\}$ under the tilted measure $\tPn$ (see \eqref{tildemun}), we have 
\begin{align*}
\begin{split}
    \tilde \P^{(n)}\left(\sum_{l=m+M+1}^{k_n}(\Wn_l)^2\geq\varepsilon\right)\leq& \varepsilon^{-1}\mathbb \E^{(n)}\left[\sum_{l=m+M+1}^{k_n}(\Wn_l)^2\right]\\
    =&\varepsilon^{-1}n^{-1}\sum_{i,j=1}^n\tEn\left[X_iX_j\right]\sum_{l=m+M+1}^{k_n}a^{(n)}_{il}a^{(n)}_{jl}\\
    =&\varepsilon^{-1}n^{-1}\left(\sum_{i,j=1}^n\tilde \E^{(n)}[X_i]\tilde \E^{(n)}[X_j]\sum_{l=m+M+1}^{k_n}a^{(n)}_{il}a^{(n)}_{jl}+\right.\\
    &\quad \quad \quad \quad \quad \left.\sum_{i=1}^n \widetilde{\operatorname{Var}}\left(X_i\right)\sum_{l=m+M+1}^{k_n}(a^{(n)}_{il})^2\right),
\end{split}
\end{align*}
where $\widetilde{\operatorname{Var}}(X_i)$ denotes the variance of $X_i$ under the tilted measure $\tPn$.
By a calculation similar to the one carried out in \eqref{eqn:xi-expectation-tilted}, we have $\widetilde{\operatorname{Var}}(X_i)=\partial_{1}^2\bar\Lambda\left(\left\langle \mathbf u^{(n)},\sqrt n\mathbf a_i^{(n)}\right\rangle,c\right)$, which is  bounded uniformly in $n \in \N$  and $i=1,\ldots,n$ by \Cref{assume4}.
Additionally, 
$\sum_{l=m+M+1}^{k_n}(a_{il}^{(n)})^2$ is uniformly bounded in $n \in \N$  and $i=1,\ldots,n$.
Therefore,
\begin{align*}
\begin{split}
    \tilde \P^{(n)}\left(\sum_{l=m+M+1}^{k_n}W_l^2\geq\varepsilon\right)\leq
\varepsilon^{-1}n^{-1}\sum_{i,j=1}^n\tilde \E^{(n)}[X_i]\tilde \E^{(n)}[X_j]\sum_{l=m+M+1}^{k_n}a^{(n)}_{il}a^{(n)}_{jl}+\varepsilon^{-1}O\left(\frac{k_n}{n}\right).
\end{split}
\end{align*}
Hence, it suffices to show that for $\sigma$-a.e. $\bm a\in\mathbb V$,
\begin{align*}
\begin{split}
  \lim_{n\rightarrow \infty}  n^{-1}\sum_{i,j=1}^n\tilde \E^{(n)}[X_i]\tilde \E^{(n)}[X_j]\sum_{l=m+M+1}^{k_n}a^{(n)}_{il}a^{(n)}_{jl}= 0.
\end{split}
\end{align*}
Recalling $\tilde \E^{(n)} [X_i]=\partial_1\bar\Lambda\left(\langle \mathbf u^{(n)},\sqrt n \mathbf a_i^{(n)}\rangle,c\right)$ from \eqref{eqn:xi-expectation-tilted}, the above convergence follows directly from Lemma \ref{lemma:second-third-convergence3}, with $h=\bar\Lambda(\cdot,c)$, and \Cref{rmk:condition-Lambda}.
\end{proof}

\begin{proof}[Proof of \eqref{eqn:auxiliary-inequality-4}]
Let $\lambda>0$ be a small parameter, to be chosen later. 
Note that $\tilde\E^{(n)}\left[\eta\left(X_i\right)\right]=\partial_2\bar\Lambda(\langle \mathbf u^{(n)},\sqrt n\mathbf a_i^{(n)}\rangle,c)$, by a calculation similar to \eqref{eqn:xi-expectation-tilted} (using \eqref{asprev} and \eqref{eqn:def-tilde-mu-n}).
Then by Chebyshev's inequality, the mutual independence of $(X_1, \dots, X_n)$ under $\tPn$, the definition of $\overline{\Lambda}$ in \eqref{asprev}, and  calculations similar to those  in the proof of \Cref{lemma:sub-Gaussian-norm-after-translation}, we have
\begin{align*}
\begin{split}
    &\tilde \P^{(n)}\left(\frac{1}{n}\sum_{i=1}^n\left(\eta\left(X_i\right)-\tilde \E^{(n)}\left[\eta\left(X_i\right)\right]\right)>\frac{\varepsilon}{2}\right)\\
    &\leq e^{-\lambda n\varepsilon/2}\prod_{i=1}^n\tilde \E^{(n)}\left[\exp\left(\lambda\eta\left(X_i\right)-\lambda\tilde \E^{(n)}\left[\eta\left(X_i\right)\right]\right)\right]\\
    &=e^{-\lambda n\varepsilon/2}\prod_{i=1}^n\exp\left(\bar\Lambda\left(\langle \mathbf u^{(n)},\sqrt n\mathbf a_i^{(n)}\rangle,c+\lambda\right)-\bar\Lambda\left(\langle \mathbf u^{(n)},\sqrt n\mathbf a_i^{(n)}\rangle,c\right)-\lambda\partial_2\bar\Lambda\left(\langle \mathbf u^{(n)},\sqrt n\mathbf a_i^{(n)}\rangle,c\right)\right)\\
    &= \exp\left(\frac{\lambda^2}{2}\sum_{i=1}^n\partial_{2}^2\bar\Lambda\left(\langle \mathbf u^{(n)},\sqrt n\mathbf a_i^{(n)}\rangle,\xi_i\right)-\lambda n\frac{\varepsilon}{2}\right),
\end{split}
\end{align*}
where $\xi_i\in(c,c+\lambda)$, $i=1,\dots ,n$. 
Together with \Cref{assume4}, we have \begin{align*}
\begin{split}
    \tilde \P^{(n)}\left(\frac{1}{n}\sum_{i=1}^n\left(\eta\left(X_i\right)-\tilde \E^{(n)}\left[\eta\left(X_i\right)\right]\right)>\frac{\varepsilon}{2}\right)
    \leq\exp\left(\left(-\frac{\lambda}{2}\varepsilon+C(1+\|\mathbf u\|_2^2)\lambda^2\right)n\right)\rightarrow0
\end{split}
\end{align*}
for all $\mathbf a\in\V$, if $\lambda$ is chosen small enough.
To prove \eqref{eqn:auxiliary-inequality-4}, it suffices to show that, for $\sigma$-a.e. $\bm a=\left(\bm a^{(1)},\bm a^{(2)},\ldots\right)\in\V$,
\begin{align}\label{eqn:eta-convergence}
\begin{split}
    \frac{1}{n}\sum_{i=1}^n\partial_2\bar\Lambda(\langle \mathbf u^{(n)},\sqrt n\mathbf a_i^{(n)}\rangle,c)\rightarrow\partial_3\Lambda_m(\hatu,b,c)=s.
\end{split}
\end{align}
Since $H=\partial_2 \bar \Lambda$ satisfies the conditions in \eqref{r:explanatoryremark} due to \Cref{assume4} and \eqref{eqn:def-of-M}, and because $\|\u^{(n)}\|_2 = D \coloneqq \sqrt{ \| \hat \u \|_2^2 + |b|^2 }$ for all $n\in \N$, it follows from \Cref{l:lln} and \Cref{r:explanatoryremark} that \eqref{eqn:eta-convergence} holds with the right-hand side equal to $\E[\partial_2 \bar \Lambda(Dg, c) ]$. Then the fact that  $\E[\partial_2 \bar \Lambda(Dg, c) ] = \partial_3\Lambda_m(\hatu,b,c)$ follows after recalling \eqref{d:g0}, \eqref{lambdam}, and $\langle \u^{(n)}, g\rangle + bg_0 \distequal Dg$ (by \eqref{eqn:def-of-M}). This completes the proof. 
\end{proof}
Since all three claims \eqref{eqn:auxiliary-inequality-1}--\eqref{eqn:auxiliary-inequality-4} have been established, this completes the proof of \Cref{lemma:small-ball-probability}.

\appendix
\section{Properties of \texorpdfstring{$p$}{p}-Gaussian Variables}\label{s:appendix}
\begin{lem}\label{l:bound-for-p-Gaussian}    Fix $p \ge 2$,  let $X$ be a random variable with probability density function equal to $\bar c_p \exp(- p^{-1} |x|^p)$ for a suitable normalization constant $\bar c_p$, and define
\[
        \widehat\Lambda(t)\coloneqq\log\mathbb E\big[\exp(tX)\big], \quad t\in \R.
\]
    Then there exists a constant $C \in (0, \infty)$ such that the following properties hold for all $t\in \R$:
    \begin{enumerate}
    \item $0\leq \widehat\Lambda(t)\leq C\left(1+t^2\right)$.
    \item $|\widehat\Lambda'(t)|\leq C\left(1+|t|\right)$.    \item $\widehat\Lambda''(t) \le C$.
    \end{enumerate}
\end{lem}

\begin{proof} Since the distribution of $X$ is symmetric by definition, the log moment gerating function $\Lambda$ satisfies $\Lambda(t) = \Lambda(-t)$ for all $t \in \R$, so we may assume that $t\ge 0$.
Let
\begin{align*}
 \begin{split}
     M(t)=\mathbb E\big[ \exp(tX)\big]=\bar c_p\int_{\mathbb R}\exp\left(tx-\frac{|x|^p}{p}\right)dx.
 \end{split}
 \end{align*}
By the change of variables $x=t^{1/(p-1)}y$, we have
\begin{align*}
\begin{split}
    M(t)=\bar c_p t^{\frac{1}{p-1}}\int_{\mathbb R}\exp\left(t^{\frac{p}{p-1}}\left(y-\frac{|y|^p}{p}\right)\right)dy.
\end{split}
\end{align*}
Define  $S(y)\coloneqq y-\frac{|y|^p}{p}$ for $y \in \R$.  Note that $y_0=1$ is the unique maximizer of $S$, $S$ is smooth in a neighborhood of $y=1$, and $S''(1)=1\neq0$. An application of Laplace's method (see \cite[Section~19.2.5,~Theorem~2(b)~\&~Remark~4]{zorich2016mathematical}) shows that, as $t\rightarrow\infty$,
\begin{align*}\label{eqn:asymptotic-expansion-of-M}
\begin{split}
    M(t)=\bar c_pt^{\frac{1}{p-1}}\sqrt{\frac{2\pi}{p-1}}\exp\left(t^{\frac{p}{p-1}}\frac{p-1}{p}\right)t^{-\frac{p}{2(p-1)}}\left(1+O\left(t^{-\frac{p}{p-1}}\right)\right),
\end{split}
\end{align*}
where in each case, the implicit constant in the big $O$ notation depends only on $p$.
Similarly, as $t\rightarrow\infty$,
\begin{align*}
\begin{split}
    M'(t)=\bar c_pt^{\frac{2}{p-1}}\sqrt{\frac{2\pi}{p-1}}\exp\left(t^{\frac{p}{p-1}}\frac{p-1}{p}\right)t^{-\frac{p}{2(p-1)}}\left(1+O\left(t^{-\frac{p}{p-1}}\right)\right),
\end{split}
\end{align*}
\begin{align*}
\begin{split}
    M''(t)=\bar c_pt^{\frac{3}{p-1}}\sqrt{\frac{2\pi}{p-1}}\exp\left(t^{\frac{p}{p-1}}\frac{p-1}{p}\right)t^{-\frac{p}{2(p-1)}}\left(1+O\left(t^{-\frac{p}{p-1}}\right)\right).
\end{split}
\end{align*}
Therefor, as $t\rightarrow \infty$, we also have
\begin{align*}
    \widehat\Lambda(t)&=\log\left(M(t)\right)=O\left(t^{\frac{p}{p-1}}\right),\\
    |\widehat\Lambda'(t)|&=\left|\frac{M'(t)}{M(t)}\right|=t^{\frac{1}{p-1}}\left(1+O\left(t^{-\frac{p}{p-1}}\right)\right),\\
    \widehat\Lambda''(t)&=\frac{M''(t)M(t)-\left(M'(t)\right)^2}{\left(M(t)\right)^2}=O\left(t^{\frac{2-p}{p-1}}\right).
\end{align*}
Since $p\ge 2$, the conclusions of the lemma follow from the previous displays and $p\geq 2$.
\end{proof}

\begin{lem}\label{l:bound-on-Lambda}
Fix $p \ge 2$,  let $X$ be
as in \Cref{l:bound-for-p-Gaussian}, and
define
\begin{align*}
\begin{split}
    \widehat\Lambda\left(t_1,t_2\right)\coloneqq \mathbb E\big[\exp(t_1X+t_2|X|^p)\big], \quad (t_1, t_2) \in \R^2.
\end{split}
\end{align*}
Then for every pair of integers $\alpha,\beta \ge 0$ with $\alpha + \beta \le 2$, there exists a continuous map $(0, p^{-1})\ni t_2 \mapsto C_{t_2,\alpha,\beta} \in (0, \infty)$
such that
\[
    \partial_1^{\alpha}\partial_2^{\beta}\widehat\Lambda\left(t_1,t_2\right)\leq C_{t_2,\alpha,\beta}\left(1+|t_1|^{2-\alpha}\right),
    \quad t_1\in \R.
\]
\end{lem}
\begin{proof}
Let $M(t)=\mathbb E[\exp(tX)]$. It was shown in
\cite[Lemma~5.7]{GKR17} that for $(t_1, t_2) \in \R \times (-\infty, p^{-1})$,
\begin{align*}
\begin{split}
    \widehat\Lambda\left(t_1,t_2\right)=-\frac{1}{p} \log \left(1-p t_{2}\right)+\log M\left(\frac{t_{1}}{\left(1-p t_{2}\right)^{1 / p}}\right).
\end{split}
\end{align*}
The lemma follows from this representation and \Cref{l:bound-for-p-Gaussian}.
\end{proof}

\section{Weingarten Calculus}\label{s:weingarten}
This appendix contains some preliminary remarks on the Weingarten calculus, and then a lemma necessary for the Gaussian approximation results proved in \Cref{s:gauss1} and \Cref{s:65}.
We recall the following definitions from \cite{CM18}.
Fix $d\in \mathbb{N}$ and let $\mathcal{M}(2d)$ be the set of pair partitions $\m$ of $\{1,2,\dots, 2d\}$; these are partitions where each block of the element has exactly two elements. They have a canonical form $\{ (\m(1), \m(2) ), \dots , (\m(2d-1), \m(2d)) \}$ for $\m(2i-1) \le \m(2i)$ and $\m(1) < \m(3) < \cdots \m(2d-1)$.

Given pair partitions $\m, \n \in \mathcal{M}(2d)$, we let $\Gamma(\m, \n)$ denote the graph that has vertices $\{1,2,\dots, 2d\}$,
and edges $\{(\m(2i-1), \m(2i)),(\n(2i-1), \n(2i))  \}_{i=1}^d$. The Gram matrix $G^{(n)}_d = [G^{(n)}( \m, \n)]_{\m,\n \in \mathcal{M}(2d)}$ is defined through its entries
\[
G^{(n)}(\m,\n) = n^{\mathrm{loop}(\m, \n)},
\]
where $\mathrm{loop}(\m, \n)$ is defined as the number of connected components of $\Gamma(\m, \n)$. The matrix $\wg^{(n)} = \wg^{(n)}_d$ is defined as the pseudo-inverse of $G^{(n)}_d$. We denote its entries by
\[
\wg^{(n)}
= [\wg^{(n)}( \m, \n)]_{\m,\n \in \mathcal{M}(2d)}.
\]

We recall the following theorem, which was originally proved in \cite{CS06}.
\begin{thm}[{\cite[Theorem 2.1]{CM18}}]
\label{t:weingarten1}
Given $i_1,\dots i_{2d}$ and $j_1,\dots, j_{2d}$ in $\{1,2,\dots, n\}$,
\begin{multline}\label{wgformula}
\int_{g\in O_n} g_{i_1 j_1} \cdots g_{i_{2d}j_{2d}}\, dg\\
= \sum_{\m,\n\in \mathcal M (2d) }
\wg^{(n)}(\m,\n) \prod_{k=1}^d \delta(i_{\m(2k-1)}, i_{\m(2k)})
\delta(j_{\m(2k-1)}, j_{\m(2k)}),
\end{multline}
where $\delta(i,j)$ denotes the Kronecker delta function satisfying $\delta(i,j)=1$ if $i=j$ and $\delta(i,j) =0$ otherwise.
\end{thm}
Note that the sum on the right side of \eqref{wgformula} is over all choices of pair partitions that pair indices with the same value. In particular, \Cref{t:weingarten1} shows that any moment $g_{i_1 j_1} \cdots g_{i_{2d}j_{2d}}$ without an even number of entries in each row and column vanishes.
\begin{ex} In the expectation of $g_{11}^2 g_{22}^2$ computed according to \eqref{wgformula}, the only possibility for a nonzero term in the sum is that both partitions are $\m=\n=\{(1,2), (3,4)\}$.
\end{ex}

Considering $\m$ and $\n$ as members of the symmetric group $S_{2d}$ (products of transpositions), we note that the value of $\wg^{(n)}(\m,\n)$ depends only on the value of $\sigma = \m^{-1}\n$ \cite[Theorem 3.1]{CM18}. Consider the graph $\overline{\Gamma}(\sigma)$ whose vertex set is $\{1,2,\dots, 2d\}$ with edges $\{2i-1, 2i\}$ and $\{\sigma(2i-1), \sigma(2i)\}$ for $1 \le i \le d$. Then $\Gamma(\sigma)$ has connected components of even sizes $2\rho_1 \ge 2 \rho_2 \ge \cdots$, which determine a partition $\rho = (\rho_1, \rho_2, \dots)$ of $d$ (in the number-theoretic sense). The length $\ell(\rho)$ of $\rho$ is defined to be the number of elements $\rho_i$ that it contains.  We let $\wg^{(n)}(\rho)$ equal the value of any $\wg^{(n)}(\m, \n)$ such that $\sigma = \m^{-1}\n$ has the corresponding partition $\rho$ of $d$.

\begin{thm}[{\cite[Corollary~2.7]{CS06}}]\label{t:weingarten2}
Fix $d \in \N$ and a partition $\rho$ of $d$.
As $n\rightarrow \infty$,
\begin{equation}\label{asymptotic}
\wg^{(n)}(\rho) = \left( \prod_{i\ge 1} c_{\rho_i -1}\right) n^{-2d + \ell(\rho)}( 1 + O(n^{-1})).
\end{equation}
Here $c_k=\frac{(2k)!}{(k+1)!k!}$ denotes the $k$-th Catalan number, and the implicit constant in the asymptotic notation depends on $d$ and $\rho$, but not $n$.
\end{thm}
\begin{ex} If $\m=\n=\{(1,2), (3,4)\}$ as in the previous example, then $\sigma = \m^{-1} \n$ is the identity permutation. The graph $\Gamma(\sigma)$ has the edges $\{1,2\}$ and $\{3,4\}$, so $\ell(\rho) = 2$, and we conclude the leading order term in the asymptotic is $n^{-2}$, as in the Gaussian case. The coefficient of this term is $1$ because $c_0 = 1$.
\end{ex}

We next recall Wick's theorem on the expectations of products of centered jointly normal random variables.
\begin{thm}[{\cite[Theorem~1.28]{Jan97}}]\label{thm:wicktheorem}
Let $g_1,...,g_n$ be independent normal random variables with mean zero and variance one. Then
\begin{equation}
    \E[g_1\cdot\cdots\cdot g_n]=\begin{cases}0,&\text{ if }n\text{ is odd,}\\
    \sum_{\m\in\mathcal M(n)}\prod_{i=1}^{n/2}\E\left[g_{\m(2i-1)}g_{\m(2i)}\right]&\text{ otherwise.}
    \end{cases}
\end{equation}
\end{thm}

\begin{lem}\label{momentexpectation}
Fix $k,n \in \N$, and suppose $\a \in \V_{n,k}$ is a random matrix distributed according to the Haar measure on $\V_{n,k}$.
Then for all $m\in\mathbb N$,
\begin{equation}\label{lem:Gaussian-moment}
\E\left[ \prod_{i=1}^{m} \sqrt{n}a_{1j_i}  \right]
= \E\left[ \prod_{i=1}^{m} g_{j_i}\right]\left( 1  + O\left( n^{-1} \right)\right).
\end{equation}
where the $g_i$ are independent standard Gaussian random variables, and the implicit constant depends only on $m$.
\end{lem}
\begin{proof}
By symmetry of Haar measure and Gaussian random variables, both sides of \eqref{lem:Gaussian-moment} are zero when $m$ is odd. It remains to consider the case $m=2d$ for some $d\in\mathbb N$.

Let $\Xi(\sigma)$ denote the coset type of $\sigma\in S_{2d}$ and set
\[\binom{\mathfrak m}{i_1\ldots i_{2d}}=\prod_{l=1}^d\delta_{i_{\mathfrak m(2l-1),\mathfrak m(2l)}}.\]
Let $\mathbf 1_{2d} \in S_{2d}$ be the identity permutation.
With these notations, we have
\begin{align}
\mathbb E&\left[\prod_{i=1}^{2d}\sqrt na_{1j_{i}}\right]\\
=&n^d\sum_{\mathfrak m,\mathfrak{n}\in\mathcal M(2d)}\wg^{(n)}(\mathfrak{m},\mathfrak{n})\binom{\mathfrak n}{j_1\ldots j_{2d}}\\
=&n^d\sum_{\rho \vdash d}\wg^{(n)}(\rho)\sum_{\Xi(\mathfrak m^{-1}\mathfrak{n})=\rho}\binom{\mathfrak n}{j_1\ldots j_{2d}}\\
=&n^{d}\wg^{(n)}\left((\mathbf{1}_{2d}),n\right)\mathbb E\left[\prod_{i=1}^{2d}g_{j_i}\right]+n^d\sum_{\substack{\rho\vdash d\\\rho\neq\mathbf 1_{2d}}}\wg^{(n)}(\rho)\sum_{\Xi(\mathfrak m^{-1}\mathfrak{n})=\rho}\binom{\mathfrak n}{j_1\ldots j_{2d}},\label{eqn:split-leading}
\end{align}
where the first equality follows from \Cref{t:weingarten1}, the second equality follows from the fact that $\wg^{(n)}(\mathfrak{m},\mathfrak{n})$ depends only on the coset type of $\m^{-1} \n$ (see \cite[Theorem~3.1]{CM18}), and the third equality follows from Wick's Theorem, \Cref{thm:wicktheorem}.

By \Cref{t:weingarten2}, the first term in \eqref{eqn:split-leading} is
\begin{align}\label{eqn:asym-leading}
\begin{split}
\mathbb E\left[\prod_{i=1}^{2d}g_{j_i}\right]\left(1+O\left(n^{-1}\right)\right),
\end{split}
\end{align}
and the absolute value of the second term of \eqref{eqn:split-leading} is upper bounded by
\begin{align}\label{eqn:asym-lots}
\begin{split}
n^d\mathbb E\left[\prod_{i=1}^{2d}g_{j_i}\right]\sum_{\substack{\rho\vdash d\\\rho\neq\mathbf 1_{2d}}}|\wg^{(n)}(\rho)|
\leq C_d\E\left[\prod_{i=1}^{2d}g_{j_i}\right]n^{-1}
\end{split}
\end{align}
for $n$ sufficiently large,
where $C_d=\max_{\rho\vdash d}\prod_{i\geq 1}c_{\rho_i-1}+1<\infty$.

Combining \eqref{eqn:split-leading}, \eqref{eqn:asym-leading} and \eqref{eqn:asym-lots}, we obtain the desired conclusion.
\end{proof}

\section{Proof of Auxiliary Lemmas in \texorpdfstring{\Cref{s:preliminarylowerProof}}{PDFstring}}\label{s:aux-proof}
Maintain the notations in \Cref{s:lowerBoundNotations}.
The proof of \Cref{lemma:second-third-concentration} is split into several parts, which we first briefly summarize:
\begin{enumerate}
\item Establish the concentration of the quantity of interest around the expectation taken over the projection direction $\bm a$. This step requires a truncation argument and good control of the truncated part , which is the subject of \Cref{s:c1} (see \Cref{lemma:deterministic-tail-bound}).
\item Prove the Gaussian approximation using the result of the first step. This second step is achieved through moment calculations, which are the subject of  \Cref{s:c2} (see \Cref{l:Gaussianapprox1}
and \Cref{l:multidimensionalMomentConvergenceTheorem}). \Cref{l:Gaussianapprox1} 
uses the Weingarten calculus (see \Cref{s:weingarten}) and \Cref{l:multidimensionalMomentConvergenceTheorem} is a multi-dimensional moment convergence theorem, which we include for completeness.
\item Finally, \Cref{s:c3} contains the proof of the asymptotic decorrelation result of \Cref{lemma:second-third-convergence3}.
\end{enumerate}
\subsection{Preliminary estimate}\label{s:c1}
\begin{rmk}\label{r:exchangeable}
In this appendix, we will use the fact that the rows $\{\a^{(n)}_i\}_{i=1}^n$ of the matrix $\a^{(n)} \in \V_{n, k_n}$ are exchangeable. This follows from the fact that $\a^{(n)}$ has the same distribution as the matrix obtained by taking a Haar-distributed element of $O_n$ and removing the last $n-k_n$ columns. 
\end{rmk}
\begin{lem}\label{lemma:deterministic-tail-bound} Fix $C,R\in(0,\infty)$, and let $h_R: \R \rightarrow \R$ be a function satisfying $|h_R(y)|\leq C R^{-1} y^2$ for all $y \in \R$.
Then the measure $\nun_j$ defined in \eqref{eqn:nu-n-j} 
satisfies
\begin{align}\label{eqn:expectation-bound-1}
\begin{split}
\left|\mathbb E\left[ \int_{\R^2} xh_R(y)\, d\nun_j(x,y)\right]\right|&\leq \frac{C_1\|\mathbf u^{(n)}\|_2^2}{R},\quad  \text{ for }j=1,\ldots,m+M,
\end{split}
\end{align}
where the constant $C_1>0$ depends only on $C$.
We also have the deterministic bounds
\begin{align}\label{eqn:deterministic-bound-1}
    \int_{\{|x|\geq R\}\times\R} |x|\, d\nun_j(x,y)&\leq\frac{1}{R},\text{ for }j=1,\ldots,m+M.
\end{align}
\end{lem}
\begin{proof}
For $j=1,\ldots,m$, applying first Jensen's inequality, next the Cauchy-Schwartz inequality and the fact that $(\bm a^{(n),\trans })_j$, the $j$-th column of $\a^{(n)}$, lies in $\mathbb S^{n-1}$, then \Cref{r:exchangeable},  the rotational invariance of the Haar measure $\sigma_n$, and  the fact that $\sqrt na_{11}$ is sub-Gaussian with $\|\sqrt n a_{11}\|_{\psi_2}$ bounded by a constant independent of $n$ \cite[Theorem~3.4.6]{Ver18}, we have
\begin{align*}
    \left(\mathbb E\left[\sum_{i=1}^na_{ij}^{(n)}h_R\left(\left\langle \un,\sqrt n\mathbf a_i^{(n)}\right\rangle\right)\right]\right)^2&\leq\mathbb E\left[\left(\sum_{i=1}^na_{ij}^{(n)}h_R\left(\left\langle \un,\sqrt n\mathbf a_i^{(n)}\right\rangle\right)\right)^2\right]\\
    &\leq\mathbb E\left[\sum_{i=1}^nh_R\left(\left\langle \un,\sqrt n\mathbf a_i^{(n)}\right\rangle\right)^2\right]\\
    &=n\mathbb E\left[h_R\left(\left\langle \un,\sqrt n\mathbf a_1\right\rangle\right)^2\right]\\
    &=n\mathbb E\left[h_R\left(\|\un\|_2\sqrt na_{11}\right)^2\right]\\
    &\leq C^2n\|\un\|_2^{4}\mathbb E\left[\frac{\left(\sqrt na_{11}\right)^{4}}{R^2}\right]\\
    &\leq \frac{C_1n\|\un\|_2^{4}}{R^2}.
\end{align*}
Inequality \eqref{eqn:expectation-bound-1} now follows by the definition of $\nu_{n}^{(j)}$. 
The inequality \eqref{eqn:deterministic-bound-1} can be proved by noticing that
\begin{align*}
\begin{split}
    \int_{\{|x|\geq R\}\times\R}|x|\,d\nun_j(x,y)\leq\frac{1}{R}\int_{\R^2} |x|^2d\nun_j(x,y)=\frac{1}{R}\sum_{i=1}^n\left(a_{ij}^{(n)}\right)^2=\frac{1}{R},
\end{split}
\end{align*}
where the last inequality uses the fact that $(\a^{(n),\trans})_j \in \S^{n-1}$. 
\end{proof}

\subsection{Proofs of Gaussian approximations}\label{s:c2}
\begin{lem}\label{l:Gaussianapprox1}
For any fixed $\alpha,\beta\in\mathbb N\cup\{{0}\}$, we have
\begin{align*}
\begin{split}
    \E\left[\int x^\alpha y^\beta d\nun_j\right]=\left(\int x^\alpha y^\beta d\nu_j\right)\left(1+O\left(\frac{1}{n}\right)\right),
\end{split}
\end{align*}
where the implicit constant depends only on $\alpha,\beta$.
\end{lem}
\begin{proof}
Using first \eqref{eqn:nu-n-j}, then \Cref{r:exchangeable}, and then \Cref{momentexpectation} with $m=\alpha+\beta$, we have
\begin{align}\label{moment-matching-calc}
\begin{split}
    \E_n\left[ \int x^{\alpha} y^{\beta} d \nun_j\right]
   &=\frac{1}{n} \E_n\left[\sum_{i=1}^{n}\left(\sqrt{n} a^{(n)}_{ij}\right)^{\alpha} \cdot\left\langle\un, \sqrt{n} \an_{i}\right\rangle^{\beta}\right]\\
   &=\mathbb{E}_n\left[\left(\sqrt{n} a^{(n)}_{1 j}\right)^{\alpha} \cdot\left(\sum_{\ell=1}^{k_n} u_{\ell} \sqrt{n} a^{(n)}_{1 \ell}\right)^{\beta}\right]\\
   &=\sum_{i_{1}, \cdots, i_{\beta}=1}^{k_n} \prod_{r=1}^\beta u_{i_r}^{(n)}\cdot \mathbb{E}_n\left[\left(\sqrt na^{(n)}_{1j}\right)^{\alpha} \cdot \prod_{r=1}^{\beta}\sqrt na^{(n)}_{1i_r}\right]\\
   &=\sum_{i_{1}, \cdots, i_{\beta}=1}^{k_n} \prod_{r=1}^\beta u_{i_r}^{(n)} \cdot \mathbb{E}\left[\left(g_{j}\right)^{\alpha} \cdot \prod_{r=1}^\beta g_{i_r}\right]\left(1+O\left(\frac{1}{n}\right)\right)\\
   &=\E\left[\left(g_j\right)^{\alpha} \cdot\left(\sum_{i=1}^{k_n} u^{(n)}_{i} g_i\right)^{\beta}\right]\left(1+O\left(\frac{1}{n}\right)\right),
\end{split}
\end{align}
where the implicit constant depends only on $\alpha,\beta$.
By the 
definition of $\nu_j$ in \eqref{eqn:lawOfJointGaussians}, we have
\begin{align*}
\begin{split}
    \E\left[\left(g_j\right)^{\alpha} \cdot\left(\sum_{i=1}^{k_n} u^{(n)}_{i} g_i\right)^{\beta}\right]=\int x^{\alpha} y^{\beta} d \nu_j.
\end{split}
\end{align*}
Together with \eqref{moment-matching-calc}, this proves the lemma.
\end{proof}

\begin{lem}\label{l:multidimensionalMomentConvergenceTheorem}
A sequence of probability measures $\{\mu_n\}$ on $\mathbb R^2$ converges weakly to a probability measure $\mu$ if the following conditions are satisfied:
\begin{enumerate}
\item All moments of $\mu$ are finite.
\item All moments of $\mu_n$ are finite and
\begin{align*}
\begin{split}
    \int x^\alpha y^\beta d\mu_n(x,y)\rightarrow \gamma_{\alpha,\beta},\quad\forall\,\alpha,\beta\in\N\cup\{0\},
\end{split}
\end{align*}
where
\begin{align*}
\begin{split}
    \gamma_{\alpha,\beta}\coloneqq \int x^{\alpha}y^\beta d\mu_n(x,y),\quad\forall\,\alpha,\beta\in\N\cup\{0\}.
\end{split}
\end{align*}
\item $\mu$ is uniquely determined by $\{\gamma_{\alpha,\beta}\}_{\alpha,\beta\in\N\cup\{0\}}$.
\end{enumerate}
\end{lem}

\begin{proof}
By property (1) and (2), for each polynomial $P$, we have
\begin{align}\label{apple1}
\begin{split}
    C_p\coloneqq \sup_{n\geq 1}\int P d\mu_n<\infty.
\end{split}
\end{align}
Let $B_R\coloneqq \{\mathbf x\in\mathbb R^2:\|\mathbf x\|_2\leq R\}$. Then Markov's inequality implies that
\begin{align}\label{apple2}
\begin{split}
    \mu_{n}\left(\left(B_R\right)^c\right)\leq \frac{C_{x^2+y^2}}{R^2}.
\end{split}
\end{align}
Therefore, $\mu_n$ is tight. By Prokhorov's Theorem, it suffices to show that if $\mu_{n_k}\rightarrow\nu$ weakly, then $\nu=\mu$.

Pick any polynomial $P$. For any $R \in (0, \infty)$ define $\phi_R$ to be a nonnegative continuous function such that $\mathbf 1_{B_R}\leq \phi_R\leq \mathbf 1_{B_{R+1}}$ and $\phi_R$ is monotonically increasing in $R$. We write
\begin{align}\label{moment-convergence-decomposition}
\begin{split}
    \int P d\mu_{n_k}=\int \phi_R Pd\mu_{n_k}+\int (1-\phi_R)Pd\mu_{n_k}.
\end{split}
\end{align}
For the first term  the right-hand side of \eqref{moment-convergence-decomposition}, since $\mu_{n_k}$ weakly converges to $\nu$, and $\phi_R P$ is bounded and continuous, it follows that
\begin{align}\label{moment-proof-1}
\begin{split}
    \lim_{k\rightarrow\infty}\int \phi_R Pd\mu_{n_k}=\int\phi_R Pd\nu.
\end{split}
\end{align}
For the second term, by the Cauchy--Schwarz inequality, \eqref{apple1}, and \eqref{apple2}, we have
\begin{align}\label{moment-proof-2}
\begin{split}
    \left|\int(1-\phi_R)Pd\mu_{n_k}\right|^2\leq \mu_{n_k}\left((B_R)^c\right)\int P^2d\mu_{n_k}\leq \frac{C_{x^2+y^2}C_{P^2}}{R^2}.
\end{split}
\end{align}
Moreover, properties (1) and (2) imply that
\begin{align}\label{moment-proof-3}
\begin{split}
    \lim_{k\rightarrow\infty}\int P\, d\mu_{n_k}=\int P\, d\mu < \infty.
\end{split}
\end{align}
Combining \eqref{moment-convergence-decomposition}-\eqref{moment-proof-3}, we have
\begin{align}\label{moment-proof-4}
\begin{split}
    \lim_{R\rightarrow\infty}\int\phi_RP\, d\nu=\int P\, d\mu .
\end{split}
\end{align}
Since $P$ is arbitrary, replacing $P$ by $P^2$ in the above display and applying the monotone convergence theorem implies that 
\begin{align*}
\begin{split}
    \int P^2d\nu=\lim_{R\rightarrow\infty} \int \phi_R P^2d\nu=\int P^2\, d\mu < \infty.
\end{split}
\end{align*}
Therefore, $P\in L^2(\nu)\subset L^1(\nu)$. An application of the dominated convergence theorem shows that the left-hand side of  \eqref{moment-proof-4} is equal to $\int P\, d \nu$ and hence that
\begin{align*}
\begin{split}
    \int P\, d\nu=\int P\, d\mu.
\end{split}
\end{align*}
Again, since $P$ is arbitrary, setting $P=x^\alpha y^\beta\,(\alpha,\beta\in\N\cup\{0\})$, and invoking property (3), we have $\mu=\nu$. This completes the proof.
\end{proof}

\begin{proof}[Proof of \Cref{lemma:second-third-concentration}] In this proof, we use $C$ to denote a constant that may depend on $C_0$, whose value might change from line to line.

A straightforward differentiation of the function $f_j^{(n)}$ defined in \eqref{def:fjn} shows that for 
 $i=1,\ldots,n$ and $l=1,2,\ldots,k_n$,
\begin{align*}
\begin{split}
    \partial_{il}f_j^{(n)}\left(\an\right) =\begin{cases}
    \frac{1}{\sqrt n}h\left(\left\langle \un,\sqrt n\an_i\right\rangle\right)+a^{(n)}_{ij}h'\left(\left\langle \un,\sqrt n\an_i\right\rangle\right) u^{(n)}_j,&l=j,\\
    a_{ij}^{(n)}h'\left(\left\langle \un,\sqrt n\an_i\right\rangle\right) u_l^{(n)} ,&l\neq j.
    \end{cases}
\end{split}
\end{align*}
Hence, by \eqref{eqn:h-growth} and the definition of $\un$ in \eqref{eqn:def-of-u}, it follows that for $j=1,\dots, n$ that 
\begin{align}\label{eqn:gradient-norm-fj}
\begin{split}
    \left\|\nabla f_j^{(n)}\left(\an\right)\right\|_2^2=& \sum_{i=1}^n\sum_{l=1}^{k_n}\left(\partial_{il}f_j^{(n)}\left(\an\right)\right)^2\\
    \leq& \frac{2}{n}\sum_{i=1}^n\left(h\left(\left\langle \un,\sqrt n\an_i\right\rangle\right)\right)^2\\
    &+2\sum_{i=1}^n\left(a_{ij}^{(n)}h'\left(\left\langle \un,\sqrt n\an_i\right\rangle\right)\right)^2\sum_{l=1}^{k_n}\left(u_j^{(n)}\right)^2\\
    \leq &\frac{C}{n}\sum_{i=1}^n\left(1+n\left\langle \un,\an_i\right\rangle^2\right)+C\sum_{i=1}^n\left(a_{ij}^{(n)}\right)^2\sum_{l=1}^{k_n}\left(u_j^{(n)}\right)^2\\
    \leq &C\left(1+\left\|\un\right\|_2^2\right)\\
    =&C\left(1+\left\|\hatu\right\|_2^2+b^2\right)
\end{split}
\end{align}
This implies that $f_j^{(n)}$ is  Lipshitz continuous  with Lipshitz constant $\sqrt{C\left(1+\|\hatu\|_2^2+b^2\right)}$.
By \Cref{t:gromov}, there exists $C'>0$ such that for any $\epsilon>0$
\begin{align}\label{apply-gromov-2}
\begin{split}
    \sigma_n\left(\left|f_j^{(n)}\left(\an\right)-\E_n\left[f_j^{(n)}\left(\an\right)\right]\right|>\epsilon\right)\leq \exp\left(-\frac{C'\epsilon^2n}{1+\|\hatu\|_2^2+b^2}\right).
\end{split}
\end{align}
Together with the Borel--Cantelli lemma, this implies that,
for $\sigma$-a.e. $\bm a=\left(\bm a^{(1)},\bm a^{(2)},\ldots\right)\in\V$, 
\begin{align}\label{concentration-fj}
\begin{split}
    \lim_{n\rightarrow\infty}\left|f_j^{(n)}\left(\an\right)-\E_n\left[f_j^{(n)}\left(\an\right)\right]\right|= 0.
\end{split}
\end{align}
Recall the definition of $\nu_j^{(n)}$ in \eqref{eqn:nu-n-j}, we may rewrite
\begin{align}\label{rewrite-fj}
\begin{split}
    f_j^{(n)}\left(\an\right)=\int_{\R^2}x\partial_1\bar\Lambda\left(y,c\right)d\nu_{j}^{(n)}(x,y).
\end{split}
\end{align}
By \Cref{l:Gaussianapprox1}, \Cref{l:multidimensionalMomentConvergenceTheorem}, and \cite[Theorem~1.7]{ollivier2014optimal}, it follows that $\E\left[\int fd\nu_j^{(n)}\right]\rightarrow \int fd\nu_j$ for all $f:\R^2\rightarrow\R$ with polynomial growth. By \Cref{assume4}, $f(x,y)\defeq x\partial_1\bar\Lambda(y,c)$ can be bounded by a quadratic polynomial. Hence, by \eqref{rewrite-fj}
\begin{align}\label{concentration-fj2}
\begin{split}
    \lim_{n\rightarrow\infty}\left|\E_{n}\left[f_j^{(n)}\left(\an\right)\right]-\int_{\R^2}x\partial_1\bar\Lambda(y,c)d\nu_j\right|=0.
\end{split}
\end{align}
Combining \eqref{concentration-fj}, \eqref{concentration-fj2} and recalling the definition of $\nu_j$ in \eqref{eqn:lawOfJointGaussians}, the proof of the lemma is complete.
\end{proof}
\subsection{Proof of asymptotic decorrelation}\label{s:c3}
The proof of \Cref{lemma:second-third-convergence3} relies on the sub-Gaussianity of a random vector chosen uniformly from a ``large enough'' portion of the scaled unit sphere, which is the content of the next lemma. We recall that the sub-Gaussian norm $\|\cdot \|_{\psi_2}$ was defined in \eqref{d:subGaussian-vector}.
\begin{lem}\label{lem:sub-Gaussian-sliced}
There exists a universal constant $K \in (0, \infty)$ such that the following holds. 
Fix $m ,n  \in \N$ such that $m < n/2$, and a collection $\bm v_1,\ldots,\bm v_{m}\in \mathbb R^n$ of random vectors that are almost surely orthonormal. 
Let $\bm v$ be a random vector whose distribution conditional on  $\{\bm v_i\}_{i=1}^{m}$ is 
uniform on the subset of $\mathbb S^{n-1}$ orthogonal to the subspace generated by $\{\bm v_i\}_{i=1}^{m}$. 
Then 
\begin{align}
\begin{split}
    \|\sqrt n\bm v\|_{\psi_2}\leq K.
\end{split}
\end{align}
\end{lem}
\begin{proof}
Let $O\in O_n$ be a random matrix such such that $O\bm v_i=\bm e_i$ for all $i=1,\ldots,m  $. By the assumption that $\bm v$ is perpendicular to each $\bm v_i$, it follows that $(O\bm v)_i=0$ for all $i=1,\ldots,m$ and $\hat{\bm v}\coloneqq \left((O\bm v)_{i}\right)_{i=m+1}^n$ is distributed uniformly on $\mathbb S^{n-m}$.
Hence, by the definition of the sub-Gaussian norm in \eqref{d:subGaussian-vector}, which shows it is invariant under orthogonal transformations, 
\begin{align}\label{use-of-hatv}
\begin{split}
\left\| \bm v\right\|_{\psi_2} = 
    \left\| O\bm v\right\|_{\psi_2}=\sup_{\bm x\in\mathbb S^{n-1}}\left\|\left\langle O\bm v,\bm x\right\rangle\right\|_{\psi_2}=\sup_{\substack{\bm x\in\mathbb S^{n-1}\\x_i=0,\,1\leq i\leq m}}\left\|\left\langle O\bm v,\bm x\right\rangle\right\|_{\psi_2}=\left\|\hat{\bm v}\right\|_{\psi_2}.
\end{split}
\end{align}
Since $\hat{\bm v}$ is uniformly distributed on $\mathbb S^{n-m}$, by \cite[Theorem~3.4.6]{Ver18} there exists a universal constant $K\in(0,\infty)$ such that
\begin{align}\label{sub-gaussian-hatv}
\begin{split}
    \|\sqrt{(n-m)}\hat{\bm v}\|_{\psi_2}\leq \frac{K}{\sqrt{2}}.
\end{split}
\end{align}
Combining \eqref{use-of-hatv}, \eqref{sub-gaussian-hatv}, and using the hypothesis that $m < n/2$, the proof is complete.
\end{proof}

\begin{proof}[Proof of \Cref{lemma:second-third-convergence3}]
   Let 
   \[S_n(\an) \coloneqq \sum_{\ell=m+M+1}^{k_n}\left(f_\ell^{(n)}\left(\an\right)\right)^2.\]
Fix $\lambda>0$, and let $\bm Z\coloneqq (Z_{m+M+1},\ldots,Z_{k_n})^\trans$ be a random vector of i.i.d. standard Gaussian random variables independent of $\bm a^{(n)}$. Since $\E\left[\exp\left(xZ_\ell \right)\right]=\exp(x^2/2)$ for any $x\in\R$ and $\ell=m+M+1,\ldots,k_n$, we have, recalling the definition of $f_\ell (\an)$ in \eqref{def:fjn},
    \begin{align}\label{eqn:momentSn}
    \begin{split}
        \E\left[\exp\left(\frac{\lambda^2}{2} S_n(\an)\right)\right]=&\,\E\left[\prod_{\ell=m+M+1}^{k_n}\E\left[\exp\left(\frac{1}{2}\left(\lambda f_\ell^{(n)}(\bm a^{(n)})\right)^2 \right)\bigg|\, \bm a^{(n)}\right]\right]\\
        =&\,\mathbb E\left[\prod_{\ell=m+M+1}^{k_n}\E\left[\exp\left(\lambda  f_\ell^{(n)}(\an)Z_\ell\right)\bigg|\,\an\right]\right]\\
        =&\,\E\left[\E\left[\exp\left(\lambda\sum_{\ell=m+M+1}^{k_n} f_{\ell}^{(n)}(\an)Z_\ell\right)\bigg|\, \bm a^{(n)}\right]\right]\\
        =&\,\E\left[\exp\left(\lambda\sum_{\ell=m+M+1}^{k_n} f_{\ell}^{(n)}(\an)Z_\ell\right)\right].
    \end{split}
    \end{align}
    Next, note that since $\{Z_\ell\}_{\ell=m+M+1}^{k_n}$ are symmetric and independent of $\an$, we have 
    \be\label{eqn:cond-0}
        \E\left[\lambda\sum_{\ell=m+M+1}^{k_n}Z_\ell f_{\ell}(\an) \right] = 
        \lambda \sum_{\ell=m+M+1}^{k_n} \E\left[ Z_\ell\right] \E\left[ f_{\ell}(\an) \right]= 0.
    \ee

Now, for $i=1,\ldots,n$, set 
    \begin{equation}\label{d:ydef}
    Y_i\equiv Y_i^{(n)}(\an)\coloneqq 
    \frac{1}{n} h\left(\sum_{j=1}^{m+M}\sqrt nu_ja_{ij}^{(n)}\right), \qquad 
    \bm Y\equiv\bm Y^{(n)}(\an)\coloneqq(Y_1,\ldots,Y_{n})^\trans,
    \end{equation}
    \[\hat A\equiv\hat A^{(n)}\coloneqq(\sqrt na^{(n)}_{i\ell})_{\substack{1\leq i\leq n,m+M+1\leq \ell\leq k_n}}.\]
    Since $\bm Z$ is a vector of standard Gaussians (and in particular has a rotationally symmetric distribution), there exists a $\R^{k_n - m - M} \times \R^{k_n - m - M}$-valued, Haar-distributed random orthogonal matrix $O$ such that $O^\trans \bm Z=\|\bm Z\|_2 \bm e_1$, where $\bm e_1 \in \R^{k_n -m -M}$. 
    By 
    the identity $O O^\trans = I$, it follows that
      \begin{align}\label{changeZtoZnorm}
    \begin{split}
        \sum_{\ell=m+M+1}^{k_n}Z_\ell f_{\ell}^{(n)}(\an)=\bm Y^\trans \hat A\bm Z
        \distequal \|\bm Z\|_2\bm Y^\trans \hat AO \bm e_1,
    \end{split}
    \end{align}
    Note that $\hat AO\bm e_1$ has norm $1$ and lies in the subset of $\mathbb{S}^{n-1}$ orthogonal to the first $m+M$ columns of $\an$ (that is, $(a^{(n)}_{i\ell})_{\substack{1\leq i\leq n,1\leq \ell\leq m+M}}$). Moreover, conditional on the first $m+M$ columns of $\an$, it is uniformly distributed on this subset (since $O$ is Haar-distributed). 
    By \Cref{lem:sub-Gaussian-sliced}, with $m$ in that lemma statement equal to the quantity $m+M$ in this proof, there exists a constant $K \in (0, \infty)$ such that for all sufficiently large $n$ (depending only on the choice of $m$ and $(\hat \u, b, c)$), we have 
    $\left\|\hat AO\bm e_1\right\|_{\psi_2}\leq K$, where $\| \cdot \|_{\psi_2}$ denotes the sub-Gaussian norm defined in \Cref{d:subGaussian}.
    Then it is immediate from \eqref{d:subGaussian-vector} that the sub-Gaussian norm  of $\bm Y^\trans \hat AO\bm e_1$ satisfies the estimate
    \begin{align}\label{eqn:y-norm}
    \begin{split}
        \|\bm Y^\trans \hat AO\bm e_1\|_{\psi_2}\leq K\|\bm Y\|_2.
    \end{split}
    \end{align}
    Combining \eqref{eqn:cond-0}, \eqref{changeZtoZnorm}, 
    \eqref{eqn:y-norm}, \eqref{d:subGaussian}, and the exponential moment estimate \cite[Proposition~2.5.2(v)]{Ver18}, we have
    \begin{align}\label{eqn:boundGivenA}
    \begin{split}
        \E\left[\exp\left(\lambda\sum_{\ell=m+M+1}^{k_n}Z_\ell f_{\ell}(\an)\right)\right]\leq \exp\left(K^2\lambda^2\|\bm  Y\|_2^2\|\bm Z\|_2^2\right).
    \end{split}
    \end{align}
    By \eqref{d:ydef} and the assumption \eqref{eqn:h-growth} on $h$, there exists a constant $C\in (0, \infty)$ such that
    \begin{align}\label{eqn:uniformBoundA}
    \begin{split}
        \|\bm Y\|_2^2\leq \frac{C}{n^2}\sum_{i=1}^n\left(1+\left\langle \un,\sqrt n\bm a_i^{(n)}\right\rangle^2\right)\leq\frac{C}{n}(1+\|\un\|_2^2).
    \end{split}
    \end{align}

    Set $\lambda\coloneqq\sqrt{\frac{n}{8CK^2(1+\|\un\|_2^2)}}$. Combining \eqref{eqn:momentSn}, \eqref{eqn:boundGivenA} and \eqref{eqn:uniformBoundA}, with Markov's inequality, and recalling the moment generating function for a Gaussian, we have
    \begin{align*}
    \begin{split}
        \P\left(S_n>\epsilon\right)\leq e^{-\frac{\lambda^2}{2}\epsilon}\E\left[e^{\frac{\lambda^2}{2}S_n}\right]\leq e^{-\frac{\lambda^2}{2} \epsilon }\prod_{\ell=m+M+1}^{k_n}\E\left[e^{\frac{ Z_\ell^2}{8}}\right]= e^{-\frac{\epsilon n}{16CK^2(1+\|\un\|_2^2)}}\left(\frac{4}{3}\right)^{k_n/2}.
    \end{split}
    \end{align*}
    Since $k_n=o(n)$ by assumption, by the Borel--Cantelli lemma, we have $S_n\rightarrow 0$ $\sigma$-a.s.
\end{proof}

\bibliographystyle{amsplain}
\bibliography{project_quenched.bib}

\end{document}